\documentclass[12pt,a4paper,reqno]{amsart}

\usepackage[latin1]{inputenc}
\usepackage[english]{babel}
\usepackage[T1]{fontenc}

\usepackage{amsmath}
\usepackage{amsfonts}
\usepackage{amssymb}
\usepackage{amsthm}
\usepackage{graphicx}

\usepackage{mathtools} 

\usepackage{enumitem}

\usepackage[final]{showkeys} 

\usepackage{cases}

\usepackage[left=2cm,right=2cm,%
top=2cm,bottom=2cm]{geometry}

\usepackage{hyperref}
\hypersetup{%
colorlinks=true, %
linkcolor=blue, %
citecolor=blue, %
filecolor=blue, %
urlcolor=blue, %
pdfauthor={R. B. Assunção, 
           O. H. Miyagaki, and 
           W. W. Santos}, %
pdftitle={Titulo}, %
pdfsubject={Elliptic partial differential equations}, %
pdfkeywords={Quasilinear elliptic equations, 
p-Laplacian operator, 
variational methods,
multiple critical nonlinearities}, %
pdfproducer={Latex}, %
pdfcreator={ps2pdf}}

\newtheorem{teorema}{Theorem}[section]
\newtheorem{lema}[teorema]{Lemma}
\newtheorem{prop}[teorema]{Proposition}

\theoremstyle{remark}
\newtheorem{afirmativa}{Claim}

\allowdisplaybreaks

\title[Quasilinear elliptic problems with cylindrical singularities]{Quasilinear elliptic problems with cylindrical singularities and multiple critical nonlinearities: existence, regularity, nonexistence}

\author[Assun\c{c}\~{a}o]{R.\@ B.\@ Assun\c{c}\~{a}o}
\address{R.\@ B.\@ Assun\c{c}\~{a}o \hfill\break\indent
Departamento de Matem\'{a}tica\,---\,Universidade Federal de Minas Gerais, UFMG\hfill\break\indent
Av.~Ant\^{o}nio Carlos, 6627\,---\,CEP 30161-970\,---\,Belo Horizonte, MG, Brasil}
\email[Corresponding author]{ronaldo@mat.ufmg.br}

\author[Santos]{W.\@ W.\@ dos Santos}
\thanks{W.\@ W.\@ dos Santos was partially supported by CAPES/Reuni.}
\address{W.\@ W.\@ dos Santos \hfill\break\indent
Departamento de Ciências Exatas e Biológicas \hfill\break\indent
Universidade Federal de S\~{a}o Jo\~{a}o del-Rei
\,---\, Campus Sete Lagoas, CSL \hfill\break\indent
Rodovia MG 424\,-\, Km 47\,---\,CEP 35701-970\,---\,Sete Lagoas, MG, Brasil}
\email{weler@ufsj.edu.br}

\author[Miyagaki]{O.\@ H.\@ Miyagaki}
\thanks{O.\@ H.\@ Miyagaki was partially supported by CNPq/Brasil and INCTMAT/Brasil.}
\address{O.\@ H.\@ Miyagaki \hfill\break\indent
Departamento de Matem\'{a}tica\,---\,Universidade Federal de Juiz de Fora, UFJF \hfill\break\indent
Cidade Universit\'{a}ria\,---\,CEP 36036-330\,---\,Juiz de Fora, MG, Brasil} 
\email{ohmiyagaki@gmail.com}

\date{Belo Horizonte, \today}



\begin{document}

\begin{abstract}
This work deals with existence of solutions for the 
class of quasilinear elliptic problems with 
cylindrical singularities 
and multiple critical nonlinearities
that can be written in the form
\begin{align*}
-\operatorname{div}\left[\frac{|\nabla u|^{p-2}}{|y|^{ap}}\nabla u\right]
-\mu\,\frac{u^{p-1}}{|y|^{p(a+1)}}
= \frac{u^{p^*(a,b)-1}}{|y|^{bp^*(a,b)}}
+ \frac{u^{p^*(a,c)-1}}{|y|^{cp^*(a,c)}}, 
\qquad (x,y) \in \mathbb{R}^{N-k}\times\mathbb{R}^k.
\end{align*}
We consider $N \geqslant 3$,
$1 \leqslant k \leqslant N$,
$1 < p < N$,
$\mu <\bar{\mu} \equiv \left\{[k-p(a+1)]/p\right\}^p$,
$ a < (k-p)/p$,
$a \leqslant b < c < a+1$, 
$p^*(a,b)=Np/[N-p(a+1-b)]$, 
and 
$p^*(a,c) \equiv Np/[N-p(a+1-c)]$;
in particular, if $\mu = 0$ we can include the cases 
$(k-p)/p \leqslant a < k(N-p)/Np$ and
$a < b< c<k(N-p(a+1))/p(N-k) < a+1$.
The existence of a positive, weak solution 
\( u \in \mathcal{D}_a^{1,p}(\mathbb{R}^N\backslash\{|y|=0\}) \) is proved with the help of the 
mountain pass theorem.
We also prove a regularity result, that is, using Moser's iteration scheme we show that
 \( u \in L\sb{\mathrm{loc}}\sp{\infty}(\Omega)\) 
 for domains 
 \( \Omega \subset \mathbb{R}\sp{N-k}\times\mathbb{R}\sp{k}\backslash \{ |y|=0 \} \) not necessarily bounded. 
Finally we show that if 
\( u \in \mathcal{D}_a^{1,p}(\mathbb{R}^N\backslash\{|y|=0\}) \)
is a weak solution to the related problem
\begin{align*}
-\operatorname{div}\left[\frac{|\nabla u|^{p-2}}{|y|^{ap}}\nabla u\right]
-\mu\,\frac{|u|^{p-2}u}{|y|^{p(a+1)}}
= \frac{|u|^{q-2}u}{|y|^{bp^*(a,b)}}
+ \frac{|u|^{p^*(a,c)-2}u}{|y|^{cp^*(a,c)}}, 
\qquad (x,y) \in \mathbb{R}^{N-k}\times\mathbb{R}^k,
\end{align*}
then \( u \equiv 0 \) when either
\( 1 < q < p^*(a,b) \), or 
\( q > p^*(a,b) \) and 
\( u \in L\sb{bp^*(a,b)/q, \operatorname{loc}}\sp{q}
(\mathbb{R}^N\backslash\{|y|=0\}) 
\cap L\sb{\mathrm{loc}}\sp{\infty}(\mathbb{R}\sp{N-k}\times \mathbb{R}\sp{k} \backslash \{ |y| = 0\})\).
This nonexistence of nontrivial solution is proved by using a Pohozaev-type identity.
\end{abstract}

\maketitle

\section{Introduction and main results}
\label{sec:intmain}

The main goal of this work is to prove existence
results for a class of quasilinear elliptic problems 
with cylindrical singularities and multiple critical 
nonlinearities that can be written in the form
\begin{align}
\label{problema}
-\operatorname{div}\left[\frac{|\nabla u|^{p-2}}{|y|^{ap}}\nabla u\right]
-\mu\, \frac{u^{p-1}}{|y|^{p(a+1)}}=\frac{u^{p^*(a,b)-1}}{|y|^{bp^*(a,b)}}+\frac{u^{p^*(a,c)-1}}{|y|^{cp^*(a,c)}}, \qquad (x,y) \in \mathbb{R}^{N-k}\times\mathbb{R}^k.
\end{align}
We consider $N \geqslant 3$,
$1 \leqslant k \leqslant N$,
$1 < p < N$,
$0 \leqslant \mu <\bar{\mu} \equiv \left\{[k-p(a+1)]/p\right\}^p$,
$0 \leqslant a < (k-p)/p$,
$a \leqslant b < c < a+1$, 
$p^*(a,b)=Np/[N-p(a+1-b)]$, 
and 
$p^*(a,c) \equiv Np/[N-p(a+1-c)]$;
in particular, if $\mu = 0$ we can include the cases 
$(k-p)/p \leqslant a < k(N-p)/Np$ and
$a < b< c<k(N-p(a+1))/p(N-k) < a+1$.

This class of problems arises in the study of 
standing waves in the anisotropic Schr\"{o}dinger equation. It also appears in models of physical 
phenomena related to the 
equilibrium of the temperature in an
anisotropic media that can be a 
\lq perfect insulator\rq\ at 
some points, represented by 
the degenerate case
\( \inf\sb{|y| \to \infty } 1/|y|\sp{ap} = 0\),
and can be a \lq perfect conductor\rq\ at other points,
represented by the singular case
\( \sup\sb{|y| \to 0 } 1/|y|\sp{ap} = \infty\).
Problem~\eqref{problema} also has some interest in astrophysics, where the dynamics of some galaxies is modeled with the use of cylindrical weights due to their axial symmetry. 
For more details, see 
Dautray and Lions~\cite{MR1036731},
Wang and Willem~\cite{MR1744147}, 
Catrina and Wang~\cite{MR1794994},
Badiale and Tarantello~\cite{MR1918928},
Dr\'{a}bek~\cite{MR2331056},
Ghergu and R\u{a}dulescu~\cite{MR2865669}, 
and references therein.
The pure mathematical interest in 
this class of problems is due to the fact that
problem~\eqref{problema} can be regarded as a model 
for a general class of quasilinear elliptic problems 
with a cylindrical weight in the $p$-Laplacian 
differential operator and also involving 
multiple nonlinearities with cylindrical weights and 
critical Maz'ya's exponents.

The choice for the intervals for the several parameters already specified is motivated by the following Maz'ya's inequality,
which plays a crucial role in our work since it allows the variational formulation of problem~\eqref{problema}.
Let
$N \geqslant 3$,
$1 \leqslant k \leqslant N$,
$z=(x,y) \in \mathbb{R}\sp{N-k}\times\mathbb{R}\sp{k}$,
$1 < p < N$,
and either 
$ a < (k-p)/p$ and
$a \leqslant b \leqslant a+1$, 
or 
$(k-p)/p \leqslant a < k(N-p)/Np$ and
$a \leqslant b < k(N-p(a+1))/p(N-k) < a+1$.
Then there exists a positive constant
$C>0$ such that
\begin{align}
\label{eq:maz}
\left(\int\sb{\mathbb{R}\sp{N}}  \frac{|u(z)|\sp{p^*(a,b)}}{|y|\sp{bp^*(a,b)}} \,dz \right)\sp{p/p^*(a,b)}
\leqslant  C
\int\sb{\mathbb{R}\sp{N}}  \frac{|\nabla u(z)|^p}{|y|\sp{ap}} \,dz
\end{align}
for every function 
$u \in C\sp{\infty}(\mathbb{R}\sp{N}\backslash \{ |y|=0 \})$, 
where
$p^*(a,b)=Np/[N-p(a+1-b)]$
is the critical Maz'ya's exponent.

The proof of inequality~\eqref{eq:maz} can be found in the book by 
Maz'ya~\cite[Section 2.1.6]{maz'ya}; 
the particular case $k=N$ of inequality~\eqref{eq:maz} was proved by 
Caffarelli, Kohn and Nirenberg~\cite{MR768824}; 
see also Lin~\cite{MR864416} for an inequality involving higher order derivatives in the case $k=N$.

In what follows we present a very brief 
historical sketch for these types of problems, mainly concerning existence results. 
To avoid unnecessary repetitions, 
we write the class of problems in the form
\begin{align}
\label{classeproblema}
-\operatorname{div}\left[\frac{|\nabla u|^{p-2}}{|y|^{ap}}\nabla u\right]
-\mu \,\frac{u^{p-1}}{|y|^{p(a+1)}}=\frac{u^{p^*(a,b)-1}}{|y|^{bp^*(a,b)}}
+\lambda\,\frac{u^{p^*(a,c)-1}}{|y|^{cp^*(a,c)}}, \qquad (x,y) \in \mathbb{R}^{N-k}\times\mathbb{R}^k.
\end{align}
We also define the infimum
\begin{align}
\label{imersaolp*(a,b)}
\displaystyle \frac{1}{K(N,p,\mu,a,b)} \equiv 
\inf\sb{\substack{u \in 
C\sp{\infty}(\mathbb{R}\sp{N}\backslash \{ |y|=0 \})
\\ u \not\equiv 0}} 
\frac{\displaystyle\int\sb{\mathbb{R}\sp{N}} \frac{|\nabla u(z)|\sp{p}}{|y|\sp{ap}}\,dz
-\mu\displaystyle\int\sb{\mathbb{R}\sp{N}} \frac{|u(z)|\sp{p}}{|y|\sp{p(a+1)}}\,dz}
{\left(\displaystyle\int\sb{\mathbb{R}\sp{N}} \frac{|u(z)|\sp{p^*(a,b)}}{|y|\sp{bp^*(a,b)}}\,dz\right)\sp{p/p^*(a,b)}},
\end{align}
which is positive for $\mu < \bar{\mu} \equiv 1/K(N,p,0,a,a+1)$
and has an independent interest.
For the determination of the optimal constant $\bar{\mu} = \left[(k-p(a+1))/p \right]^p$, see the paper by
Secchi, Smets and Willem~\cite{MR1990020}.

First we consider the case \( \lambda = 0 \).
For \( k = N \), 
which represents spherical weights, 
\( 1 < p < N \), 
\( \mu = 0 \), 
\( a = 0 \), \( b = 0 \), and
$p^*(a,a) = p^* = Np/(N-p)$, 
problem~\eqref{classeproblema} was treated in the well known papers by 
Aubin~\cite{MR0448404} 
and Talenti~\cite{MR0463908};
they computed the value of the best constant \(K(N,p,0,0,0)\) for the Sobolev inequality
 and presented the class of the functions that assume this infimum. For more information on best constants, see also the papers by 
Chou and Chu~\cite{MR1223899},
Horiuchi~\cite{MR1731336}
and references therein. 
Wang and Willem~\cite{MR1744147}
showed the existence of solution to 
problem~\eqref{classeproblema} 
in the case \( p = 2 \), 
which is the Laplacian operator,
with critical spherical weights represented by  
\( k = N \), 
\( 0 \leqslant a < (N-2)/2 \),
\( a \leqslant b < a+1 \),
with a homogeneous linear term, represented by \( 0 \leqslant \mu < \bar{\mu} \),
and
\( 2^*(a,b) < 2\sp{*}\equiv 2N/(N-2) \). 
Catrina and Wang~\cite{MR1794994} also considered these cases but with \( a < (N-2)/2 \) and obtained several existence, nonexistence, as well as symmetry breaking of solutions to problem~\eqref{classeproblema}. 
Problem~\eqref{classeproblema} in the case
of the $p$-Laplacian operator
represented by \( 1 < p < N \),
with critical spherical weights represented by
\( k = N \),
\( a \leqslant b < a+1 \),
and a homogeneous nonlinearity,
represented by
\( \mu < \bar{\mu} \), was also studied by 
Assun\c{c}\~{a}o, Carri\~{a}o and Miyagaki~\cite{MR2277772},
who extended the results by 
Wang and Willem~\cite{MR1744147}.
For problems with cylindrical weights, 
represented by 
$1 \leqslant k \leqslant N$,
we cite the paper by
Musina~\cite{MR2416099}, who studied the case 
$N\geqslant 3$, 
$1 \leqslant k \leqslant N$, 
$p=2$, 
$a=0$, 
$0 < b \leqslant 1$, 
$\mu < \bar{\mu}$
and proved that problem~\eqref{classeproblema} has a ground state solution; 
in particular, if $k=1$ then the support of this solution is the half-space 
$y\geqslant 0$. 
Additionally, if $b=0$ and $0 < \mu < \bar{\mu}$, problem~\eqref{classeproblema} has a ground state solution when 
either 
$2 < k \leqslant N$,
or $k = 1$ and $N \geqslant 4$; 
in this last case the ground state also has support in a half-space. 
Gazzini and Musina~\cite{MR2499889} studied problem~\eqref{classeproblema} in the case 
$ N \geqslant 3$, 
\( 1 \leqslant k \leqslant N \),
\( 1 < p < N \),  
\( \mu = 0 \), 
and either 
\( 0 < a < (k-p)/p \),
and
\( a \leqslant b < a +1 \),
or 
$a \leqslant 0$ and $a < b < a+1$,
and obtained an existence result;
they also proved an existence result to problem~\eqref{classeproblema} when  
$(k-p)/p \leqslant a < k(N-p)/(Np) < a+1$ and 
\( a < b < k[N-p(a+1)]/p(N-k) \).
In these cases, the infimum
$1/K(N,p,0,a,b)$ is attained when either 
$p^*(a,b) < p\sp{*}$,
or  
$p^*(a,a) = p\sp{*}$ and 
$1/K(N,p,0,a,a) < 1/K(N,p,0,0,0)$.
Bhakta~\cite{MR2934676} generalized these results by considering
problem~\eqref{classeproblema}
with a homogenenous  nonlinearity, represented by \( \mu \neq 0 \),
with 
\(N \geqslant 3\),
\( 1 \leqslant k \leqslant N\),
\( 1 < p < N \),
and either
\( 0 = a = b \) and \( 0 \leqslant \mu < \bar{\mu} \),
or
\( a < (k-p)/p \),
\( a < b < a+1 \) and
$\mu <\bar{\mu}$, 
or still
\( 0 < a = b \) and \( \mu^* < \mu < \bar{\mu} \),
where
$\mu^* < \bar{\mu} 
[(N-1)/(N-p)][-ap^2/(N-ap)] <0$. 
In each one of these cases, there exists solution to problem~\eqref{classeproblema}.
For the case of a nonlinearity with a cylindrical weight
that is not a pure power we cite the papers by 
Badiale and Tarantello~\cite{MR1918928}
and by Sintzoff~\cite{MR2262258}.

On the other hand, in the case 
\( \lambda \neq 0 \) we cite the paper by
Filippucci, Pucci and Robert~\cite{MR2498753}, 
where they proved some existence results for the problem~\eqref{classeproblema}
with \( k = N \), \( 1 < p < N \), 
without singularities in the differential operator, represented by 
\( a = 0 \), but with a homogeneous nonlinearity, that is, \( 0 \leqslant \mu < \bar{\mu} \),
and multiple critical nonlinearities, represented by
\( b = 0 \) and \(0 < c <1 \), that is, only one of them with spherical weight. 
For a generalization of this result, see 
Xuan and Wang~\cite{MR2606810},
where the case
\( N \geqslant 3 \),
\( k = N \),
\( 1 < p < N \),
\( 0 \leqslant \mu < \bar{\mu} \),
\( 0 \leqslant a < (N-p)/p \),
\( a \leqslant b < a+1 \), 
and \( a \leqslant c < a+1 \)
is studied; see also
Sun~\cite{MR3072144}, who studied the case
\( N \geqslant 3 \),
\( 2 < k < N \),
\( 1 < p < k \),
\( 0 \leqslant \mu < \bar{\mu} \),
\( a = 0 \),
\( b = 0 \), 
and \( 0 \leqslant c < 1 \). See also Ambrosetti, Br\'{e}zis and Cerami~\cite{MR1276168} for a related problem with sublinear and 
superlinear linearities in the case of the Laplacian operator without singular weights.

Inspired by Gazzini and Musina~\cite{MR2499889} and by  
Bhakta~\cite{MR2934676} regarding the nature of the cylindrical singularities, 
and by 
Filippucci, Pucci and Robert~\cite{MR2498753},
by Xuan and Wang~\cite{MR2606810},
and by Sun~\cite{MR3072144} 
with respect to the presence of multiple critical nonlinearities, 
our first result deals with existence of a positive, weak solution 
to problem~\eqref{problema}. In its statement, we mention the Sobolev space
$\mathcal{D}_a^{1,p}(\mathbb{R}^N
\backslash\{|y|=0\})$, whose definition appears at the beginning of 
section~\ref{sec:passodamontanha}.

Our first result reads as follows.

\begin{teorema} \label{teo:existencia}
Let $2\leqslant k\leqslant N$, $1 < p < N$ and $a <(k-p)/p$;
let $\bar{\mu} \equiv \left[(k-p(a+1))/p \right]^p$. 
Suppose that the parameters
$b$ and $c$ verify one of the following cases.
\begin{enumerate}[label={\upshape(\arabic*)}, 
align=left, widest=3, leftmargin=*]
\item $0=a = b < c < 1$ and $\mu<\bar{\mu}$;
\label{teoexistenciacaso2}
\item $a < b < c < a+1$, $\mu<\bar{\mu}$;
in particular, if $\mu = 0$ we can include the cases 
$(k-p)/p \leqslant a < k(N-p)/Np$ and
$a < b< c<k(N-p(a+1))/p(N-k) < a+1$.
\label{teoexistenciacaso1}
\item $0<a = b < c < a+1$, $\mu^*<\mu<\bar{\mu}$,
\label{teoexistenciacaso3}
where
$\mu^* < \bar{\mu}[-ap^2/(N-ap)][(N-1)/(N-p)] <0$; 
in particular, if 
$\mu = 0$ we can include the same special cases of 
item~\ref{teoexistenciacaso1}. 
\end{enumerate}
Then there exists a function 
$u\in \mathcal{D}_a^{1,p}(\mathbb{R}^N
\backslash\{|y|=0\})$
such that $u>0$ in $\mathbb{R}^N\backslash \{|y|=0\}$ 
and $u$ is a weak solution to problem~\eqref{problema} 
in $\mathbb{R}^N\backslash \{|y|=0\}$.
\end{teorema}

There are several difficulties to prove this existence result. 
In our case we consider
$1<p<N$, and the Sobolev space $\mathcal{D}_a^{1,p}(\mathbb{R}^N\backslash\{|y|=0\})$ does not have the structure of Hilbert spaces except in the particular case $p = 2$. 
Moreover,
we have to deal with cylindrical critical singularities both on the differential operator and on the homogeneous nonlinearity, as well as on the multiple critical nonlinearities; hence, the operator is not uniformly elliptic.
We also have to overcome with the lack of compactness because we consider critical exponents. Indeed, let 
$(u\sb{n})\sb{n \in \mathbb{N}} \subset \mathcal{D}_a^{1,p}(\mathbb{R}^N\backslash\{|y|=0\})$ 
be a minimizing sequence to 
$1/K(N,p,\mu,a,b)$. Then, for arbitrary sequences
$(t\sb{n})\sb{n \in \mathbb{N}} \subset \mathbb{R}\sp{+}$ and
$(\eta\sb{n})\sb{n \in \mathbb{N}} \subset \mathbb{R}\sp{N-k}$, 
the sequence of functions
$(\tilde{u}\sb{n})\sb{n \in \mathbb{N}} \subset \mathcal{D}_a^{1,p}(\mathbb{R}^N\backslash\{|y|=0\})$ 
defined by
\begin{align}
\label{eq:uhtilde}
\tilde{u}\sb{n}(x,y) & \equiv t\sb{n}\sp{(N-p(a+1))/p} \,u\sb{n}(t\sb{n}x+\eta\sb{n},t\sb{n}y)
\end{align}
is also a minimizing sequence for
$1/K(N,p,\mu,a,b)$ because all the integrals 
involved in the definition of the infimum are invariant 
under the action of this group of transformations.
Hence, there exist non-compact minimizing sequences for $1/K(N,p,\mu,a,b)$; this means, for example, that the mountain pass theorem yields Palais-Smale sequences, but not necessarily critical points. To overcome this difficulty, we have to establish sufficient conditions under which the Palais-Smale sequences have strongly convergent subsequences.
Another difficulty is to prove the almost everywhere convergence of the sequences involving integrals of the gradients of the functions, since this result does not follow directly from the already established ones in the literature. Due to these several aspects of problem~\eqref{problema}, the classical methods of critical point theory of the calculus of variations cannot be applied directly.

Additionally, the combination of multiple nonlinearities
with critical exponents and cylindrical weights yields a more subtle and difficult problem
because we have to perform a very detailed analysis of the terms of the energy functional associated to them. 
The main difficulty in this step is to understand the behavior of the Palais-Smale sequences. Indeed, 
in our problem there is a phenomenon that
Filippucci, Pucci and Robert~\cite{MR2498753} named
\lq asymptotic competition\rq\ between the energies carried by the two critical nonlinearities.
And when one of them dominates the other, there is the vanishing of the weakest one and we obtain solutions to problems with only one critical nonlinearity; in other words, we do not obtain nontrivial solutions to problem~\eqref{problema}.
Therefore, we have to avoid the dominance of one term over the other. 
To accomplished this goal, we will maintain the correspondence between each nonlinearity and its singularity 
with the exponents given by the Maz'ya's inequality, and we will choose a suitable level for the mountain pass theorem 
which involves the best Maz'ya's constant defined in~\eqref{imersaolp*(a,b)} and the functions that assume this value.
This constitutes the major contribution of our work.

To prove Theorem~\ref{teo:existencia}, in section~\ref{sec:passodamontanha} 
we show that the energy functional verifies the hypotheses of the mountain pass theorem; 
consequently, there exist Palais-Smale sequences for this functional. 
Then, under appropriate hypotheses, we show that the energy level of these 
Palais-Smale sequences are such that we can recover their strong convergence, 
up to passage to a subsequence. 
In section~\ref{sec:convergenciafracazero} we study the structure of the 
Palais-Smale sequences that are weakly convergent to zero; in this way we can 
identify an appropriate level to avoid the dominance of the energy carried 
by one of the critical nonlinearities over the other. 
Finally, in section~\ref{conclusaoteoexistencia} we show that the limit of this 
sequence is a nontrivial solution to problem~\eqref{problema}.

\bigskip

To complement our existence theorem, we also study a regularity result of weak, positive solution to a problem related to problem~\eqref{problema}.
As is usual in the theory of nonlinear elliptic equations, to show the class of differentiability of the solution we use the iteration scheme introduced by Moser~\cite{MR0132859}. This technique is also described in the books by 
Gilbarg and Trudinger~\cite[Seção~8.6]{MR1814364} and 
by Struwe~\cite[Appendix B]{MR1736116};
see also the paper by 
Br\'{e}zis and Kato~\cite{MR539217}.
Applications of this method can be found in the papers by 
Egnell~\cite{MR956567},
Chou and Chu~\cite{MR1223899}, 
Chou and Geng~\cite{MR1386127}, 
Xuan~\cite{MR2036200},
Alves and Souto~\cite{MR2902126}, 
Vassilev~\cite{MR2719670},
and
Bastos, Miyagaki and Vieira~\cite{bastos}.

More precisely, consider the class of quasilinear 
elliptic equations with cylindrical singularities and multiple nonlinearities
\begin{align}
\label{problemapucciservadei}
-\operatorname{div}\left[\frac{|\nabla u|^{p-2}}{|y|^{ap}}\nabla u\right]-\mu \, \frac{|u|^{p-2}u}{|y|^{p(a+1)}}
& =\frac{(u_+)^{p^*(a,b)-1}}{|y|^{bp^*(a,b)}}+\frac{(u_+)^{p^*(a,c)-1}}{|y|^{cp^*(a,c)}} \qquad (x,y) \in \Omega,
\end{align}
where the domain
\( \Omega \subset \mathbb{R}^{N-k}\times \mathbb{R}^{k}\backslash\{|y|=0\} \) is not necessarily bounded.

Our regularity result can be stated in the following way.

\begin{teorema}
\label{teo:reg}
Suppose that 
\( 1 \leqslant k \leqslant N \),
\( 1 < p < N \),
\(a < (k-p)/p\),
and
\( a \leqslant b < c < a+1 \) and consider the domain 
$\Omega \subset \mathbb{R}\sp{N-k}\times \mathbb{R}\sp{k}\backslash\{|y|=0\}$, not necessarily bounded.
If $u\in\mathcal{D}_a^{1,p}(\Omega)$
is a weak solution to problem~\eqref{problemapucciservadei}, then 
$u \in L_{\mathrm{loc}}^\infty(\Omega)$.
\end{teorema}

The corresponding regularity theorem by 
Filippucci, Pucci and Robert~\cite{MR2498753} is a direct application 
of the results by 
Pucci and Servadei~\cite{MR2492235}, 
by 
Druet~\cite{MR1776675},
and by Guedda and Ver\'{o}n~\cite{MR1009077}.
In our case we cannot apply these conclusions due to the presence of the singularity on the differential operator and we have to prove the result independently.
Inspired by Pucci and Servadei~\cite{MR2492235}, 
in section~\ref{regularidadesecao1} we show, through an inductive step, that 
$u \in L_{\gamma,\textrm{loc}}^m(\Omega)$ 
for every $m\in [1,\infty)$ and for some appropriate weight $\gamma=\gamma(m) \in \mathbb{R}^+$. The main difficulties involved in this part of the proof are related to the required estimates not only for one 
but for multiple critical nonlinearities with cylindrical weights; we also have to make estimates for the term of the energy functional involving the gradient which also has a cylindrical weight. 
In section~\ref{regularidadesecao2} we show, using the Moser's iteration scheme, that
$\lim_{m\to \infty}\|u\|_{L_{\gamma,\textrm{loc}}^m(\Omega)}$ is finite; finally, we conclude that $u\in L_\textrm{loc}^\infty (\Omega)$. 

\bigskip

We note that in Theorem~\ref{teo:existencia} both critical exponents are the ones that make problem~\eqref{problema} invariant under the group of transformations defined by~\eqref{eq:uhtilde}. 
A natural question is what happens when one of the nonlinearities has a different exponent. 
In other words, we consider the class of problems where the exponent in one of the nonlinearities is not critical as determined by Maz'ya's inequality.
In this case, we also observe the \lq asymptotic competition\rq\ phenomenon, and there exists only the trivial solution to the problem. 

More precisely, consider the problem
\begin{align}\label{fprteorema3}
-\operatorname{div}\left[\frac{|\nabla u|^{p-2}}{|y|^{ap}}\nabla u\right]-\mu \, \frac{|u|^{p-2}u}{|y|^{p(a+1)}}
=\frac{|u|^{q-2}u}{|y|^{bp^*(a,b)}} +\frac{|u|^{p^*(a,c)-2}u}{|y|^{cp^*(a,c)}}.
\end{align}

Our nonexistence result of nontrivial solution reads as follows.

\begin{teorema}\label{teo:naoexistencia}
Let $1 \leqslant k \leqslant N$, 
$1<p<N$, $a < (k-p)/p$, 
$a\leqslant b<c<a+1$, 
$0\leqslant \mu < \bar{\mu}$. 
If $u \in \mathcal{D}_{a}^{1,p}(\mathbb{R}^N\backslash\{|y|=0\})$ is a weak solution to problem~\eqref{fprteorema3}, 
then $u \equiv 0$ when 
either $1<q<p^*(a,b)$,  
or $q > p^*(a,b)$ and 
$u \in 
L\sb{bp^*(a,b)/q, \operatorname{loc}}\sp{q}
(\mathbb{R}^N\backslash\{|y|=0\}) 
\cap
L_{\mathrm{loc}}^\infty(\mathbb{R}^N
\backslash\{|y|=0\})$.
\end{teorema}

As is usual in the proofs of nonexistence results, in section~\ref{naoexistenciasecao1} we prove a 
Pohozaev-type identity. The main difficulty is to identify an expression involving at least two terms for the gradient and its corresponding cylindrical singularity due to the fact that we work in 
\( \mathbb{R}\sp{N-k} \times \mathbb{R}\sp{k}\). 
Then, in section~\ref{naoexistenciasecao2} 
we show that applying the Pohozaev-type identity to a solution of a problem related to problem~\eqref{fprteorema3} leads to the vanishing of a somewhat involved integral. Finally, we show that if the dimensional balance involving one of the nonlinearities and its corresponding singular term does not verify the Maz'ya's relation, then the norm of the  solution is zero and we conclude that problem~\eqref{fprteorema3} only has the trivial solution. 

\section{Existence of Palais-Smale sequences}
\label{sec:passodamontanha}

To use the direct method of the calculus of variations, 
we look for solutions to problem~\eqref{problema} 
in the Sobolev space 
$\mathcal{D}_a^{1,p}(\mathbb{R}^N\backslash\{|y|=0\})$ 
defined as the completion of the space
$C\sb{c}\sp{\infty}(\mathbb{R}\sp{N})$ 
of smooth functions with compact support
with respect to the norm defined by
\begin{align}
\label{eq:normagrad}
\left\|\nabla u\right\|_{L_a^p(\mathbb{R}^N)} \equiv \left(\int_{\mathbb{R}^N}\frac{|\nabla u|^p}{|y|^{ap}}\, dz\right)^{1/p}.
\end{align}

It is a well known fact that 
$\mathcal{D}_a^{1,p}(\mathbb{R}^N\backslash\{|y|=0\})$ 
is a reflexive Banach space and that its elements
can be identified with measurable functions up to subsets of measure zero.
Additionaly, from inequality~\eqref{eq:maz} we can deduce that the embedding  
$\mathcal{D}_a^{1,p}(\mathbb{R}^N\backslash\{|y|=0\}) 
\hookrightarrow L\sb{b}\sp{p^*(a,b)}(\mathbb{R}\sp{N})$ is continuous, where 
$L\sb{b}\sp{p^*(a,b)}(\mathbb{R}\sp{N})$ 
denotes the Lebesgue space
$L\sp{p^*(a,b)}(\mathbb{R}\sp{N})$ with weight
$|y|\sp{-bp^*(a,b)}$
and norm defined by
\begin{align*}
\|u\|\sb{L\sb{b}\sp{p^*(a,b)}} 
& \equiv \left( 
\displaystyle\int\sb{\mathbb{R}\sp{N}} 
\frac{|u(z)|\sp{p^*(a,b)}}{|y|\sp{bp^*(a,b)}} 
\, dz \right)^{1/p^*(a,b)}.
\end{align*}

Using Maz'ya's inequality~\eqref{eq:maz} we can show that the embedding 
$\mathcal{D}_a^{1,p}(\mathbb{R}^N\backslash\{|y|=0\}) 
\hookrightarrow L^p_{a+1}(\mathbb{R}\sp{N})$
is continuous for the parameters in the specified 
intervals. More precisely, the inequality
\begin{align*}
\bar{\mu}\int_{\mathbb{R}\sp{N}}
\frac{|u|^p}{|y|^{p(a+1)}}\, dz 
\leqslant 
\int_{\mathbb{R}\sp{N}}
\frac{|\nabla u|^p}{|y|^{ap}}\, dz  
\end{align*}
is valid for every function 
$u\in\mathcal{D}_a^{1,p}(\mathbb{R}^N\backslash\{|y|=0\})$. 
From this inequality,
for $\mu<\bar{\mu}$ we can define a norm 
$\|\cdot\| \colon 
\mathcal{D}_a^{1,p}(\mathbb{R}^N\backslash\{|y|=0\}) 
\to \mathbb{R}$ 
by
\begin{align}
\label{norma}
\|u\|  \equiv \left(
\int_{\mathbb{R}\sp{N}}\frac{|\nabla u|^p}{|y|^{ap}}\, dz 
- \mu\int_{\mathbb{R}\sp{N}}\frac{|u|^p}{|y|^{p(a+1)}}\, dz\right)^\frac{1}{p}
\end{align}
which is well defined in the Sobolev space  
$\mathcal{D}_a^{1,p}(\mathbb{R}^N\backslash\{|y|=0\})$. 
We note that 
the inequalities
\begin{align*}
\left(1-\frac{\mu_+}{\bar{\mu}}\right)\left\|\nabla u\right\|_{L^p_a(\mathbb{R}^N)}^p \leqslant \|u\|^p \leqslant 
\left(1 + \frac{\mu_-}{\bar{\mu}}\right)\left\|\nabla u\right\|_{L^p_a(\mathbb{R}^N)}^p
\end{align*}
are valid for every function 
$u \in \mathcal{D}_a^{1,p}(\mathbb{R}^N\backslash\{|y|=0\})$, 
where 
$\mu_+ = \max\{\mu,0\}$ and 
$\mu_-= \max\{-\mu,0\}$. 
For that reason, the norms defined by~\eqref{eq:normagrad} and~\eqref{norma} are equivalent.

A weak solution to problem~\eqref{problema}
is a function
$u\in \mathcal{D}_a^{1,p}(\mathbb{R}^N\backslash\{|y|=0\})$ 
such that the relation
\begin{align}\label{solucaofraca}
&\int_{\mathbb{R}\sp{N}}\frac{|\nabla u|^{p-2}}{|y|^{ap}}\nabla u\nabla v \, dz 
-\mu\int_{\mathbb{R}\sp{N}} \frac{u^{p-1}}{|y|^{p(a+1)}}v \,dz
=\int_{\mathbb{R}\sp{N}}\frac{(u)_+^{p^*(a,b)-1}}{|y|^{bp^*(a,b)}}v\, dz 
+ \int_{\mathbb{R}\sp{N}}\frac{(u)_+^{p^*(a,c)-1}}{|y|^{cp^*(a,c)}}v\, dz
\end{align}
is valid for every function 
$v \in \mathcal{D}_a^{1,p}(\mathbb{R}^N\backslash\{|y|=0\})$.
Now we define the energy functional
$\varphi\colon \mathcal{D}_a^{1,p}(\mathbb{R}^N\backslash\{|y|=0\}) \to \mathbb{R}$ by
\begin{align}
\label{eq:funcenergia}
\begin{split}
\varphi(u) 
& =\frac{1}{p}\int\sb{\mathbb{R}^N}\frac{|\nabla u|^{p}}{|y|^{ap}}\, dz
-\frac{\mu}{p}\int\sb{\mathbb{R}^N}\frac{|u|^p}{|y|^{p(a+1)}}\, dz \\
& \qquad -\frac{1}{p^*(a,b)}\int\sb{\mathbb{R}^N}\frac{(u_+)^{p^*(a,b)}}{|y|^{bp^*(a,b)}}\, dz
-\frac{1}{p^*(a,c)}\int\sb{\mathbb{R}^N}\frac{(u_+)^{p^*(a,c)}}{|y|^{cp^*(a,c)}}\, dz,
\end{split}
\end{align}
where we use the notation $u_+(x,y)=\max\{u(x,y),0\}$.
It is standard to verify that its G\^{a}teaux derivative is given by 
\begin{align*}
\begin{split}
\langle \varphi' (u), v \rangle
& = \int\sb{\mathbb{R}^N}\frac{|\nabla u|^{p-2}}{|y|^{ap}} 
\langle \nabla u , \nabla v \rangle \, dz
-\mu \int\sb{\mathbb{R}^N}\frac{|u|^{p-2}}{|y|^{p(a+1)}} u v \, dz \\
& \qquad - \int\sb{\mathbb{R}^N}\frac{(u_+)^{p^*(a,b)-2}}{|y|^{bp^*(a,b)}} u v \, dz
-\lambda \int\sb{\mathbb{R}^N}\frac{(u_+)^{q-2}u}{|y|^{cp^*(a,c)}} u v\, dz
\end{split}
\end{align*}
for every $u,v \in \mathcal{D}\sb{a}\sp{1,p}(\mathbb{R}\sp{N} \backslash \{|y|=0\})$. 
Therefore, critical points of this functional are weak solutions to problem~\eqref{problema}.

To prove Theorem~\ref{teo:existencia}, in the first place we show the existence of Palais-Smale sequences for 
suitable levels that will allow us to recover the compactness.
\begin{prop}\label{sequenciaps}
Suppose that the hypotheses of 
Theorem~\ref{teo:existencia} 
are valid and let
$\varphi: \mathcal{D}_a^{1,p}(\mathbb{R}^N\backslash\{|y|=0\}) \to \mathbb{R}$ be the energy functional 
defined in~\eqref{eq:funcenergia}.
Then there exist a Palais-Smale sequence for
$\varphi$ at a level 
\begin{align}\label{intervalod}
0 < d < d_* & \equiv  
\min\sb{m \in \{ b,c \}}
\left\{\left(\frac{1}{p}-\frac{1}{p^*(a,m)}\right)K(N,p,\mu,a,m)^{-\frac{p^*(a,m)}{p^*(a,m)-p}}\right\}.
\end{align}
More specifically, there exists a sequence $(u_n)_{n\in\mathbb{N}} \in \mathcal{D}_a^{1,p}(\mathbb{R}^N\backslash\{|y|=0\})$ such that
\begin{align*}
0 <\lim_{n\to \infty} \varphi(u_n)=d < d_*
\quad \mbox{and} \quad \lim_{n \to \infty} \varphi'(u_n) = 0  
\mbox{ strongly in} 
\left(\mathcal{D}_a^{1,p}(\mathbb{R}^N\backslash\{|y|=0\})\right)^*.
\end{align*}
\end{prop}
To prove Proposition~\ref{sequenciaps}, we begin by showing that we can apply Ambrosetti and Rabinowitz's mountain pass theorem.
See Willem~\cite[Theorem~2.10]{MR1400007}.

\begin{lema}\label{claim2.1} 
For the parameters in the specified intervals, the energy functional~\eqref{eq:funcenergia} verifies the hypotheses of the mountain pass theorem for every
$u \in \mathcal{D}_a^{1,p}(\mathbb{R}^N\backslash\{|y|=0\})$ such that $u_+ \not\equiv 0$, that is,
\begin{enumerate}[label={\upshape(\arabic*)}, 
align=left, widest=2, leftmargin=*]
\item
$\varphi(0)=0$ and there exist $R, \lambda >0$ such that $\varphi|_{\partial B_R(0)} \geqslant \lambda > 0$.
\item
For any $u \in \mathcal{D}_a^{1,p}(\mathbb{R}^N\backslash\{|y|=0\})$, there exists $t_u >0$, such that $\varphi(t_uu) \leqslant 0$, and $\|t_uu\| \geqslant R$.
\end{enumerate}

\end{lema}
\begin{proof}
Clearly we have $\varphi (0) = 0$; moreover, we can prove that
$\varphi \in C^1(\mathcal{D}_a^{1,p}(\mathbb{R}^N\backslash\{|y|=0\}))$
by using standard arguments. 
Using the definition~\eqref{imersaolp*(a,b)} of the optimal constant of the Sobolev embedding, we obtain
\begin{align*}
\varphi(u) 
&\geqslant \frac{1}{p}||u||^p-\frac{(K(N,p,\mu,a,b))^\frac{p^*(a,b)}{p}}{p^*(a,b)}||u||^{p^*(a,b)}
-\frac{(K(N,p,\mu,a,c))^\frac{p^*(a,c)}{p}}{p^*(a,c)}||u||^{p^*(a,c)}.
\end{align*}
Since $p<p^*(a,c)<p^*(a,b)$, there exist $R,\lambda>0$ such that 
$\varphi(u)\geqslant \lambda$ for every 
$u\in \mathcal{D}_a^{1,p}(\mathbb{R}^N\backslash\{|y|=0\})$ 
that verifies the condition $\|u\|=R$;
moreover, 
$\lim_{t \to +\infty} \varphi (tu) = -\infty$. 

Now let $t_u > 0$ be a number such that 
$\varphi(tu) <0$ for every $t \geqslant t_u$, and $\|t_uu\|>R$.
To determine the minimax level, we consider the class of paths connecting the zero function to $t_uu$, that is,
\begin{align*}
\Gamma_u \equiv \left\{\gamma \in C^0\left([0,1], \mathcal{D}_a^{1,p}(\mathbb{R}^N\backslash\{|y|=0\})\right) | \gamma(0)=0 
\mbox{ and } \gamma(1)=t_uu\right\};
\end{align*}
finally, the energy level is given by
\begin{align}\label{eq:du}
d_u \equiv \inf_{\gamma \in \Gamma_u}\sup_{t\in[0,1]}\varphi(\gamma(t))> 0.
\end{align}
Thus, all the hypotheses of the mountain pass theorem are verified by the 
functional \( \varphi \).
\end{proof}
Using the mountain pass theorem, there exists a sequence 
$(u_n)_{n\in\mathbb{N}} \in \mathcal{D}_a^{1,p}(\mathbb{R}^N\backslash\{|y|=0\})$ 
such that
\begin{align*}
\lim_{n\to \infty} \varphi(u_n)=d\sb{u} > 0
\quad \mbox{and} \quad \lim_{n \to \infty} \varphi'(u_n) = 0  
\mbox{ strongly in} 
\left(\mathcal{D}_a^{1,p}(\mathbb{R}^N\backslash\{|y|=0\})\right)^*.
\end{align*}

In the next two lemmas we show that the number $d_u$ is below an appropriate level, for which we can recover the compactness of the Palais-Smale. 
\begin{lema}\label{claim2.2}
Suppose that the hypotheses on one of the items of
Theorems~\ref{teo:existencia} are valid. 
Then there exists a function
$u \in \mathcal{D}_a^{1,p}(\mathbb{R}^N\backslash \{|y|=0\})$ 
such that $u \geqslant 0$ and
\begin{align*}
d_u < \left(\frac{1}{p}-\frac{1}{p^*(a,b)}\right)K(N,p,\mu,a,b)^{-\frac{p^*(a,b)}{p^*(a,b)-p}}.
\end{align*}
\end{lema}
\begin{proof}
Let $u \in \mathcal{D}_a^{1,p}(\mathbb{R}^N\backslash\{|y|=0\})$ be a nonnegative function 
that attains the infimum \( 1/K(N,p,\mu,a,b) \) defined by~\eqref{imersaolp*(a,b)}; the proof of 
the existence of such a function can be seen in the paper by 
Bhakta~\cite[Theorems~1.1 and~1.2]{MR2934676}. Using the definition of $d_u$ given in the proof of Lemma~\ref{claim2.1}, we obtain
$$
d_u \leqslant \sup_{t \geqslant 0} \varphi(tu).
$$

Now let the function $f\sb{1} \colon \mathbb{R}\sb{*}\sp{+} \to \mathbb{R}$ be defined by
\begin{align*}
f\sb{1}(t)
&\equiv \frac{t^p}{p}\left(\int\sb{\mathbb{R}^N}\frac{|\nabla u|^{p}}{|y|^{ap}}\, dz
-\frac{\mu}{p}\int\sb{\mathbb{R}^N}\frac{|u|^p}{|y|^{p(a+1)}}\, dz\right)
- \frac{t^{p^*(a,b)}}{p^*(a,b)}\left(\int_{\mathbb{R}^N}\frac{|u|^{p^*(a,b)}}{|y|^{bp^*(a,b)}}\, dz\right).
\end{align*}
Denoting by \( t\sb{\max} \) the point of maximum for \( f\sb{1} \), we obtain
\begin{align*}
d_u \leqslant \sup_{t \geqslant 0} \varphi(tu) 
\leqslant \sup_{t \geqslant 0} f_1(t)
= f\sb{1}(t\sb{\max})
= \left(\frac{1}{p}-\frac{1}{p^*(a,b)}\right)K(N,p,\mu,a,b)^{-\frac{p^*(a,b)}{p^*(a,b)-p}}.
\end{align*}

To show that this inequality is strict, we argue by contradiction and we suppose that the equality is valid.
Denoting by $t_0 >0$ the factor of the extremal $u$ where the supremum of the energy functional is attained, 
we obtain 
\begin{align*}
d_u = \varphi(t_0 u) = f_1(t_0) - \frac{t_0^{p^*(a,c)}}{p^*(a,c)}\int_{\mathbb{R}^N}\frac{|u|^{p^*(a,c)}}{|y|^{cp^*(a,c)}}\, dz
=f_1(t\sb{\max}).
\end{align*}
This means that $f_1(t\sb{\max}) < f_1(t_0)$, which is a contradiction.
The result follows.
\end{proof}

\begin{lema}\label{claim2.3}
Suppose that the hypotheses on one of the items of
Theorems~\ref{teo:existencia} are valid. 
Then there exists a function 
$u \in \mathcal{D}_a^{1,p}(\mathbb{R}^N\backslash\{|y|=0\})$ 
such that $u \geqslant 0$ and $0 < d_u < d_*$.
\end{lema}

\begin{proof}
If
\(
d\sb{*} = [1/p -1/p^*(a,b)] K(N,p,\mu,a,b)^{-p^*(a,b)/(p^*(a,b)-p)}
\)
then it suffices to consider the function 
$u \in \mathcal{D}_a^{1,p}(\mathbb{R}^N\backslash\{|y|=0\})\backslash\{0\}$ 
given in Lemma~\ref{claim2.2} and we get $d_u < d^*$.
Otherwise,
let $u \in \mathcal{D}_a^{1,p}(\mathbb{R}^N\backslash\{|y|=0\})$ be a nonnegative function such 
that the infimum \( 1/K(N,p,\mu,a,c) \) defined by~\eqref{imersaolp*(a,b)} is attained and let the function \( f\sb{2} \colon \mathbb{R}\sb{*}\sp{+} \to \mathbb{R} \) be defined by 
\begin{align*}
f\sb{2}(t) \equiv \frac{t^p}{p}\left(\int\sb{\mathbb{R}^N}\frac{|\nabla u|^{p}}{|y|^{ap}}\, dz
-\frac{\mu}{p}\int\sb{\mathbb{R}^N}\frac{|u|^p}{|y|^{p(a+1)}}\, dz\right)
-\frac{t^{p^*(a,c)}}{p^*(a,c)}\int\sb{\mathbb{R}^N}\frac{|u|^{p^*(a,c)}}{|y|^{cp^*(a,c)}}\, dz.
\end{align*}
Arguing as in the proof of Lemma~\ref{claim2.2} we obtain the inequality $d_u < d^*$. 
Finally, the mountain pass theorem guarantees that $d_u > 0$. 
This concludes the proof of the lemma.
\end{proof}
\begin{proof}[Proof of Proposition~\ref{sequenciaps}]
By Lemma~\ref{claim2.1} the energy functional $\varphi$ verifies the hypotheses of the mountain pass theorem. 
Hence, there exists a Palais-Smale sequence 
$(u_n)_{n\in\mathbb{N}}\subset \mathcal{D}_a^{1,p}(\mathbb{R}^N\backslash\{|y|=0\})$  
for the functional $\varphi$ at the level $d$; 
and by Lemma~\ref{claim2.3}, we conclude that $d<d_*$. The proposition is proved.
\end{proof}

\section{Sequences weakly convergent to zero}
\label{sec:convergenciafracazero}
Now we are going to study the behavior of the Palais-Smale sequences.

\begin{prop}
\label{convergenciapara0}
Let \( (u_n)_{n\in\mathbb{N}} \subset \mathcal{D}_a^{1,p}(\mathbb{R}^N\backslash\{|y|=0\}) \)
be a Palais-Smale sequence for the functional \( \varphi \) 
at a level \( d \) such that 
\(0 < d < d_* \) as in Proposition~\ref{sequenciaps} 
and suppose that the hypotheses on one of the items of 
Theorem~\ref{teo:existencia} are valid. 
If $u_n \rightharpoonup 0$ weakly in $\mathcal{D}_a^{1,p}(\mathbb{R}^N\backslash\{|y|=0\})$ 
as $n \to \infty$, then
for every  $\delta >0$ one of the following claims is valid.
\begin{enumerate}[label={\upshape(\arabic*)}, 
align=left, widest=2, leftmargin=*]
\item
\(\displaystyle\lim_{n \to \infty} \int_{B_\delta(0)}\frac{(u_n)_+^{p*(a,b)}}{|y|^{bp*(a,b)}}\, dz  =0\) 
and
\(\displaystyle\lim_{n \to \infty} \int_{B_\delta(0)}\frac{(u_n)_+^{p*(a,c)}}{|y|^{cp*(a,c)}}\, dz =0\).
\item
\(\displaystyle\limsup_{n \to \infty} \int_{B_\delta(0)}\frac{(u_n)_+^{p*(a,b)}}{|y|^{bp*(a,b)}}\, dz  \geqslant \epsilon_0\)
and  
\(\displaystyle\limsup_{n \to \infty} \int_{B_\delta(0)}\frac{(u_n)_+^{p*(a,c)}}{|y|^{cp*(a,c)}}\, dz  \geqslant \epsilon_0\)
for some number
\(\epsilon_0 = \epsilon_0(N,p,\mu,c,d) > 0\).
\end{enumerate}
\end{prop}
To prove Proposition~\ref{convergenciapara0} we establish some lemmas.

\begin{lema}
\label{claim3.1}
Let $(u_n)_{n\in\mathbb{N}} \subset \mathcal{D}_a^{1,p}(\mathbb{R}^N\backslash\{|y|=0\})$ 
be a Palais-Smale sequence as in Proposition~\ref{convergenciapara0}. 
If $u_n \rightharpoonup 0$ weakly in $\mathcal{D}_a^{1,p}(\mathbb{R}^N\backslash\{|y|=0\})$ 
as $n \to \infty$, then for every compact subset
$\omega \Subset \mathbb{R}^N\backslash\{|y|=0\}$,
up to passage to a subsequence we have
\begin{alignat}{3}
\label{10}
\lim_{n \to \infty} \int_\omega\frac{|u_n|^p}{|y|^{p(a+1)}}\, dz
& = 0 
& \quad \textrm{ and } && \quad 
\lim_{n \to \infty} \int_\omega\frac{|u_n|^{p*(a,c)}}{|y|^{cp*(a,c)}}\, dz
& = 0,\\
\lim_{n \to \infty} \int_\omega\frac{|u_n|^{p*(a,b)}}{|y|^{bp*(a,b)}}\, dz
& = 0 
& \quad \textrm{ and } && \quad
\lim_{n \to \infty} \int_\omega\frac{|\nabla u_n|^p}{|y|^{ap}}\, dz 
& = 0. \label{11}
\end{alignat}
\end{lema}
\begin{proof}
Let $\omega \Subset \mathbb{R}^N\backslash\{|y|=0\}$ be a fixed compact subset. Thus, the expression $|y|+|y|^{-1}$ is bounded for every $ z=(x,y) \in \omega$. 
Since $p^*(a,a+1)=p$, by Maz'ya's inequality~\eqref{eq:maz} and by
the compact embedding 
$\mathcal{D}_a^{1,p}(\mathbb{R}^N\backslash\{|y|=0\}) 
\hookrightarrow L_{a+1}^p(\mathbb{R}^N\backslash\{|y|=0\})$, 
it follows that the first limit in~\eqref{10} is valid. 
Likewise, using the fact that \( a \leqslant c < a+1 \) we can show that
the second limit in~\eqref{10} is valid also.

To show that the limits~\eqref{11} are valid we make some estimates.
Let $\eta \in C_c^\infty (\mathbb{R}^N\backslash\{|y|=0\}$) be a cut off function such that $0 \leqslant \eta \leqslant 1$, with 
\( \operatorname{supp} \nabla \eta \equiv \omega \). 

\begin{afirmativa}
\label{afirmativa14}
It is valid the relation
\begin{align*}
\int_{\mathbb{R}^N}\frac{|\nabla (\eta u_n)|^p}{|y|^{ap}}\, dz
= \int_{\mathbb{R}^N}\frac{|\eta \nabla u_n|^p}{|y|^{ap}}\, dz + o(1).
\end{align*}
\end{afirmativa}

\begin{proof}
Applying the inequality
$\left ||X+Y|^P-|X|^P\right | \leqslant C_p\left(|X|^{p-1}+|Y|^{p-1}\right)|Y|$
to the values
$X=|\eta\nabla u_n|/|y|^a$ and 
$Y=|u_n\nabla \eta|/|y|^{a}$, we obtain
\begin{align}
\label{desigualdadedaafirmativa1}
\left| \left| \frac{\nabla (\eta u_n)}{|y|^{a}}\right|^p-\left|\frac{|\eta\nabla u_n|}{|y|^{a}}\right|^p\right|
&\leqslant C_p\frac{|\eta\nabla u_n|^{p-1}}{|y|^{a(p-1)}}\frac{|u_n\nabla \eta|}{|y|^a}
+C_p\frac{|u_n\nabla \eta|^{p}}{|y|^{ap}}.
\end{align}

To proceed, we apply H\"{o}lder's inequality to the integral over 
\( \mathbb{R}\sp{N} \) of the first term on the right-hand side of 
inequality~\eqref{desigualdadedaafirmativa1} to obtain
\begin{align*}
\int_{\mathbb{R}^N}\frac{|\eta\nabla u_n|^{p-1}}{|y|^{a(p-1)}}\frac{|u_n\nabla \eta|}{|y|^a}\, dz
&\leqslant C\sb{p} \left(\int_{\mathbb{R}^N}\frac{|\nabla u_n|^p}{|y|^{ap}}\, dz\right)^\frac{p-1}{p} 
\left(\int_{\omega}\frac{|u_n|^p}{|y|^{p(a+1)}}\, dz\right)^\frac{1}{p} 
= o(1).
\end{align*}
On the other hand, using the first limit in~\eqref{10}, 
it follows that the integral over \( \mathbb{R}\sp{N} \)
of the second term on the right-hand side of inequality~\eqref{desigualdadedaafirmativa1} is such that
\begin{align*}
\int_{\mathbb{R}^N} \frac{|u_n\nabla \eta|^{p}}{|y|^{ap}}\, dz 
\leqslant C\int_{\omega}\frac{|u_n|^{p}}{|y|^{p(a+1)}}\, dz = o(1).
\end{align*}
for some positive constant \( C > 0 \).
Combining these results the claim follows.
\end{proof}

Recall that
$(u_n)_{n\in\mathbb{N}}$ is a Palais-Smale sequence 
and that $\eta^p u_n \in \mathcal{D}_a^{1,p}(\mathbb{R}^N\backslash\{|y|=0\})$; 
for that reason, as $n \to \infty$ we have
\begin{align}\label{ode1}
\left\langle \varphi'(u_n),\eta^p u_n \right\rangle
= o(1).
\end{align}
Now we note that
$$
\lim_{n\to+\infty} \left\|\frac{u_n}{|y|^{a+1}}\right\|_{L^p(\omega)}
=\lim_{n\to+\infty}\int_{\omega}\frac{|u_n|^p}{|y|^{p(a+1)}}\, dz=0
$$
by the first limit in~\eqref{10}, where we used 
$\omega=\operatorname{supp} \nabla \eta 
\Subset \mathbb{R}^N\backslash \{|y|=0\}$. 
Furthermore, the sequence of the norms of the gradients
$(\|\nabla u_n\|_{L^p_a(\mathbb{R}^N\backslash \{|y|=0\})})_{n\in\mathbb{N}} 
\subset \mathbb{R}$ 
is bounded due to the weak convergence
$u_n \rightharpoonup 0$ in 
$\mathcal{D}_a^{1,p}(\mathbb{R}^n\backslash\{|y|=0\})$. 
Hence, using H\"{o}lder's inequality and the fact that
$|y|$ is bounded in 
$\operatorname{supp}|\nabla\eta| = \omega $, we obtain
\begin{align*}
\int_{\mathbb{R}^N}\frac{|\nabla u_n|^{p-1}}{|y|^{ap}}|\nabla \eta^p||u_n|\, dz
&\leqslant C\left(\int_{\mathbb{R}^N}\frac{|\nabla u_n|^p}{|y|^{ap}}\, dz\right)^\frac{p-1}{p}
\left(\int_{\omega}\frac{|u_n|^p}{|y|^{p(a+1)}}\, dz\right)^\frac{1}{p} = o(1),
\end{align*}
as $n \to +\infty$. 
Then, by the limits~\eqref{10} and~\eqref{ode1}, by the previous inequality together with Claim~\ref{afirmativa14} and H\"{o}lder's inequality, it follows that
\begin{align}\label{13}
\int_{\mathbb{R}^N}\frac{|\nabla (\eta u_n)|^p}{|y|^{ap}}\, dz
& = 
\int_{\mathbb{R}^N}\frac{(u_n)_+^{p^*(a,b)}\eta ^p}{|y|^{bp^*(a,b)}}\, dz 
+ o(1) \nonumber \\
& \leqslant 
\left(\int_{\mathbb{R}^N}\frac{(u_n)_+^{p^*(a,b)}}{|y|^{bp^*(a,b)}}\, dz\right)^\frac{p^*(a,b)-p}{p^*(a,b)}
\left(\int_{\mathbb{R}^N}\frac{|\eta u_n|^{p^*(a,b)}}{|y|^{bp^*(a,b)}}\, dz\right)^\frac{p}{p^*(a,b)}+o(1) \nonumber \\
& \leqslant \left(\int_{\mathbb{R}^N}\frac{(u_n)_+^{p^*(a,b)}}{|y|^{bp^*(a,b)}}\, dz\right)^\frac{p^*(a,b)-p}{p^*(a,b)}
K(N,p,\mu,a,b)\int_{\mathbb{R}^N}\frac{|\nabla (\eta u_n)|^p}{|y|^{ap}}\, dz + o(1).
\end{align}
Consequently,
\begin{align}\label{15}
\Bigg(1-\left(\int_{\mathbb{R}^N}\frac{(u_n)_+^{p^*(a,b)}}{|y|^{bp^*(a,b)}}\, dz\right)^\frac{p^*(a,b)-p}{p^*(a,b)}
K(N,p,\mu,a,b)\Bigg)\int_{\mathbb{R}^N}\frac{|\nabla (\eta u_n)|^p}{|y|^{ap}}\, dz \leqslant o(1).
\end{align}

On the other hand, since 
$(u_n)_{n\in\mathbb{N}}\subset 
\mathcal{D}_a^{1,p}(\mathbb{R}^N\backslash\{|y|=0\})$
is a Palais-Smale sequence at a level $d$, 
direct computations show that
\begin{align}
\label{16}
d + o(1) & =
\varphi(u_n) - \dfrac{1}{p}\left\langle \varphi'(u_n),u_n \right\rangle \nonumber \\
& =
\left(\frac{1}{p}-\frac{1}{p^*(a,b)}\right)\int_{\mathbb{R}^N}\frac{(u_n)_+^{p^*(a,b)}}{|y|^{bp^*(a,b)}}\, dz
+\left(\frac{1}{p}-\frac{1}{p^*(a,c)}\right)\int_{\mathbb{R}^N}\frac{(u_n)_+^{p^*(a,c)}}{|y|^{cp^*(a,c)}}\, dz + o(1) \\
& \geqslant 
\left(\frac{1}{p}-\frac{1}{p^*(a,c)}\right)\int_{\mathbb{R}^N}\frac{(u_n)_+^{p^*(a,c)}}{|y|^{cp^*(a,c)}}\, dz + o(1) \nonumber.
\end{align}
This implies that 
\begin{align}\label{17}
&\int_{\mathbb{R}^N}\frac{(u_n)_+^{p^*(a,b)}}{|y|^{bp^*(a,b)}}\, dz 
\leqslant d\left(\frac{1}{p}-\frac{1}{p^*(a,b)}\right)^{-1} + o(1).
\end{align}
Replacing inequality~\eqref{17} in the relation~\eqref{15}, it follows that
\begin{align*}
&\Bigg( 1 - \Bigg( d \Bigg(\frac{1}{p}-\frac{1}{p^*(a,b)}\Bigg)^{-1} \Bigg)^{\frac{p*(a,b)-p}{p^*(a,b)}}K(N,p,\mu,a,b)\Bigg)
\int_{\mathbb{R}^N}\frac{|\nabla(\eta u_n)|^p}{|y|^{ap}}\, dz \leqslant o(1).
\end{align*}
Finally, to show that the factor that multiplies the previous integral is positive, we use inequality~\eqref{intervalod}.
Hence the second limit in~\eqref{11} is valid and this concludes the proof of the lemma.
\end{proof}

For a given $\delta > 0$, we define
\begin{align}
\label{18}
  \begin{array}{r@{{}\equiv{}}l@{\qquad}r@{{}\equiv{}}l}
    \alpha 
    & \displaystyle\limsup_{n \to +\infty} \int_{B_\delta (0)}\frac{(u_n)_+^{p^*(a,b)}}{|y|^{bp^*(a,b)}}\, dz, 
    & \beta 
    & \displaystyle\limsup_{n \to +\infty} \int_{B_\delta (0)}\frac{(u_n)_+^{p^*(a,c)}}{|y|^{cp^*(a,c)}}\, dz, \\[\jot]
\multicolumn{4}{c}{\gamma \equiv
\limsup_{n \to +\infty}\bigg( 
\displaystyle\int_{B_\delta (0)}\left(\frac{|\nabla u_n|^p}{|y|^{ap}} \, dz \right)
- \mu \displaystyle\int_{B_\delta (0)}
\left(\frac{|u_n|^p}{|y|^{p(a+1)}}\right) \, dz 
\bigg).}
\end{array}
\end{align}
It follows from Lemma~\ref{claim3.1} that these values are well defined and do not depend on the choice of  $\delta > 0$.

\begin{lema}\label{claim3.2}
Let $(u_n)_{n\in\mathbb{N}} \subset \mathcal{D}_a^{1,p}(\mathbb{R}^N\backslash\{|y|=0\})$ 
be Palais-Smale sequence as in 
Proposition~\ref{convergenciapara0} 
and let $\alpha$, $\beta$, and $\gamma$ 
be defined as in~\eqref{18}. 
If $u_n \rightharpoonup 0$ weakly in  $\mathcal{D}_a^{1,p}(\mathbb{R}^N\backslash\{|y|=0\})$ 
as $n \to +\infty$, then
\begin{align}\label{19}
\alpha^\frac{p}{p^*(a,b)} \leqslant K(N,p,\mu,a,b)\gamma \qquad \mbox{and} \qquad
\beta^\frac{p}{p^*(a,c)} \leqslant K(N,p,\mu,a,c)\gamma.
\end{align}
\end{lema}

\begin{proof}
Let $R >\delta >0$ ant let
$\eta \in C_c^\infty(\mathbb{R}^N)$ 
be a cut off function such that
$\eta|_{B_\delta(0)} \equiv 1$ and 
$\eta|_{\mathbb{R}^N\backslash B_R(0)} \equiv 0$. 
By the definition of the infimum~\eqref{imersaolp*(a,b)} 
we have
\begin{align*}
{\left(\displaystyle\int_{\mathbb{R}^N}\frac{(\eta u_n)_+^{p^*(a,b)}}{|y|^{bp^*(a,b)}}\, dz\right)^\frac{p}{p^*(a,b)}}
\leqslant K(N,p,\mu,a,b)\left( \displaystyle\int_{\mathbb{R}^N}\frac{|\nabla (\eta u_n)|^p}{|y|^{ap}} \, dz 
-\mu\int_{\mathbb{R}^N}\frac{|\eta u_n|^p}{|y|^{p(a+1)}} \, dz\right).
\end{align*}
This inequality, together with
Claim~\ref{afirmativa14} and Lemma~\ref{claim3.1} 
imply that
\begin{align}\label{19.3}
\left(\displaystyle\int_{B_\delta(0)}\frac{(u_n)_+^{p^*(a,b)}}{|y|^{bp^*(a,b)}}\, dz\right)^\frac{p}{p^*(a,b)}
& \leqslant K(N,p,\mu,a,b)
\left(\displaystyle\int_{B_\delta(0)}
\frac{|\nabla u_n|^p}{|y|^{ap}} \, dz
-\mu\int_{B_\delta(0)}\frac{|u_n|^p}{|y|^{p(a+1)}} \, dz\right) + o(1).
\end{align}
Using this inequality, we conclude that
\begin{align*}
\alpha^\frac{p}{p*(a,b)} &= \limsup_{n \to \infty} \left(\displaystyle\int_{B_\delta(0)}\frac{(u_n)_+^{p^*(a,b)}}{|y|^{bp^*(a,b)}}\, dz\right)^\frac{p}{p^*(a,b)}\\
& \leqslant \limsup_{n \to \infty} K(N,p,\mu,a,b)\left(\displaystyle\int_{B_\delta(0)}\frac{|\nabla u_n|^p}{|y|^{ap}} \, dz
-\mu\int_{B_\delta(0)}\frac{|u_n|^p}{|y|^{p(a+1)}} \, dz\right)\\
& =  K(N,p,\mu,a,b)\gamma,
\end{align*}
which is the first inequality in~\eqref{19}. The proof of the other inequality in~\eqref{19} is similar.
\end{proof}

\begin{lema}\label{claim3.3}
Let $(u_n)_{n\in\mathbb{N}} 
\subset \mathcal{D}_a^{1,p}(\mathbb{R}^N\backslash\{|y|=0\})$ be a 
Palais-Smale for the functional $\varphi$ 
at the level $d \in (0,d_*)$ 
and let $\alpha, \beta$, and $\gamma$ 
be defined as in em~\eqref{18}. 
If $u_n \rightharpoonup 0$ weakly in
$\mathcal{D}_a^{1,p}(\mathbb{R}^N\backslash\{|y|=0\})$ as $n \to +\infty$,
then $\gamma \leqslant \alpha + \beta$.
\end{lema}

\begin{proof}
By the hypothesis on the sequence
$(u_n)_{n\in\mathbb{N}} 
\subset \mathcal{D}_a^{1,p}(\mathbb{R}^N\backslash\{|y|=0\})$,
we have
\begin{align*}
0 
&=\limsup_{n\to \infty} \left\langle \varphi'(u_n),\eta u_n
\right\rangle\\
&=\limsup_{n\to \infty} \left\{
\splitfrac{\displaystyle\int_{B_\delta(0)}
\frac{|\nabla u_n|^{p-2}}{|y|^{ap}}
\nabla u_n\nabla(\eta u_n) \, dz
-\mu \displaystyle\int_{B_\delta(0)}
\frac{|u_n|^{p-2}}{|y|^{p(a+1)}}u_n(\eta u_n) \, dz }
{- \displaystyle\int_{B_\delta(0)}
\frac{(u_n)_+^{p^*(a,b)-2}}
{|y|^{bp^*(a,b)}}(u_n)_+(\eta u_n)\, dz
- \displaystyle\int_{B_\delta(0)}
\frac{(u_n)_+^{p^*(a,c)-2}}{|y|^{cp^*(a,c)}}
(u_n)_+(\eta u_n)\, dz } \right\}.
\end{align*}
Hence, as $\eta \equiv 1$ in $B_\delta(0)$, 
we obtain $\gamma \leqslant \alpha + \beta$,
which concludes the proof of the lemma.
\end{proof}

\begin{proof}[Proof of 
Proposition~\ref{convergenciapara0}]
Let $(u_n)_{n\in\mathbb{N}} 
\subset \mathcal{D}_a^{1,p}(\mathbb{R}^N\backslash\{|y|=0\})$ 
be a Palais-Smale sequence for the functional $\varphi$
at the level $d \in (0,d_*)$. 
Lemmas~\ref{claim3.2} and~\ref{claim3.3} imply that
\begin{align*}
&\alpha^\frac{p}{p^*(a,b)}\left(1-K(N,p,\mu,a,b)\alpha^\frac{p^*(a,b)-p}{p^*(a,b)}\right)\leqslant K(N,p,\mu,a,b)\beta.
\end{align*}
Moreover, passing to the limit superior in both sides of inequality~\eqref{17}, we obtain
\begin{align*}
\alpha \leqslant d\left(\frac{1}{p}-\frac{1}{p^*(a,b)}\right)^{-1}.
\end{align*}
Combining these two inequalities, we get
$$
\alpha^\frac{p}{p^*(a,b)}
\Bigg(1-K(N,p,\mu,a,b)
\Bigg(d\left(\frac{1}{p}
-\frac{1}{p^*(a,b)}\right)^{-1}
\Bigg)^\frac{p^*(a,b)-p}{p^*(a,b)}
\Bigg)
\leqslant K(N,p,\mu,a,b)\beta.
$$
Finally, to show that the factor that multiplies the integral is positive, again we use inequality~\eqref{intervalod}.
Hence, there exists a positive constant
$\delta_1=\delta_1(N,p,\mu,a,b,d) > 0$ 
such that
$\alpha^\frac{p}{p^*(a,b)}\leqslant \delta_1 \beta$. Repeating the computations for 
$\beta$ in the place of  $\alpha$, 
we conclude that there exists a positive constant
$\delta_2=\delta_2(N,p,\mu,a,b,d) > 0$ 
such that
$\beta^\frac{p}{p^*(a,c)} \leqslant \delta_2 \alpha$. 
In particular, from these inequalities it follows that there exists a positive constant 
$\epsilon_0 =\epsilon_0(N,p,\mu,b,c,d) > 0$ 
such that
either
$\alpha = 0 $ and $\beta = 0$, 
or $\alpha \geqslant \epsilon_0$ and $\beta \geqslant \epsilon_0$. 
This concludes the proof of the proposition.
\end{proof}

\section{Conclusion of the proof of 
Theorem~\ref{teo:existencia}}
\label{conclusaoteoexistencia}

To prove Theorem~\ref{teo:existencia} we have to 
estimate both the limits superior in Proposition~\ref{convergenciapara0}.

\begin{lema}\label{claim4.1}
Let $(u_n)_{n\in\mathbb{N}} \subset \mathcal{D}_a^{1,p}(\mathbb{R}^N\backslash\{|y|=0\})$ be a Palais-Smale sequence for the functional $\varphi$ at the level 
$d \in (0,d_*)$. Then
\begin{align*}
\min\sb{m \in \{ b,c \}}
\left\{\limsup_{n\to +\infty} \int_{\mathbb{R}^N}
\frac{(u_n)_+^{p^*(a,m)}}{|y|^{bp^*(a,m)}}\, dz
\right\} > 0.
\end{align*}
\end{lema}

\begin{proof}
Without loss of generality and arguing by contradiction, we can suppose that  
\begin{align}\label{23}
0 = \limsup_{n\to +\infty} \int_{\mathbb{R}^N}\frac{(u_n)_+^{p*(a,b)}}{|y|^{bp^*(a,b)}}\, dz
\leqslant \limsup_{n\to +\infty}\int_{\mathbb{R}^N}\frac{(u_n)_+^{p*(a,c)}}{|y|^{bp^*(a,c)}}\, dz.
\end{align}
Therefore, from the fact that 
$\left\langle \varphi'(u_n), u_n \right\rangle \to 0$ as \( n \to \infty \), we obtain
\begin{align*}
\int_{\mathbb{R}^N}\frac{|\nabla u_n|^p}{|y|^{ap}} \, dz -\mu \int_{\mathbb{R}^N}\frac{|u_n|^p}{|y|^{p(a+1)}}\, dz
= \int_{\mathbb{R}^N}\frac{(u_n)_+^{p^*(a,c)}}{|y|^{cp^*(a,c)}}\, dz + o(1).
\end{align*}
Using this inequality and  the definition of the infimum~\eqref{imersaolp*(a,b)}, it follows that
\begin{align*}
&\left(\int_{\mathbb{R}^N}\frac{(u_n)_+^{p^*(a,c)}}{|y|^{cp^*(a,c)}}\, dz\right)^\frac{p}{p^*(a,c)}
 \leqslant K(N,p,\mu,a,c) \int_{\mathbb{R}^N}\frac{(u_n)_+^{p^*(a,c)}}{|y|^{cp^*(a,c)}}\, dz + o(1).
\end{align*}
Consequently,
\begin{align}\label{24}
\Bigg(\int_{\mathbb{R}^N}\frac{(u_n)_+^{p^*(a,c)}}{|y|^{cp^*(a,c)}}\, dz\Bigg)^\frac{p}{p^*(a,c)}
\Bigg(1-K(N,p,\mu,a,c)\Bigg(\int_{\mathbb{R}^N}\frac{(u_n)_+^{p^*(a,c)}}{|y|^{cp^*(a,c)}}\, dz\Bigg)^\frac{p^*(a,c)-p}{p^*(a,c)}\Bigg) & \leqslant o(1).
\end{align}

On the other hand, using equality~\eqref{16} and inequality~\eqref{23}, for $n\in\mathbb{N}$ big enough we have
$$
\int_{\mathbb{R}^N}\frac{(u_n)_+^{p^*(a,c)}}{|y|^{cp^*(a,c)}}\, dz
=d\left(\frac{1}{p}-\frac{1}{p^*(a,c)}\right)^{-1} + o(1).
$$
Replacing this equality in~\eqref{24}, it follows that
\begin{align}\label{24b}
&\left(\int_{\mathbb{R}^N}\frac{(u_n)_+^{p^*(a,c)}}{|y|^{cp^*(a,c)}}\, dz\right)^\frac{p}{p^*(a,c)}\nonumber\\
& \qquad \times\Bigg(1-K(N,p,\mu,a,c)
\Bigg(d \Bigg(\frac{1}{p}-
\frac{1}{p^*(a,c)}\Bigg)^{-1}
+ o(1)\Bigg)^\frac{p^*(a,c)-p}{p^*(a,c)}\Bigg) 
\leqslant o(1).
\end{align}
Using once again the definition~\eqref{intervalod} of $d_*$, we conclude that
$$
\lim_{n \to +\infty} \int_{\mathbb{R}^N}\frac{(u_n)_+^{p^*(a,c)}}{|y|^{cp^*(a,c)}}\, dz = 0.
$$
And this equality, together with~\eqref{16} imply that 
$d=0$, which is a contradiction with the hypothesis that 
$d>0$. This concludes the proof of the lemma.
\end{proof}

Using the value of 
$\epsilon_0 >0$ identified in Proposition~\ref{convergenciapara0} and also 
Lemma~\ref{claim4.1}, we can establish one more result that guarantees that the action of the group of transformations defined in~\eqref{eq:uhtilde} preserves the Palais-Smale sequences at the level \( d \).

\begin{lema}\label{claim4.2}
Let $(u_n)_{n\in\mathbb{N}} 
\subset \mathcal{D}_a^{1,p}(\mathbb{R}^N\backslash\{|y|=0\})$ be a Palais-Smale sequence for the functional
$\varphi$ at a level 
\( d \in (0,d_*)\) verifying the hypothesis of Proposition~\ref{convergenciapara0}. 
Then there exists \( \epsilon_1 \in (0,\epsilon_0/2] \) 
and there exist sequences
$(t_n)_{n\in\mathbb{N}} \subset \mathbb{R}$ and $(\eta_n)_{n\in\mathbb{N}} \subset \mathbb{R}^{N-k}$ such that for every $\epsilon \in (0,\epsilon_1)$, 
there exists a sequence 
$(\tilde{u}_n)_{n\in\mathbb{N}} \subset \mathcal{D}_a^{1,p}(\mathbb{R}^N\backslash\{|y|=0\})$, 
defined by~\eqref{eq:uhtilde}
that is also a Palais-Smale sequence for the functional  
$\varphi$ at the level $d$. 
Moreover, this sequence 
verifies the equality
\begin{align}\label{25}
\int_{B_1(0)}\frac{(\tilde{u}_n)_+^{p^*(a,b)}}{|y|^{bp^*(a,b)}}\, dz = \epsilon
\end{align}
for every \( n \in \mathbb{N} \).
\end{lema}

\begin{proof}
Let $\lambda \equiv \displaystyle\limsup_{n \to \infty}\int_{\mathbb{R}^N}\frac{(u_n)_+^{p^*(a,b)}}{|y|^{bp^*(a,b)}}\, dz$. 
It follows from Lemma~\ref{claim4.1} 
that $\lambda > 0$. 
Let $\epsilon_1 \equiv \min\{\epsilon_0/2,\lambda\}$; for the rest of the proof we fix 
$\epsilon \in (0,\epsilon_1)$. 
Passing to a subsequence, 
still denoted by 
$(u_n)_{n\in\mathbb{N}}\subset \mathcal{D}_a^{1,p}(\mathbb{R}^N\backslash\{|y|=0\})$, 
for every $n\in\mathbb{N}$ there exists
$(r_n)_{n\in\mathbb{N}}>0$ such that
$$
\int_{B_{r_n}(0)}\frac{(u_n)_+^{p^*(a,b)}}{|y|^{bp^*(a,b)}}\, dz = \epsilon.
$$
Due to the invariance of this integral under the action of the group of homoteties and translations, it is standard to prove that the sequence 
$(\tilde{u}_n)_{n\in\mathbb{N}}
\subset \mathcal{D}_a^{1,p}
(\mathbb{R}^N\backslash\{|y|=0\})$ verifies the 
equality~\eqref{25} and also the conclusions of 
Proposition~\ref{convergenciapara0}. This concludes the proof of the lemma.
\end{proof}

\begin{prop}
\label{fprclaim4.3}
Let 
$(u_n)_{n\in\mathbb{N}} 
\subset \mathcal{D}_a^{1,p}(\mathbb{R}^N\backslash\{|y|=0\})$ 
be a Palais-Smale for the functional
$\varphi$ at the level $d$ 
and let
$(\tilde{u}_n)_{n\in\mathbb{N}} 
\subset \mathcal{D}_a^{1,p}
(\mathbb{R}^N\backslash\{|y|=0\})$ 
be the sequence defined in~\eqref{eq:uhtilde}. 
Then there exists a function 
$\tilde{u}_\infty \in \mathcal{D}_a^{1,p}(\mathbb{R}^N\backslash\{|y|=0\})$ 
such that 
$\tilde{u}_n \rightharpoonup \tilde{u}_\infty$ 
weakly in 
$\mathcal{D}_a^{1,p}(\mathbb{R}^N\backslash\{|y|=0\})$ as $n \to \infty$, possibly after passage to a subsequence. Furthermore, 
$\tilde{u}_\infty > 0$ in
$\mathbb{R^N}\backslash \{|y|=0\}$ 
and $\tilde{u}_\infty$ is a weak solution to 
problem~\eqref{problema}.
\end{prop}
\begin{proof}
Let $(u_n)_{n\in\mathbb{N}} 
\subset \mathcal{D}_a^{1,p}(\mathbb{R}^N\backslash\{|y|=0\})$ 
be a Palais-Smale for the functional
$\varphi$ at the level $d \in (0,d_*)$ 
and let 
$(\tilde{u}_n)_{n\in\mathbb{N}} 
\subset \mathcal{D}_a^{1,p}
(\mathbb{R}^N\backslash\{|y|=0\})$ 
be the sequence defined in~\eqref{eq:uhtilde}.

The sequence $(\tilde{u}_n)_{n\in\mathbb{N}}$ 
is bounded in $\mathcal{D}_a^{1,p}(\mathbb{R}^N\backslash\{|y|=0\})$. Indeed, since
$p < p^*(a,c) < p^*(a,b)$, we have
\begin{align*}
c_1 + c_2\|\tilde{u}_n\| &\geqslant \varphi(\tilde{u}_n) - \frac{1}{p^*(a,c)}\left\langle \varphi'(\tilde{u}_n),\tilde{u}_n\right\rangle\\
&=\left(\frac{1}{p}-\frac{1}{p^*(a,c)}\right)\|\tilde{u}_n\|^p-\left(\frac{1}{p^*(a,b)}
-\frac{1}{p^*(a,c)}\right)\|(\tilde{u}_n)_+\|_{L_b^{p^*(a,b)}}^{p^*(a,b)}\\
&\geqslant \left(\frac{1}{p}-\frac{1}{p^*(a,c)}\right)\|\tilde{u}_n\|^p,
\end{align*}
and our claim follows.
Thus, there exists a function
$\tilde{u}_\infty \in \mathcal{D}_a^{1,p}(\mathbb{R}^N\backslash\{|y|=0\})$ such that, up to passage to a subsequence still denoted in the same way, $\tilde{u}_n \rightharpoonup \tilde{u}_\infty$ weakly in 
$\mathcal{D}_a^{1,p}(\mathbb{R}^N\backslash\{|y|=0\})$ as $n \to \infty$. Additionally, we have that
$(|\nabla \tilde{u}_n|^{p-2}\nabla \tilde{u}_n)_{n\in\mathbb{N}} 
\subset (L^{p'}_a(\mathbb{R}^N))^N$, 
$(|\tilde{u}_n|^{p-2} \tilde{u}_n)_{n\in\mathbb{N}} 
\subset L^{p'}_{a+1}(\mathbb{R}^N)$, 
$(|\tilde{u}_n|^{p^*(a,b)-2} \tilde{u}_n)_{n\in\mathbb{N}} \subset L^{(p^*(a,b))'}_{b}(\mathbb{R}^N)$ and
$(|\tilde{u}_n|^{p^*(a,c)-2} \tilde{u}_n)_{n\in\mathbb{N}} \subset L^{(p^*(a,c))'}_{c}(\mathbb{R}^N)$
are bounded sequences in these spaces.
Hence, we have the following convergences:
\begin{enumerate}[label={\upshape(\arabic*)}, 
align=left, widest=3, leftmargin=*]
\item
$(|\nabla \tilde{u}_n|^{p-2}\nabla \tilde{u}_n)_{n\in\mathbb{N}} \rightharpoonup T$ 
weakly in $(L^{p'}_a(\mathbb{R}^N))^N$, 
for some $T \in (L^{p'}_a(\mathbb{R}^N))^N$;
\item
$(|\tilde{u}_n|^{p-2} \tilde{u}_n)_{n\in\mathbb{N}} \rightharpoonup |\tilde{u}_\infty|^{p-2} \tilde{u}_\infty$ 
weakly in $L^{p'}_{a+1}(\mathbb{R}^N)$;
\item
$(|\tilde{u}_n|^{p^*(a,b)-2} \tilde{u}_n)_{n\in\mathbb{N}} \rightharpoonup |\tilde{u}_\infty|^{p^*(a,b)-2} \tilde{u}_\infty$ 
weakly in  $L^{(p^*(a,b))'}_{b}(\mathbb{R}^N)$;
\item
$(|\tilde{u}_n|^{p^*(a,c)-2} \tilde{u}_n)_{n\in\mathbb{N}} \rightharpoonup |\tilde{u}_\infty|^{p^*(a,c)-2} \tilde{u}_\infty$ 
weakly in $L^{(p^*(a,c))'}_{c}(\mathbb{R}^N)$.
\end{enumerate}

Now we show that this weak limit 
$\tilde{u}_\infty$ is not identically zero. We argue by contradiction and we suppose that 
$\tilde{u}_\infty \equiv 0$. 
Applying Proposition~\ref{convergenciapara0}, by equality~\eqref{25} the first case is excluded;  
and since $0<\epsilon < \epsilon_0/2$,
again by equality~\eqref{25} we have
\begin{align*}
\epsilon_0 \leqslant \int_{B_{t_n}(0)}\frac{(u_n)_+^{p^*(a,b)}}{|y|^{bp^*(a,b)}}\, dz = \epsilon < \frac{\epsilon_0}{2},
\end{align*}
which is a contradiction. Hence, 
we necessarily have $\tilde{u}_\infty \not\equiv 0$.

In what follows we show that 
$\tilde{u}_\infty$ is a weak solution to 
problem~\eqref{problemapucciservadei}.
From the convergence 
$\varphi'(\tilde{u}_n) \to 0$ in $\left(\mathcal{D}_a^{1,p}(\mathbb{R}^N\backslash\{|y|=0\})\right)^*$ 
as $n \to +\infty$, we obtain
\begin{align}\label{4.6xuanwang2010}
o(1) 
= \left\langle \varphi'(\tilde{u}_n),v \right\rangle 
&
= \int_{\mathbb{R}^N}\frac{|\nabla \tilde{u}_n|^{p-2}}{|y|^{ap}}\left\langle\nabla \tilde{u}_n, \nabla v \right\rangle \, dz
- \mu \int_{\mathbb{R}^N}\frac{|\tilde{u}_n|^{p-2}\tilde{u}_n v}{|y|^{p(a+1)}}\, dz \nonumber\\
& \qquad - \int_{\mathbb{R}^N}\frac{(\tilde{u}_n)_+^{p^*(a,b)-2}\tilde{u}_n v}{|y|^{bp^*(a,b)}}\, dz
- \int_{\mathbb{R}^N}\frac{(\tilde{u}_n)_+^{p^*(a,c)-2}\tilde{u}_n v}{|y|^{cp^*(a,c)}}\, dz,
\end{align}
for every $v \in \mathcal{D}_a^{1,p}(\mathbb{R}^N\backslash\{|y|=0\})$. 
And from the weak convergences previously determined, it follows that
\begin{align*}
\mu \int_{\mathbb{R}^N}\frac{|\tilde{u}_n|^{p-2}\tilde{u}_n v}{|y|^{p(a+1)}}\, dz
&\to \mu \int_{\mathbb{R}^N}\frac{|\tilde{u}_\infty|^{p-2}\tilde{u}_\infty v}{|y|^{p(a+1)}}\, dz\\
\int_{\mathbb{R}^N}\frac{(\tilde{u}_n)_+^{p^*(a,b)-2}\tilde{u}_n v}{|y|^{bp^*(a,b)}}\, dz
&\to \int_{\mathbb{R}^N}\frac{(\tilde{u}_\infty)_+^{p^*(a,b)-2}\tilde{u}_\infty v}{|y|^{bp^*(a,b)}}\, dz\\
\int_{\mathbb{R}^N}\frac{(\tilde{u}_n)_+^{p^*(a,c)-2}\tilde{u}_n v}{|y|^{cp^*(a,c)}}\, dz
&\to \int_{\mathbb{R}^N}\frac{(\tilde{u}_\infty)_+^{p^*(a,c)-2}\tilde{u}_\infty v}{|y|^{cp^*(a,c)}}\, dz
\end{align*}
as $n \to \infty$ 
for every function 
$v \in \mathcal{D}_a^{1,p}(\mathbb{R}^N\backslash\{|y|=0\})$.
It remains to show that 
\begin{align}\label{4.7xuanwang2010}
\int_{\mathbb{R}^N}\frac{|\nabla \tilde{u}_n|^{p-2}}{|y|^{ap}}\left\langle\nabla \tilde{u}_n, \nabla v \right\rangle \, dz
\to 
\int_{\mathbb{R}^N}\frac{|\nabla \tilde{u}_\infty|^{p-2}}{|y|^{ap}}\left\langle\nabla \tilde{u}_\infty, \nabla v \right\rangle \, dz,
\end{align}
as $n \to +\infty$ 
for every function 
$v \in \mathcal{D}_a^{1,p}(\mathbb{R}^N\backslash\{|y|=0\})$. 
To achieve this goal, we have to show the almost everywhere convergence of the sequence of the gradients of our original sequence. This follows from 
Lemmas~\ref{benmouloudlema2} and~\ref{benmouloudlema3}. 
\end{proof}

The next two lemmas are crucial to prove the almost everywhere convergence of the sequence of the gradients of our original Palais-Smale sequence. In their proofs, we adapted the ideas from the paper by 
Benmouloud, Echarghaoui, and Sba\"{\i}~\cite{MR2524370}
and by Xuan and Wang~\cite{MR2606810}.

\begin{lema}\label{benmouloudlema2}
Suppose that the hypotheses of 
Proposition~\ref{fprclaim4.3} are valid 
and let $\eta \in C_0^\infty(\mathbb{R}^N)$ be a 
cut off function. Then, up to passage to a subsequence, we have
\begin{align*}
\limsup_{n \to \infty} 
\int_{\mathbb{R}^N}
\eta\bigg|\frac{\big\langle|\nabla \tilde{u}_n|^{p-2}\nabla \tilde{u}_n - |\nabla \tilde{u}_\infty|^{p-2}\nabla \tilde{u}_\infty, 
\nabla(\tilde{u}_n-\tilde{u}_\infty)\big\rangle}{|y|^{ap}}\bigg|^\frac{1}{p} \, dz
= 0.
\end{align*}
\end{lema}
\begin{proof}
Let us cover the space $\mathbb{R}^N$ 
by a sequence of balls
$(B_{R_j})_{j\in \mathbb{N}} \subset \mathbb{R}^N$ centered at the origin, where the sequence of radii $(R_j)_{j\in \mathbb{N}} \subset \mathbb{R}$ is crescent. Let us also fix 
$j\in \mathbb{N}$ and consider the cut off function
$\eta \in C_0^\infty (B_{R_{j+1}})$ 
such that $0 \leqslant \eta \leqslant 1$ 
and also $\eta (z) = 1$ for every $z \in B_{R_{j}}$.

We recall that
for every
$X, Y \in \mathbb{R}^N$ 
it is valid the inequality
$\langle |X|^{p-2}X-|Y|^{p-2}Y,X-Y\rangle \geqslant 0$, 
which follows from the monotonicity of the function
$t \mapsto |t|^{p-2}t$. Moreover, the equality is valid if, and only if, $X = Y$.  Hence,
\begin{align}\label{4.8xuanwang2010}
\big\langle|\nabla \tilde{u}_n|^{p-2}\nabla \tilde{u}_n - |\nabla \tilde{u}_\infty|^{p-2}\nabla \tilde{u}_\infty, \nabla(\tilde{u}_n-\tilde{u}_\infty)\big\rangle
\geqslant 0.
\end{align}

Given 
$0 < \epsilon < 1$, $n \in \mathbb{N}$ 
and $m \in \mathbb{N}$ with $ m\geqslant 1$, 
we define the subsets
\begin{align*}
E_m &\equiv \left\{z \in \mathbb{R}^N \colon 
|\tilde{u}_\infty(z)| > m \right\}, \\
A^+_{n,m,\epsilon} &\equiv \left\{z \in \mathbb{R}^N
\colon |\tilde{u}_\infty(z)| \leqslant m \textrm{ and } 
|\tilde{u}_n(z)- \tilde{u}_\infty(z)| \geqslant \epsilon  \right\},\\
A^-_{n,m,\epsilon} &\equiv \left\{z \in \mathbb{R}^N
\colon |\tilde{u}_\infty(z)| \leqslant m \textrm{ and } 
|\tilde{u}_n(z)- \tilde{u}_\infty(z)| < \epsilon  \right\}.\\
\end{align*}
In this way, we have
\begin{align}\label{benmouloud2}
&\int_{\mathbb{R}^N}
\eta\bigg|\frac{\big\langle|\nabla \tilde{u}_n|^{p-2}\nabla \tilde{u}_n - |\nabla \tilde{u}_\infty|^{p-2}\nabla \tilde{u}_\infty, 
\nabla(\tilde{u}_n-\tilde{u}_\infty)\big\rangle}{|y|^{ap}}\bigg|^\frac{1}{p} \, dz\nonumber\\
&\qquad = \int_{E_m}
\eta\bigg|\frac{\big\langle|\nabla \tilde{u}_n|^{p-2}\nabla \tilde{u}_n - |\nabla \tilde{u}_\infty|^{p-2}\nabla \tilde{u}_\infty, 
\nabla(\tilde{u}_n-\tilde{u}_\infty)\big\rangle}{|y|^{ap}}\bigg|^\frac{1}{p} \, dz\nonumber\\
&\qquad \qquad +\int_{A^+_{n,m,\epsilon}}
\eta\bigg|\frac{\big\langle|\nabla \tilde{u}_n|^{p-2}\nabla \tilde{u}_n - |\nabla \tilde{u}_\infty|^{p-2}\nabla \tilde{u}_\infty, 
\nabla(\tilde{u}_n-\tilde{u}_\infty)\big\rangle}{|y|^{ap}}\bigg|^\frac{1}{p} \, dz\nonumber\\
&\qquad \qquad + \int_{A^-_{n,m,\epsilon}}
\eta\bigg|\frac{\big\langle|\nabla \tilde{u}_n|^{p-2}\nabla \tilde{u}_n - |\nabla \tilde{u}_\infty|^{p-2}\nabla \tilde{u}_\infty, 
\nabla(\tilde{u}_n-\tilde{u}_\infty)\big\rangle}{|y|^{ap}}\bigg|^\frac{1}{p} \, dz.
\end{align}

Hereafter, we use the letter $C$ to denote positive constants, which are independent  from $n$, from $m$, and from $\epsilon$; besides, these constants can change from one passage to the other in the computations. 

We remark that the sequence $(w_n)_{n\in\mathbb{N}} \subset L^1(\mathbb{R}^N)$ given by
$$
w_n \equiv \frac{\big\langle|\nabla \tilde{u}_n|^{p-2}\nabla \tilde{u}_n - |\nabla \tilde{u}_\infty|^{p-2}\nabla \tilde{u}_\infty, 
\nabla(\tilde{u}_n-\tilde{u}_\infty)\big\rangle}{|y|^{ap}}
$$
is bounded. 
Thus, by H\"{o}lder's inequality we obtain
\begin{align}\label{benmouloud3}
\int_{E_m}
\eta\bigg|\frac{\big\langle|\nabla \tilde{u}_n|^{p-2}\nabla \tilde{u}_n - |\nabla \tilde{u}_\infty|^{p-2}\nabla \tilde{u}_\infty, 
\nabla(\tilde{u}_n-\tilde{u}_\infty)\big\rangle}{|y|^{ap}}\bigg|^\frac{1}{p} \, dz
\leqslant C|E_m|^\frac{p-1}{p}.
\end{align}
Moreover, denoting by
\( \chi_A \colon \mathbb{R}\sp{N} \to \mathbb{R}\) the 
characteristic function on the subset 
\( A \subset \mathbb{R}\sp{N} \),
H\"{o}lder's and Maz'ya's inequalities imply that
\begin{align*}
|E_m| &\leqslant \frac{1}{m} \int_{E_m} \chi_{E_m} \frac{\tilde{u}_\infty(z)}{|y|^{a+1}} \, dz
\leqslant \frac{1}{m} C |E_m|^\frac{p-1}{p},
\end{align*}
that is,
\( |E_m| \leqslant C/m^p \).
Combining this result with inequality~\eqref{benmouloud3} we obtain an estimate 
for the first term on the right-hand side of 
equality~\eqref{benmouloud2}, 
namely,
\begin{align}\label{benmouloud4}
\int_{E_m}
\eta\bigg|\frac{\big\langle|\nabla \tilde{u}_n|^{p-2}\nabla \tilde{u}_n - |\nabla \tilde{u}_\infty|^{p-2}\nabla \tilde{u}_\infty, 
\nabla(\tilde{u}_n-\tilde{u}_\infty)\big\rangle}{|y|^{ap}}\bigg|^\frac{1}{p} \, dz
\leqslant \frac{C}{m^{p-1}}.
\end{align}

Applying H\"{o}lder's inequality to the second term of equality~\eqref{benmouloud2}, we obtain
\begin{align}\label{benmouloud4a}
&\int_{A^+_{n,m,\epsilon}}
\eta\bigg|\frac{\big\langle|\nabla \tilde{u}_n|^{p-2}\nabla \tilde{u}_n - |\nabla \tilde{u}_\infty|^{p-2}\nabla \tilde{u}_\infty, 
\nabla(\tilde{u}_n-\tilde{u}_\infty)\big\rangle}{|y|^{ap}}\bigg|^\frac{1}{p} \, dz\nonumber\\
& \qquad \leqslant C \left|\left\{z \in B_{R_{j+1}} \big| |\tilde{u}_n(z) - \tilde{u}_\infty(z)| \geqslant \epsilon \right\}\right|^\frac{p-1}{p}.
\end{align}
Moreover, we have
$\chi_{\{z \in B_{R_{j+1}} \colon |\tilde{u}_n(z) - \tilde{u}_\infty(z)| \geqslant \epsilon\}} \leqslant \chi_{B_{R_{j+1}}}$
and, 
$\chi_{\{z \in B_{R_{j+1}} \colon |\tilde{u}_n(z) - \tilde{u}_\infty(z)| \geqslant \epsilon\}} \to 0$ 
a.\@ e.\@ $\mathbb{R}^N$ 
as $n \to \infty$.
Since $B_{R_{j+1}}(0) \subset \mathbb{R}\sp{N}$ is bounded for every \( n \in \mathbb{N} \), by the 
Lebesgue dominated convergence theorem it follows that
$
\left|\left\{z \in B_{R_{j+1}} \colon |\tilde{u}_n(z) - \tilde{u}_\infty(z)| \geqslant \epsilon \right\}\right| \to 0
$
as $n\to \infty$. Hence, there exists
$n(\epsilon) \in \mathbb{N}$ independent from $m$, 
such that 
$
\left|\left\{z \in B_{R_{j+1}} \colon |\tilde{u}_n(z) - \tilde{u}_\infty(z)| \geqslant \epsilon \right\}\right| < \epsilon
$ 
for every \( n \in \mathbb{N} \) such that 
$n > n(\epsilon)$.
Inequality~\eqref{benmouloud4a} together with previous one 
imply that the second term on the right-hand side of equality~\eqref{benmouloud2} 
is such that 
\begin{align}\label{benmouloud5}
\int_{\mathbb{R}^N} \chi_{A^+_{n,m,\epsilon}}
\eta\bigg|\frac{\big\langle|\nabla \tilde{u}_n|^{p-2}\nabla \tilde{u}_n - |\nabla \tilde{u}_\infty|^{p-2}\nabla \tilde{u}_\infty, 
\nabla(\tilde{u}_n-\tilde{u}_\infty)\big\rangle}{|y|^{ap}}\bigg|^\frac{1}{p} \, dz
\leqslant C\epsilon^\frac{p-1}{p}
\end{align}
for every \( n \in \mathbb{N} \) such that 
$n \geqslant n(\epsilon)$.

Finally, the third term on the right-hand side of the
equality~\eqref{benmouloud2} can be estimated using H\"{o}lder's inequality. Indeed, 
\begin{align}\label{benmouloud6}
\int_{A^-_{n,m,\epsilon}}
&\eta\bigg|\frac{\big\langle|\nabla \tilde{u}_n|^{p-2}\nabla \tilde{u}_n - |\nabla \tilde{u}_\infty|^{p-2}\nabla \tilde{u}_\infty, 
\nabla(\tilde{u}_n-\tilde{u}_\infty)\big\rangle}{|y|^{ap}}\bigg|^\frac{1}{p} \, dz\nonumber\\
& \qquad = \int_{B_{R_{j+1}}} \chi_{A^-_{n,m,\epsilon}} \eta^\frac{p-1}{p}
\bigg|\frac{\eta\big\langle|\nabla \tilde{u}_n|^{p-2}\nabla \tilde{u}_n - |\nabla \tilde{u}_\infty|^{p-2}\nabla \tilde{u}_\infty, 
\nabla(\tilde{u}_n-\tilde{u}_\infty)\big\rangle}{|y|^{ap}}\bigg|^\frac{1}{p} \, dz\nonumber\\
& \qquad \leqslant \Bigg(\int_{B_{R_{j+1}}} 
\chi_{A^-_{n,m,\epsilon}} \eta \, dz \Bigg)^\frac{p-1}{p}\nonumber\\
& \qquad \qquad \times \Bigg(\int_{B_{R_{j+1}}} 
\chi_{A^-_{n,m,\epsilon}}
\frac{\eta\big\langle|\nabla \tilde{u}_n|^{p-2}\nabla \tilde{u}_n
 - |\nabla \tilde{u}_\infty|^{p-2}\nabla \tilde{u}_\infty,
\nabla(\tilde{u}_n-\tilde{u}_\infty)\big\rangle}{|y|^{ap}} \, dz \Bigg)^\frac{1}{p} \nonumber \\
& \qquad \leqslant C 
\left\{ \splitfrac{\displaystyle\int_{B_{R_{j+1}}} 
\chi_{A^-_{n,m,\epsilon}}
\frac{\eta\big\langle|\nabla \tilde{u}_n|^{p-2}\nabla \tilde{u}_n, 
\nabla(\tilde{u}_n-\tilde{u}_\infty)\big\rangle}{|y|^{ap}} \, dz }
{ - \displaystyle \int_{B_{R_{j+1}}} \chi_{A^-_{n,m,\epsilon}}
\frac{\eta\big\langle|\nabla \tilde{u}_\infty|^{p-2}\nabla \tilde{u}_\infty, 
\nabla(\tilde{u}_n-\tilde{u}_\infty)\big\rangle}{|y|^{ap}} \, dz } \right\}^\frac{1}{p}.
\end{align}

Now we define the function 
$\psi_\epsilon \colon \mathbb{R} \to \mathbb{R}$ by
\begin{align*}
\psi_\epsilon(\sigma) & \equiv
\begin{cases}
\sigma, &\textrm{if $|\sigma|\leq\epsilon$;}\\
\epsilon \operatorname{sign}(\sigma), &\textrm{if $ |\sigma| > \epsilon$.}
\end{cases}
\end{align*}
To estimate the integral on the right-hand side of inequality~\eqref{benmouloud6} we note that the set
$A^-_{n,m,\epsilon}$ can be written as
\begin{align*}
A^-_{n,m,\epsilon} 
& = \left\{z \in B_{R_{j+1}} \colon |\tilde{u}_\infty(z)| \leqslant m \textrm{ and } 
|\tilde{u}_n(z)- \psi_m(\tilde{u}_\infty(z))| < \epsilon  \right\}.
\end{align*}
In this way, we obtain
\begin{align}
\label{benmouloud3a1}
&\int_{B_{R_{j+1}}} \chi_{A^-_{n,m,\epsilon}}
\frac{\eta\big\langle|\nabla \tilde{u}_n|^{p-2}\nabla \tilde{u}_n, 
\nabla(\tilde{u}_n-\tilde{u}_\infty)\big\rangle}{|y|^{ap}} \, dz \nonumber \\
& \qquad = \int_{B_{R_{j+1}}} \chi_{A^-_{n,m,\epsilon}}
\frac{\eta|\nabla \tilde{u}_n|^{p}}{|y|^{ap}} \, dz
 - \int_{B_{R_{j+1}}} \chi_{A^-_{n,m,\epsilon}}
\frac{\eta\big\langle|\nabla \tilde{u}_n|^{p-2}\nabla \tilde{u}_n,\nabla \tilde{u}_\infty\big\rangle}{|y|^{ap}} \, dz \nonumber \\
& \qquad \leqslant \int_{B_{R_{j+1}}} \chi_{\left\{z \in B_{R_{j+1}} \colon |\tilde{u}_n(z)- \psi_m(\tilde{u}_\infty(z))| < \epsilon  \right\}}
\frac{\eta|\nabla \tilde{u}_n|^{p}}{|y|^{ap}} \, dz \nonumber \\
& \qquad \qquad 
- \int_{B_{R_{j+1}}} 
\chi_{\left\{z \in B_{R_{j+1}} \colon |\tilde{u}_n(z)- \psi_m(\tilde{u}_\infty(z))| < \epsilon  \right\}}
\frac{\eta\big\langle|\nabla \tilde{u}_n|^{p-2}\nabla \tilde{u}_n,\nabla \tilde{u}_\infty\big\rangle}{|y|^{ap}} \, dz \nonumber \\
& \qquad = \int_{B_{R_{j+1}}} 
\frac{\eta\big\langle|\nabla \tilde{u}_n|^{p-2}\nabla \tilde{u}_n, 
\nabla(\psi_\epsilon(\tilde{u}_n-\psi_m(\tilde{u}_\infty)))\big\rangle}{|y|^{ap}} \, dz \nonumber \\
& \qquad \leqslant \int_{B_{R_{j+1}}} 
\frac{\big\langle|\nabla \tilde{u}_n|^{p-2}\nabla \tilde{u}_n, 
\nabla(\eta\psi_\epsilon(\tilde{u}_n-\psi_m(\tilde{u}_\infty)))\big\rangle}{|y|^{ap}} \, dz \nonumber \\
& \qquad \qquad + \int_{B_{R_{j+1}}} 
\frac{\big\langle|\nabla \tilde{u}_n|^{p-2}\nabla \tilde{u}_n, 
\psi_\epsilon(\tilde{u}_n-\psi_m(\tilde{u}_\infty))\nabla\eta\big\rangle}{|y|^{ap}} \, dz \nonumber \\
& \qquad \leqslant \int_{B_{R_{j+1}}} 
\frac{\big\langle|\nabla \tilde{u}_n|^{p-2}\nabla \tilde{u}_n, 
\nabla(\eta\psi_\epsilon(\tilde{u}_n-\psi_m(\tilde{u}_\infty)))\big\rangle}{|y|^{ap}} \, dz
+ C\epsilon.
\end{align}

Now we use the function
$\eta\psi_\epsilon(\tilde{u}_n-\psi_m(\tilde{u}_\infty)) \in \mathcal{D}_a^{1,p}(\mathbb{R}^N\backslash\{|y|=0\})$ as a test function in the definition of weak solution to the problem~\eqref{problemapucciservadei} to estimate the integral in~\eqref{benmouloud3a1}. Thus, we obtain
\begin{align}
\label{benmouloud3a2}
{} &\int_{\mathbb{R}^N} \frac{\big\langle|\nabla \tilde{u}_n|^{p-2}\nabla \tilde{u}_n, 
\nabla (\eta \psi_\epsilon (\tilde{u}_n - \psi_m (\tilde{u}_\infty)))\big\rangle}{|y|^{ap}} \, dz \nonumber \\
& \qquad = \left\langle \varphi'(\tilde{u}_n), \eta \psi_\epsilon (\tilde{u}_n - \psi_m (\tilde{u}_\infty))\right\rangle
+ \mu \int_{\mathbb{R}^N} \frac{|\tilde{u}_n|^{p-2} \tilde{u}_n 
(\eta \psi_\epsilon (\tilde{u}_n - \psi_m (\tilde{u}_\infty)))}{|y|^{p(a+1)}} \, dz \nonumber \\
& \qquad \qquad  + \int_{\mathbb{R}^N} \frac{(\tilde{u}_n)_+^{p^*(a,b)-2} (\tilde{u}_n)_+
(\eta \psi_\epsilon (\tilde{u}_n - \psi_m (\tilde{u}_\infty)))}{|y|^{bp^*(a,b)}} \, dz \nonumber \\
& \qquad \qquad + \int_{\mathbb{R}^N} \frac{(\tilde{u}_n)_+^{p^*(a,c)-2} (\tilde{u}_n)_+
(\eta \psi_\epsilon (\tilde{u}_n - \psi_m (\tilde{u}_\infty)))}{|y|^{cp^*(a,c)}} \, dz.
\end{align}
Since the sequences 
$(\|\tilde{u}_n\|_{L_{a+1}^{p}})_{n\in \mathbb{N}} \subset \mathbb{R}$, $(\|\tilde{u}_n\|_{L_b^{p^*(a,b)}})_{n\in \mathbb{N}} \subset \mathbb{R}$ 
and $(\|\tilde{u}_n\|_{L_c^{p^*(a,c)}})_{n\in \mathbb{N}} \subset \mathbb{R}$ 
are bounded, regardless of $n \in \mathbb{N}$, 
it follows from inequality~\eqref{benmouloud3a1} and equality~\eqref{benmouloud3a2} that
\begin{align}\label{benmouloud8}
\int_{B_{R_{j+1}}} \chi_{A^-_{n,m,\epsilon}}
\frac{\eta\big\langle|\nabla \tilde{u}_n|^{p-2}\nabla \tilde{u}_n, 
\nabla(\tilde{u}_n-\tilde{u}_\infty)\big\rangle}{|y|^{ap}} \, dz
& \leqslant \left\langle \varphi'(\tilde{u}_n), \eta \psi_\epsilon (\tilde{u}_n - \psi_m (\tilde{u}_\infty))\right\rangle
+ C\epsilon.
\end{align}

Regarding the second integral on the right-hand side of inequality~\eqref{benmouloud6}, we note that it can be written in the form
\begin{align}\label{benmouloud7}
&\int_{B_{R_{j+1}}} \chi_{A^-_{n,m,\epsilon}}
\frac{\eta\big\langle|\nabla \tilde{u}_\infty|^{p-2}\nabla \tilde{u}_\infty, 
\nabla(\tilde{u}_n-\tilde{u}_\infty)\big\rangle}{|y|^{ap}} \, dz\nonumber\\
& \qquad = \int_{B_{R_{j+1}}} \chi_{\{z \in \mathbb{R}\sp{N} \colon |\tilde{u}_\infty(z)| \leqslant m\}}
\frac{\eta\big\langle|\nabla \tilde{u}_\infty|^{p-2}\nabla \tilde{u}_\infty, 
\nabla(\psi_\epsilon(\tilde{u}_n-\tilde{u}_\infty))\big\rangle}{|y|^{ap}} \, dz.
\end{align}

Now we combine inequalities~\eqref{benmouloud4},~\eqref{benmouloud5},~\eqref{benmouloud8}, 
and~\eqref{benmouloud7} with equality~\eqref{benmouloud2} 
and we deduce that, for every $n > n(\epsilon)$ 
and for every $m \in \mathbb{N}^*$, 
it is valid the inequality
\begin{align}
\label{benmouloud8a}
{} & \int_{\mathbb{R}^N}
\frac{\eta\big\langle|\nabla \tilde{u}_n|^{p-2}\nabla \tilde{u}_n - |\nabla \tilde{u}_\infty|^{p-2}\nabla \tilde{u}_\infty, 
\nabla(\tilde{u}_n-\tilde{u}_\infty)\big\rangle}{|y|^{ap}} \, dz \nonumber \\
& \qquad \leqslant \frac{C}{m^{p-1}} + C\epsilon^\frac{p-1}{p} 
+C\left\{
\splitfrac{-\displaystyle\int_{B_{R_{j+1}}} 
\chi_{\{|\tilde{u}_\infty| \leqslant m\}}
\frac{\eta\big\langle|\nabla \tilde{u}_\infty|^{p-2}\nabla \tilde{u}_\infty, 
\nabla(\psi_\epsilon(\tilde{u}_n-\tilde{u}_\infty))\big\rangle}{|y|^{ap}} \, dz }
{ + \displaystyle\left\langle \varphi'(\tilde{u}_n), \eta \psi_\epsilon (\tilde{u}_n - \psi_m (\tilde{u}_\infty))\right\rangle
+ C\epsilon } \right\}^\frac{1}{p}.
\end{align}

For $n \geqslant n(\epsilon)$ fixed, the sequence
$(\eta \psi_\epsilon (\tilde{u}_n 
- \psi_m (\tilde{u}_\infty)))\sb{n \in \mathbb{N}} \subset 
\mathcal{D}_a^{1,p}(\mathbb{R}^N\backslash\{|y|=0\})$ is bounded.
By the reflexivity of the Sobolev space
$\mathcal{D}_a^{1,p}(\mathbb{R}^N\backslash\{|y|=0\})$, up to passage to a subsequence always denoted in the same way, there exists a function 
\( \eta \psi_\epsilon (\tilde{u}_n - \tilde{u}_\infty) \in 
\mathcal{D}_a^{1,p}(\mathbb{R}^N\backslash\{|y|=0\}) \)
such that
$\eta \psi_\epsilon (\tilde{u}_n - \psi_m (\tilde{u}_\infty)) \rightharpoonup \eta \psi_\epsilon (\tilde{u}_n - \tilde{u}_\infty)$ weakly in
$\mathcal{D}_a^{1,p}(\mathbb{R}^N\backslash\{|y|=0\})$ as $m \to \infty$. 
Passing to the limit as 
$m \to \infty$ in inequality~\eqref{benmouloud8a} it follows that
\begin{align}\label{benmouloud8b}
&\int_{\mathbb{R}^N}
\frac{\eta\big\langle|\nabla \tilde{u}_n|^{p-2}\nabla \tilde{u}_n - |\nabla \tilde{u}_\infty|^{p-2}\nabla \tilde{u}_\infty, 
\nabla(\tilde{u}_n-\tilde{u}_\infty)\big\rangle}{|y|^{ap}} \, dz\nonumber\\
& \qquad \leqslant C\epsilon^\frac{p-1}{p}
+C\left\{ \splitfrac{- \displaystyle\int_{\mathbb{R}^N} 
\frac{\eta\big\langle|\nabla \tilde{u}_\infty|^{p-2}\nabla \tilde{u}_\infty, 
\nabla(\psi_\epsilon(\tilde{u}_n-\tilde{u}_\infty))\big\rangle}{|y|^{ap}} \, dz }
{ + \left\langle \varphi'(\tilde{u}_n), \eta \psi_\epsilon (\tilde{u}_n - \tilde{u}_\infty)\right\rangle
+ C\epsilon } \right\}^\frac{1}{p}.
\end{align}

On the other hand, by the definition of weak solution and by the fact that the space 
$C_0^\infty(B_{R_{j+1}})$ is dense in $L^{p'}_{a+1}(B_{R_{j+1}})$, we deduce that 
\begin{align}
\nabla \psi_\epsilon (\tilde{u}_n - \tilde{u}_\infty) \rightharpoonup 0 \quad  \textrm{weakly in } L^{p}_{a+1}(B_{R_{j+1}})
\end{align}
as $n \to +\infty$. Since the sequence
$( \psi_\epsilon (\tilde{u}_n - \tilde{u}_\infty) )\sb{n \in \mathbb{N}} \subset 
\mathcal{D}_a^{1,p}(\mathbb{R}^N\backslash\{|y|=0\})
$ is also bounded,
up to the passage to a subsequence we have
\begin{align*}
\nabla \psi_\epsilon (\tilde{u}_n - \tilde{u}_\infty) \rightharpoonup 0 \quad \textrm{weakly in } \mathcal{D}_a^{1,p}(\mathbb{R}^N\backslash\{|y|=0\})
\end{align*}
as $n \to \infty$. 
Passing to the limit superior as $n \to \infty$ 
in inequality~\eqref{benmouloud8b}, it follows that
\begin{align*}
\limsup_{n \to \infty} 
\int_{\mathbb{R}^N}
\bigg|\frac{\eta\big\langle|\nabla \tilde{u}_n|^{p-2}\nabla \tilde{u}_n - |\nabla \tilde{u}_\infty|^{p-2}\nabla \tilde{u}_\infty, 
\nabla(\tilde{u}_n-\tilde{u}_\infty)\big\rangle}{|y|^{ap}}\bigg|^\frac{1}{p} \, dz
\leqslant C \big(\epsilon^\frac{p-1}{p} + \epsilon^\frac{1}{p}\big).
\end{align*}
And since $0<\epsilon <1$ is arbitrary, using the Cantor's diagonal argument we conclude that
$$
\limsup_{n \to \infty} 
\int_{\mathbb{R}^N}
\bigg|\frac{\eta\big\langle|\nabla \tilde{u}_n|^{p-2}\nabla \tilde{u}_n - |\nabla \tilde{u}_\infty|^{p-2}\nabla \tilde{u}_\infty, 
\nabla(\tilde{u}_n-\tilde{u}_\infty)\big\rangle}{|y|^{ap}}\bigg|^\frac{1}{p} \, dz
= 0.
$$
The lemma is proved.
\end{proof}

\begin{lema}\label{benmouloudlema3}
Suppose that the hypotheses of
Proposition~\ref{fprclaim4.3} are valid. 
Then, 
as $n \to \infty$ we have
\(\nabla \tilde{u}_n \to \nabla \tilde{u}_\infty \) 
a.\@ e.\@ in $\mathbb{R}^N$.
\end{lema}
\begin{proof}
As in the proof of Lemma~\ref{benmouloudlema2},
let us cover the space $\mathbb{R}^N$ 
by a sequence of balls
$(B_{R_j})_{j\in \mathbb{N}} \subset \mathbb{R}^N$ centered at the origin, where the sequence of radii $(R_j)_{j\in \mathbb{N}} \subset \mathbb{R}$ is crescent. Let us also fix 
$j\in \mathbb{N}$ and consider the cut off function
$\eta \in C_0^\infty (B_{R_{j+1}})$ 
such that $0 \leqslant \eta \leqslant 1$ 
and also $\eta (z) = 1$ for every $z \in B_{R_{j}}$.

We recall that 
there exists a constant $C > 0$ such that 
$
\langle |X|^{s-2}X - |Y|^{s-2}Y, X-Y \rangle 
\geqslant C |X-Y|^s
$ 
for every $X,Y \in \mathbb{R}^N$ with $s \geqslant 2$.

Using this inequality and Lemma~\ref{benmouloudlema2} it follows that
\begin{align*}
\limsup_{n \to \infty} 
\int_{B_{R_j}}
\frac{|\nabla \tilde{u}_n - \nabla \tilde{u}_\infty|^{p}}{|y|^{ap}} \, dz
= 0,
\end{align*}
that is, 
\(|\nabla \tilde{u}_n|^p 
\to  |\nabla \tilde{u}_\infty|^{p} \)
strongly in \( L^1(B_{R_j}) \)
as $n \to \infty$
and up to a subsequence, still denoted in the same way, we deduce that
\(
\nabla \tilde{u}_n \to  \nabla \tilde{u}_\infty
\textrm{a.\@ e.\@ in } B_{R_j}
\) as $n \to \infty$.
Since $j \in \mathbb{N}$ is arbitrary, using Cantor's diagonal argument we conclude that
\(
\nabla \tilde{u}_n \to  \nabla \tilde{u}_\infty 
 \textrm{ a.\@ e.\@ in } \mathbb{R}^N
\) as $n \to \infty$.
The lemma is proved.
\end{proof}

\begin{proof}[Proof of Theorem~\ref{teo:existencia}]
The conclusion follows directly from Propositions~\ref{sequenciaps} 
and~\ref{fprclaim4.3}.
\end{proof}

\bigskip

In the proof of Theorem~\ref{teo:existencia} we used, among other ideas, the existence results by 
Bhakta~\cite[Theorems~1 and~2]{MR2934676},
which do not apply in the case $a = b$ and $\mu < 0$; 
however, problem~\eqref{problema} 
still has solution in this situation. 
As can be expected, to prove this existence result we have to use ideas different from the ones already used. We adapt some arguments by 
Filippucci, Pucci and Robert~\cite{MR2498753}
by making a convenient translation in the variable of the cylindrical singularity and, through an analysis of the asymptotic behavior of this function, we obtain the existence result that can be stated as follows.

\begin{teorema} 
\label{teo:existenciamunegativo}
Let
$1 \leqslant k \leqslant N$, $1 < p < N$, 
$0 = a = b < c < 1$, and
$\mu < 0$.
Then there exists
$u\in \mathcal{D}_a^{1,p}
(\mathbb{R}^N\backslash\{|y|=0\})$ 
such that $u>0$ in $\mathbb{R}^N\backslash \{|y|=0\}$ 
and $u$ is a weak solution to problem~\eqref{problema}.
\end{teorema}

\begin{proof}
Le $v \in \mathcal{D}_a^{1,p}(\mathbb{R}^N\backslash\{|y|=0\})$ 
be a nonnegative function that assumes the 
optimal constant $1/K(N,p,0,0,0)$.
Let $e_2 \in \mathbb{R}^k$ 
be a unitary vector 
and set 
$v_\alpha(z) \equiv v(x,y-\alpha e_2)$. 
Moreover, let the functional $\varphi$ be defined as in~\eqref{eq:funcenergia}.

\begin{afirmativa}\label{afirmativaparacasomunegativo} It is valid the inequality
\begin{align*}
\max_{t\geqslant 0} \varphi(tv_\alpha) < \left(\frac{1}{p}-\frac{1}{p^*(0,0)}\right)K(N,p,\mu,0,0)^{-\frac{p^*(0,0)}{p^*(0,0)-p}}.
\end{align*}
\end{afirmativa}
\begin{proof} Let the function $f_3 \colon \mathbb{R}^+ \to \mathbb{R}$ be defined by
\begin{align*}
f_3(t)
&\equiv \frac{1}{p}\|v_\alpha\|^p t^p
- \frac{1}{p^*(0,0)}\left(\int_{\mathbb{R}^N}({v_\alpha}_+)^{p^*(0,0)}\, dz\right)t^{p^*(0,0)}.
\end{align*}
Denoting by $t_{\max}$ the point of maximum for $f_3$ and using the invariance of the three integrals involved in the definition of the infimum under the translations, we obtain
\begin{align*}
\lim_{\alpha \to +\infty} f\sb{3}(t\sb{\max}) 
&= \lim_{\alpha \to +\infty}\left(\frac{1}{p}-\frac{1}{p^*(0,0)}\right)\left\{\frac{\displaystyle\int_{\mathbb{R}^N}|\nabla v_\alpha|^p \, dz 
-\mu\int_{\mathbb{R}^N}\frac{|v_\alpha|^p}{|y+\alpha e_2|^{p}} \, dz}
{\left(\displaystyle\int_{\mathbb{R}^N}(v_{\alpha+})^{p^*(0,0)}\, dz\right)^\frac{p}{p^*(0,0)}}\right\}^\frac{p^*(0,0)}{p^*(0,0)-p}\\
&= \left(\frac{1}{p}-\frac{1}{p^*(0,0)}\right)K(N,p,0,0,0)^{-\frac{p^*(0,0)}{p^*(0,0)-p}}.
\end{align*}
And since
$\mu < 0$ implies 
$1/K(N,p,0,0,0) < 1/K(N,p,\mu,0,0)$, 
we deduce that
\begin{align*}
\max_{t\geqslant 0} \varphi (tv_\alpha) 
&\leqslant \sup_{t\geqslant 0} f\sb{3}(t) = f\sb{3}(t\sb{\max})
 < \left(\frac{1}{p}-\frac{1}{p^*(0,0)}\right)K(N,p,\mu,0,0)^{-\frac{p^*(0,0)}{p^*(0,0)-p}},
\end{align*}
for $\alpha \in \mathbb{R}^+$ big enough.
This concludes the proof of the claim.
\end{proof}

Combining the definition~\eqref{eq:du} of $d_u$ 
and Claim~\ref{afirmativaparacasomunegativo}, 
it follows that
\begin{align}\label{dvaplha}
0 < d_{v_\alpha} \leqslant \max_{t\geqslant 0} \varphi(tv_\alpha) < \left(\frac{1}{p}-\frac{1}{p^*(0,0)}\right)K(N,p,\mu,0,0)^{-\frac{p^*(0,0)}{p^*(0,0)-p}},
\end{align}
that is, the conclusion of Lemma~\ref{claim2.2} still holds in the case $a=b=0$ and $\mu < 0$.

Our next goal is to prove that the conclusion of
Lemma~\ref{claim2.3} also holds in this case.
Using the definition~\eqref{intervalod} of $d_*$, we obtain the inequality 
$0 < d_{v} < d_*$.

The conclusions of 
Proposition~\ref{convergenciapara0} 
are still valid in the case
$a=b=0$ and $\mu <0$, with the same proof.
Finally, combining Propositions~\ref{sequenciaps} and~\ref{fprclaim4.3} we obtain the conclusion of the theorem.
\end{proof}

\bigskip

We can unify the proofs of the existence of solution to problem~\eqref{problema} in the cases 
$0 \leqslant a = b < (k-p)/p$ and $\mu < 0$ 
by considering the subspace of
\(
\mathcal{D}_{a}^{1,p}(\mathbb{R}^N\backslash\{|y|=0\})
\)
of the functions that are cylindrically symmetric in the variable of the singularity,
that is,
in the subspace of the functions such that
$u(z) = u(x,y) = u(x,|y|)$ 
for every $y \in \mathbb{R}^k$.
However, in this case the solution obtained is not necessarily a ground state solution, due to some possible breaking of symmetry.
Our result can be stated as follows.

\begin{teorema}\label{teo:existenciamunegativo2}
Let 
$k\geqslant 2$, $1 <p <N$, $0 \leqslant a < \frac{k-p}{p}$, 
$a=b<c<a+1$ and
$\mu < \bar{\mu}$.
Then problem~\eqref{problema}
has a nonnegative solution.
\end{teorema}

\begin{proof}
Let us define the subspace
$$
\mathcal{D}_{a,\textrm{cyl}}^{1,p}(\mathbb{R}^N\backslash\{|y|=0\})
\equiv\left\{u \in \mathcal{D}_{a}^{1,p}(\mathbb{R}^N\backslash\{|y|=0\}) \colon u(z) = u(x,|y|)\right\}
$$
as the subset of $\mathcal{D}_{a}^{1,p}(\mathbb{R}^N\backslash\{|y|=0\})$ 
consisting of functions cylindrically symmetric in the variable of the singularity. Similarly, we also define the infimum
\begin{align}\label{melhorconstantebcil}
\frac{1}{K_{\textrm{cyl}}(N,p,\mu,a,b)}\equiv \inf_{u\in \mathcal{D}_{a,\text{cyl}}^{1,p}(\mathbb{R}\backslash \{|y|=0\})}
\frac{\displaystyle\int_{\mathbb{R}^N}\frac{|\nabla u|^p}{|y|^{ap}} \, dz -\mu\int_{\mathbb{R}^N}\frac{|u|^p}{|y|^{p(a+1)}} \, dz}
{\left(\displaystyle\int_{\mathbb{R}^N}\frac{(u_+)^{p^*(a,b)}}{|y|^{bp^*(a,b)}}\, dz\right)^\frac{p}{p^*(a,b)}}.
\end{align}

Using the existence result by 
Bhakta~\cite[Teorema 1.3]{MR2934676}, we deduce that there exists a nonnegative, cylindrically symmetric function 
$u(z) = u(x,y) = u(x,|y|) \in \mathcal{D}_{a,\textrm{cyl}}^{1,p}(\mathbb{R}^N\backslash\{|y|=0\})$ that assumes the infimum 
$1/K_{\textrm{cyl}}(N,p,\mu,a,a)$. 
Using a convenient multiple of this function, we obtain a nonnegative, cylindrically symmetric solution to problem 
\begin{align*}
-\operatorname{div}\left[\frac{|\nabla u|^{p-2}}{|y|^{ap}}\nabla u\right]-\mu\frac{|u|^{p-2}u}{|y|^{p(a+1)}}=\frac{|u|^{p^*(a,a)-2}u}{|y|^{bp^*(a,a)}}, \qquad (x,y) \in \mathbb{R}^{N-k}\times\mathbb{R}^k.
\end{align*}

Finally, defining
\begin{align*}
0 < d < d_{*,\textrm{cyl}}
& \equiv  
\min_{m \in \{a,c\}} \bigg\{
\left(\dfrac{1}{p}-\dfrac{1}{p^*(a,m)}\right)
K_{\textrm{cyl}}
(N,p,\mu,a,m)^{-\frac{p^*(a,m)}{p^*(a,m)-p}}\bigg\},
\end{align*}
we can repeat the proofs of the several results in sections~\ref{sec:passodamontanha} and~\ref{sec:convergenciafracazero};
thus, we can extend Theorem~\ref{teo:existencia} 
to the cases 
$0 \leqslant a=b < (k-p)/p$ 
and $\mu <\bar{\mu}$. 
This concludes the proof of the theorem.
\end{proof}

\section{The inductive step of the Moser's iteration scheme}
\label{regularidadesecao1}

In this section we state some auxiliary results that are important to prove our regularity theorem.
It is worth mentioning that our result does not follow directly from the corresponding regularity result by
Filippucci, Pucci, and Robert~\cite{MR2498753}, 
since they applied the conclusions of the theorems by
Druet~\cite{MR1776675}, and by 
Guedda and Ver\'{o}n~\cite{MR1009077}; 
in our case we have to deal, among other things, with the singularity on the differential operator. 
In our proof we apply Moser's iteration scheme and we follow closely the arguments by 
Pucci and Servadei~\cite[Theorem 2.2]{MR2492235}.

Our first result can be stated as follows.

\begin{lema}
\label{phi}
Let $\Omega \subset \mathbb{R}^N\backslash\{|y|=0\}$ 
be a not necessarily bounded domain 
and let
$\eta \colon \Omega \to \mathbb{R}$ 
be a differentiable function with compact support. 
Given the constants $d>0$ and $l>1$, let the functions
$\rho \colon \mathbb{R} \to \mathbb{R}$ and
$\zeta \colon \mathbb{R} \to \mathbb{R}$ be defined by
\begin{align}
\label{2.2}
\rho(t)\equiv t\min\{|t|^{dp},l^{p}\}
\quad \text{and} \quad \zeta(t)\equiv\rho(t)|\eta|^{p}.
\end{align}
Then $\zeta(u) \in \mathcal{D}_{a,0}^{1,p}(\Omega)$.
\end{lema}
\begin{proof}
We begin by defining the subsets
\[
  \begin{array}{r@{{}\equiv{}}l@{\qquad}r@{{}\equiv{}}l}
    \Omega_0 
    & \{z\in \operatorname{supp} \eta: |u(z)|^d < l\}, 
    & \Omega_1 
    & \{z\in \operatorname{supp} \eta: |u(z)|^d = l\}, \\[\jot]
\multicolumn{4}{c}{\Omega_2 \equiv
\{z\in \operatorname{supp} \eta: |u(z)|^d > l\}.}
  \end{array}
\]
Using a result by 
Gilbarg and Trudinger~\cite[Lemma~7.7]{MR1814364},
it follows that
$\rho \in \mathcal{D}_{a}^{1,p}(\Omega_1)$ 
and
$\rho \in \mathcal{D}_{a}^{1,p}(\Omega_2)$, 
because 
$u \in \mathcal{D}_{a,\textrm{loc}}^{1,p}(\Omega)$. 
It remains to show that
$\rho \in \mathcal{D}_{a}^{1,p}(\Omega_0)$.

To do this we consider the function
$G: \mathbb{R} \to \mathbb{R}$ 
defined by
\begin{align*}
G(t) & \equiv
\begin{cases}
t|t|^{pd}, & \textrm{if $|t|^d \leqslant l$};\\
(pd+1)l^pt - pdl^{(pd+1)/d}, &\textrm{if $t^d>l$},\\
(pd+1)l^pt + pdl^{(pd+1)/d}, &\textrm{if $t<-l^{\frac{1}{d}}$.}
\end{cases}
\end{align*}
It is clear that
$G(0)=0$ and $G \in C^1(\mathbb{R})$; 
moreover,
$G' \in L^\infty(\mathbb{R})$. 
Since $\rho (u) = G\circ u|_{\Omega_0}$, 
using another result by 
Gilbarg and Trudinger~\cite[Lemma~7.5]{MR1814364} we deduce that
$\rho \in \mathcal{D}_{a}^{1,p}(\Omega_0)$. 
In $\Omega_0$ we have
\begin{align*}
\nabla \rho(u) &= \nabla(G(u)) = \nabla(u|u|^{pd})= (1+pd)|u|^{pd}\nabla u.
\end{align*}
In $\Omega_2$ we have $|u(z)|^d > l$; 
hence, $\rho(u) = u \min \{|u|^{pd},l^p\} = ul^p$ 
and $\nabla\rho(u) = \nabla(l^pu)=l^p\nabla u.$ 
Finally, in $\Omega_1$
we have $\nabla\rho(u) = 0$ 
because $|u(z)|=l^{1/d}$. Thus, we obtain
\begin{align}\label{2.3}
\nabla\rho(u) &= 
\begin{cases}
(pd+1)|u|^{pd}\nabla u, 
& \textrm{if $z \in \Omega_0$};\\
0, 
& \textrm{if $z \in \Omega_1$};\\
l^p\nabla u, 
& \textrm{if $z \in\Omega_2$.}
\end{cases}
\end{align}
And since $\eta \in C^1_0(\Omega)$,
it follows that
$\zeta \in \mathcal{D}_{a,0}^{1,p}(\Omega)$.
This completes the proof of the lemma.
\end{proof}

Now we prove a crucial result that has an independent value.

\begin{lema}
\label{lema2.18}
Let
$\Omega \subset \mathbb{R}^N\backslash\{|y|=0\}$ 
be a not necessarily bounded domain.
Consider the parameters in the already specified intervals
and let $u\in\mathcal{D}_a^{1,p}(\Omega)$ 
be a weak solution to 
problem~\eqref{problemapucciservadei}. 
For every \( d > 0 \) it is valid the proposition
\begin{align}
\label{2.18}
\textrm{If } u \in 
\mathcal{D}_{a,\operatorname{loc}}^{1,p}(\Omega) 
 \cap L_{a/(d+1),
 \operatorname{loc}}^{p(d+1)}(\Omega), 
\textrm{ then } 
u \in L_{b/(d+1),
\operatorname{loc}}^{p^*(a,b)(d+1)}(\Omega) 
 \cap L_{c/(d+1),
\operatorname{loc}}^{p^*(a,c)(d+1)}(\Omega).
\end{align}
\end{lema}
\begin{proof}
Let the function
$u\in \mathcal{D}^{1,p}_{a,\textrm{loc}}(\Omega)$ 
be a weak solution to problem~\eqref{problemapucciservadei}. 
We define the function
\( f \colon \Omega \times
\mathcal{D}^{1,p}_{a,\textrm{loc}}(\Omega)
\to \mathbb{R} \) by
\begin{align*}
f(z,u)= \mu\frac{|u|^{p-2}u}{|y|^{p(a+1)}}+\frac{(u_+)^{p^*(a,b)-1}}{|y|^{bp^*(a,b)}}+\frac{(u_+)^{p^*(a,c)-1}}{|y|^{cp^*(a,c)}},
\end{align*}
where we denote \( z = (x,y) \in \Omega \). And for \( l > 1 \) we define the function
$\tilde\rho \colon \mathbb{R} \to \mathbb{R}$ 
by
$\tilde{\rho}(t)  \equiv t\min\{|t|^d,l\}$.
Arguing as we have already done in the proof of Lemma~\ref{phi}, we obtain
\begin{align}\label{2.7}
\nabla\tilde{\rho}(u) &= 
\begin{cases}
(d+1)|u|^{d}\nabla u, 
& \textrm{if $z \in \Omega_0$};\\
0, 
& \textrm{if $z \in \Omega_1$};\\
l\nabla u,
& \textrm{if $z \in \Omega_2$}.
\end{cases}
\end{align}
Moreover, it is valid the inequality
\begin{align}\label{comparacaogradientespucciservadei}
|\nabla u |^{p-1}|\rho| \leqslant |\nabla \tilde{\rho}|^{p-1}|\tilde{\rho}| \quad \textrm{a.\@ e.\@ in } \Omega.
\end{align}

Using the function 
$\zeta = \rho|\eta|^p$ 
as a test function in equation~\eqref{solucaofraca}, as well the definition of $f$, we obtain
\begin{align}
\label{2.19}
\int_\Omega \frac{|\nabla u|^{p-2}}{|y|^{ap}} \langle \nabla u , \nabla \zeta \rangle \, dz 
& \leqslant \int_\Omega |\tilde{\rho}\eta|^p\left(|\mu|\frac{1}{|y|^{p(a+1)}}+\frac{|u|^{p^*(a,b)-p}}{|y|^{bp^*(a,b)}}
+\frac{|u|^{p^*(a,c)-p}}{|y|^{cp^*(a,c)}}\right)\, dz.
\end{align}

Using the definition of the function $\zeta$, together with Bernoulli's inequality, that is,
$0 < pd+1 < (d+1)^p $, for $d \geqslant -1$ and $p > 1$, we deduce that
\begin{align*}
{} & \int_\Omega \frac{|\nabla u|^{p-2}}{|y|^{ap}} \langle \nabla u , \nabla \zeta \rangle \, dz \nonumber \\
{} & \qquad =(pd+1)\int_{\Omega_0} \frac{|\nabla u|^{p-2}}{|y|^{ap}}|\eta|^p|u|^{pd}\langle\nabla u , \nabla u\rangle\, dz
+ \int_{\Omega_2} \frac{|\nabla u|^{p-2}}{|y|^{ap}}|\eta|^pl^p\langle\nabla u , \nabla u\rangle\, dz \nonumber \\
{} & \qquad \qquad + \int_\Omega \frac{|\nabla u|^{p-2}}{|y|^{ap}}\rho p |\eta|^{p-2}\eta \langle\nabla u , \nabla \eta\rangle\, dz \nonumber \\
{} &\qquad > \frac{(pd+1)}{(d+1)^p}\int_\Omega \frac{|\eta \nabla\tilde{\rho}|^p}{|y|^{ap}}\, dz
- p\int_\Omega \frac{|\nabla u|^{p-1}}{|y|^{ap}}|u|\min\{|u|^{pd},l^p\} |\eta|^{p-1}|\nabla \eta|\,dz.
\end{align*}
Isolating the first term on the right-hand side of the previous inequality, and using
inequalities~\eqref{comparacaogradientespucciservadei}
and~\eqref{2.19}, we deduce that
\begin{align*}
\frac{(pd+1)}{(d+1)^p}\int_\Omega \frac{|\eta \nabla\tilde{\rho}|^p}{|y|^{ap}}\, dz
&<\int_\Omega |\tilde{\rho}\eta|^p\left(\frac{|\mu|}{|y|^{p(a+1)}}+\frac{|u|^{p^*(a,b)-p}}{|y|^{bp^*(a,b)}}
+\frac{|u|^{p^*(a,c)-p}}{|y|^{cp^*(a,c)}}\right)\, dz \nonumber\\
& \qquad +p\int_\Omega \frac{|\eta \nabla \tilde{\rho}|^{p-1}}{|y|^{ap}}|\tilde{\rho}\nabla \eta| \, dz.
\end{align*}

On the other hand, using Young's inequality on the integrand of the last term but without the singularity, we find
\begin{align*}
|\eta\nabla\tilde{\rho}|^{p-1}|\tilde{\rho}\nabla \eta| 
&\leqslant \frac{|\eta \nabla \tilde{\rho}|^p}{p'\epsilon^{p'}} + \frac{\epsilon^p|\tilde{\rho}\nabla\eta|^p}{p};
\end{align*}
and since $1 < pd+1 $, it follows that
\begin{align*}
\frac{1}{(d+1)^p}\int_\Omega \frac{|\eta \nabla \tilde{\rho}|^{p}}{|y|^{ap}} \, dz
&< \int_\Omega |\tilde{\rho}\eta|^p\left(\frac{|\mu|}{|y|^{p(a+1)}}+\frac{|u|^{p^*(a,b)-p}}{|y|^{bp^*(a,b)}}
+\frac{|u|^{p^*(a,c)-p}}{|y|^{cp^*(a,c)}}\right)\, dz \nonumber\\
&\qquad + \frac{p-1}{\epsilon^{p'}}\int_\Omega \frac{|\eta \nabla \tilde{\rho}|^{p}}{|y|^{ap}} \, dz
+ \epsilon^p\int_\Omega \frac{|\tilde{\rho}\nabla\eta|^{p}}{|y|^{ap}} \, dz.
\end{align*}
Consequently,
\begin{align*}
\left(\frac{1}{(d+1)^p}- \frac{p-1}{\epsilon^{p'}}\right)\int_\Omega \frac{|\eta \nabla \tilde{\rho}|^{p}}{|y|^{ap}} \, dz
&< \int_\Omega |\tilde{\rho}\eta|^p\left(\frac{|\mu|}{|y|^{p(a+1)}}+\frac{|u|^{p^*(a,b)-p}}{|y|^{bp^*(a,b)}}
+\frac{|u|^{p^*(a,c)-p}}{|y|^{cp^*(a,c)}}\right)\, dz\\
& \qquad + \epsilon^p\int_\Omega \frac{|\tilde{\rho}\nabla\eta|^{p}}{|y|^{ap}} \, dz.
\end{align*}
Now we choose $\epsilon^{p'}=2(p-1)(d+1)^p$ 
and we denote
$c_1= \max\{2,2^p(p-1)^{p-1}\}$;
hence, 
\begin{align}\label{2.23}
\left\|\eta\nabla \tilde{\rho}\right\|_{L_a^p(\Omega)}^p
&< 2(d+1)^p\int_\Omega |\tilde{\rho}\eta|^p\left(\frac{|\mu|}{|y|^{p(a+1)}}+\frac{|u|^{p^*(a,b)-p}}{|y|^{bp^*(a,b)}}
+\frac{|u|^{p^*(a,c)-p}}{|y|^{cp^*(a,c)}}\right)\, dz\nonumber\\
&\qquad + 2^p(p-1)^{p-1}(d+1)^{p^2}\int_\Omega \frac{|\tilde{\rho}\nabla\eta|^{p}}{|y|^{ap}} \, dz\nonumber\\
&< c_1(d+1)^{p^2} \int_\Omega |\tilde{\rho}|^p\bigg(\frac{|\eta|^p|\mu|}{|y|^{p(a+1)}}+\frac{|\eta|^p|u|^{p^*(a,b)-p}}{|y|^{bp^*(a,b)}}
+\frac{|\eta|^p|u|^{p^*(a,c)-p}}{|y|^{cp^*(a,c)}} +\frac{|\nabla\eta|^p}{|y|^{ap}}\bigg) dz.
\end{align}

Applying H\"{o}lder's inequality to the second and third integrals on the right-hand side of inequality~\eqref{2.23}, we deduce that 
\begin{align}\label{2.24}
\int_\Omega |\tilde{\rho}\eta|^p\frac{|u|^{p^*(a,b)-p}}{|y|^{bp^*(a,b)}}\, dz
&\leqslant \|\tilde{\rho}\eta\|^p_{L^{p^*(a,b)}_b(\Omega)}\|u\|^{p^*(a,b)-p}_{L_b^{p^*(a,b)}(\operatorname{supp}\eta)}
\end{align}
and
\begin{align}\label{2.24a}
\int_\Omega |\tilde{\rho}\eta|^p\frac{|u|^{p^*(a,c)-p}}{|y|^{cp^*(a,c)}}\, dz
&\leqslant \|\tilde{\rho}\eta\|^p_{L^{p^*(a,c)}_c(\Omega)}\|u\|^{p^*(a,c)-p}_{L_c^{p^*(a,c)}(\operatorname{supp}\eta)}.
\end{align}

Applying Maz'ya's inequality, together with inequality
\( (X + Y)\sp{p} \leqslant 
2\sp{p-1}(X\sp{p} + Y\sp{p})\) 
for \(X,Y \in \mathbb{R}\sp{+} \)
and $p>1$, 
we obtain
\begin{align}\label{2.24a-2}
\|\tilde{\rho}\eta\|^p_{L^{p^*(a,b)}_b(\Omega)}
&\leqslant K(N,p,\mu,a,b)\|\nabla(\tilde{\rho}\eta)\|^p_{L^{p}_a(\Omega)}\nonumber \\
&\leqslant 2^{p-1}K(N,p,\mu,a,b)\left(\|\eta\nabla\tilde{\rho}\|^p_{L^{p}_a(\Omega)}
+ \|\tilde{\rho}\nabla\eta\|^p_{L^{p}_a(\Omega)}\right).
\end{align}
Combining inequalities~\eqref{2.23},~\eqref{2.24},~\eqref{2.24a} and~\eqref{2.24a-2}, we deduce that
\begin{align}\label{fpr2.24a}
\|\tilde{\rho}\eta\|^p_{L^{p^*(a,b)}_b(\Omega)}
&\leqslant 2^{p-1}K(N,p,\mu,a,b)[c_1(d+1)^{p^2}+1]\int_\Omega|\tilde{\rho}|^p\left(\frac{|\mu||\eta|^p}{|y|^{p(a+1)}}+\frac{|\nabla\eta|^p}{|y|^{ap}}\right)\, dz \nonumber\\
&\qquad +2^{p-1}K(N,p,\mu,a,b)c_1(d+1)^{p^2}\|\tilde{\rho}\eta\|^p_{L^{p^*(a,b)}_b(\Omega)}
\|u\|^{p^*(a,b)-p}_{L^{p^*(a,b)}_b(\operatorname{supp}\eta)} \nonumber\\
&\qquad +2^{p-1}K(N,p,\mu,a,b)c_1(d+1)^{p^2}\|\tilde{\rho}\eta\|^p_{L^{p^*(a,c)}_c(\Omega)}
\|u\|^{p^*(a,c)-p}_{L^{p^*(a,c)}_c(\operatorname{supp}\eta)}.
\end{align}
Similarly, we obtain
\begin{align}\label{fpr2.24b}
\|\tilde{\rho}\eta\|^p_{L^{p^*(a,c)}_c(\Omega)}
&\leqslant 2^{p-1}K(N,p,\mu,a,c)[c_1(d+1)^{p^2}+1]\int_\Omega|\tilde{\rho}|^p\left(\frac{|\mu||\eta|^p}{|y|^{p(a+1)}}+\frac{|\nabla\eta|^p}{|y|^{ap}}\right)\, dz \nonumber\\
&\qquad +2^{p-1}K(N,p,\mu,a,c)c_1(d+1)^{p^2}\|\tilde{\rho}\eta\|^p_{L^{p^*(a,b)}_b(\Omega)}
\|u\|^{p^*(a,b)-p}_{L^{p^*(a,b)}_b(\operatorname{supp}\eta)} \nonumber\\
&\qquad +2^{p-1}K(N,p,\mu,a,c)c_1(d+1)^{p^2}\|\tilde{\rho}\eta\|^p_{L^{p^*(a,c)}_c(\Omega)}
\|u\|^{p^*(a,c)-p}_{L^{p^*(a,c)}_c(\operatorname{supp}\eta)}.
\end{align}

Since $u\in L^{p^*(a,b)}_{b,\textrm{loc}}(\Omega) 
\cap L^{p^*(a,c)}_{c,\textrm{loc}}(\Omega)$, 
for every
$z_0 \in \operatorname{supp}\eta$ 
there exists $R = R(z_0) > 0$ 
such that 
$B_R(z_0)\Subset \operatorname{supp}\eta$ 
and it is valid the inequality
\begin{align}\label{fpr2.26}
\max\sb{m \in \{ b,c \}}
\left(\int_{B_{2R}(z_0)}\frac{|u|^{p^*(a,m)}}{|y|^{mp^*(a,m)}}\, dz\right)^\frac{p^*(a,m)-p}{p^*(a,m)}
&\leqslant \min\sb{m \in \{ b,c \}} 
\left[2^{p+1}K(N,p,\mu,a,m)c_1(d+1)^{p^2}\right]^{-1}.
\end{align}

Now we choose the cut off function 
$\eta \in C_0^1(\Omega)$ with the following additional properties: the support of 
$\eta$ is such that
$\operatorname{supp}\eta
\subset B_{2R}(z_0)$, 
$0\leqslant \eta \leqslant 1$, 
$\eta \equiv 1$ in $B_R(z_0)$, 
and $|\nabla\eta|\leq 2/R$.

Using this cut off function in 
inequality~\eqref{fpr2.24a}, 
together inequality~\eqref{fpr2.26} with the appropriate values for \(m \in \{b,c\}\) on both sides, as well as 
the fact that
$|\mu||y|^{-1}|\eta|^p\leqslant c_2$ in $\operatorname{supp}\eta$ 
for some positive constant $c_2 >0$, 
we deduce that
\begin{align}
\label{2.26-3}
{} & \|\tilde{\rho}\eta\|^p_{L^{p^*(a,b)}_b(B_{2R}(z_0))} \nonumber \\
&\qquad \leqslant \dfrac{2^{p+1}}{3} K(N,p,\mu,a,b)\left[c_1(d+1)^{p^2}+1\right]\int_{B_{2R}(z_0)}|\tilde{\rho}|^p\left(\frac{|\mu||\eta|^p}{|y|^{p(a+1)}}+
\frac{|\nabla\eta|^p}{|y|^{ap}}\right)\, dz \nonumber\\
&\qquad \qquad 
+ \frac{1}{3}\, \|\tilde{\rho}\eta\|^p_{L^{p^*(a,c)}_c(B_{2R}(z_0))} \nonumber \\
&\qquad = \dfrac{2^{p+1}}{3} K(N,p,\mu,a,b)\left[c_1(d+1)^{p^2}+1\right]
\int_{B_{2R}(z_0)}\frac{\left(|u|\min\{|u|^d,l\}\right)^p}{|y|^{ap}}\left(\frac{|\mu||\eta|^p}{|y|^p}+|\nabla\eta|^p\right)\, dz \nonumber \\
&\qquad \qquad 
+ \frac{1}{3}\, \|\tilde{\rho}\eta\|^p_{L^{p^*(a,c)}_c(B_{2R}(z_0))} \nonumber \\
&\qquad \leqslant \dfrac{2^{p+1}}{3} K(N,p,\mu,a,b)\left[c_1(d+1)^{p^2}+1\right]
\left(c_2+\frac{2^{p}}{R^p}\right)
\int_{\operatorname{supp}\eta}\left(\frac{|u|^{d+1}}{|y|^a}\right)^p\, dz \nonumber \\
&\qquad \qquad 
+ \frac{1}{3}\, \|\tilde{\rho}\eta\|^p_{L^{p^*(a,c)}_c(B_{2R}(z_0))} \nonumber \\
& \qquad = c_3\|u\|^{p(d+1)}_{L^{p(d+1)}_{a/(d+1)}(\operatorname{supp}\eta)}
+ \frac{1}{3} \,\|\tilde{\rho}\eta\|^p_{L^{p^*(a,c)}_c(B_{2R}(z_0))} 
\end{align}
where $c_3 = 2^{p}K(N,p,\mu,a,b)\left[c_1(d+1)^{p^2}+1\right]\left(c_2+\dfrac{2^{p}}{R^p}\right)$.

Similarly, we obtain
\begin{align}
\label{2.26-4}
\|\tilde{\rho}\eta\|^p_{L^{p^*(a,c)}_c(B_{2R}(z_0))}
\leqslant c_4\|u\|^{p(d+1)}_{L^{p(d+1)}_{a/(d+1)}(\operatorname{supp}\eta)}
+ \frac{1}{3}\,\|\tilde{\rho}\eta\|^p_{L^{p^*(a,b)}_b(B_{2R}(z_0))},
\end{align}
where $c_4 = 2^{p}K(N,p,\mu,a,c)\left[c_1(d+1)^{p^2}+1\right]\left(c_2+\dfrac{2^{p}}{R^p}\right)$.

Adding termwise both sides of 
inequalities~\eqref{2.26-3} and~\eqref{2.26-4}, 
it follows that
\begin{align}
\label{fpr2.26-5-6}
\|\tilde{\rho}\eta\|^p_{L^{p^*(a,b)}_b(B_{2R}(z_0))}
+ \|\tilde{\rho}\eta\|^p_{L^{p^*(a,c)}_c(B_{2R}(z_0))}
\leqslant \dfrac{3}{2}(c_3+c_4)\|u\|^{p(d+1)}_{L^{p(d+1)}_{a/(d+1)}(\operatorname{supp}\eta)}.
\end{align}

Passing to the limit as 
$l \to \infty$ in the first term 
on the left-hand side of inequality~\eqref{fpr2.26-5-6},
we obtain
\begin{align*}
\|u\|^{p(d+1)}_{L^{p^*(a,b)(d+1)}_{b/(d+1)}(B_{R}(z_0))}
\leqslant \dfrac{3}{2}(c_3+c_4) \|u\|^{p(d+1)}_{L^{p(d+1)}_{a/(d+1)}(\operatorname{supp}\eta)}.
\end{align*}
Since $\operatorname{supp} \eta$ can be covered by a finite number of balls with these properties, say \( M \) balls, using  the notation $c_5=M[(3/2)(c_3+c_4)]^{1/(pd+p)}$ we infer that
\begin{align}\label{2.26c}
\|u\|_{L^{p^*(a,b)(d+1)}_{b/(d+1)}(\operatorname{supp}\eta)}
\leqslant c_5 \|u\|_{L^{p(d+1)}_{a/(d+1)}(\operatorname{supp}\eta)}.
\end{align}
Similarly, we have
\begin{align}\label{2.26d}
\|u\|_{L^{p^*(a,c)(d+1)}_{c/(d+1)}(\operatorname{supp}\eta)}
\leqslant c_5 \|u\|_{L^{p(d+1)}_{a/(d+1)}(\operatorname{supp}\eta)}.
\end{align}
Finally, since the cut off function 
$\eta \in C^1_0(\Omega)$ is arbitrary, we conclude the proof of the lemma.
\end{proof}

Now we use the iteration scheme to conclude that
$u \in L^m_{\gamma, \textrm{loc}}$
for every $m \in [1,+\infty)$ and some appropriate
weight $\gamma$.
\begin{lema}
\label{iteracao}
Let
$\Omega \subset \mathbb{R}^N\backslash\{|y|=0\}$ 
be a not necessarily bounded domain.
Consider the parameters in the already specified intervals and let $u\in\mathcal{D}_a^{1,p}(\Omega)$ 
be a weak solution
to problem~\eqref{problemapucciservadei}. 
Then
$u\in L^m_{\gamma, \textrm{loc}}$, 
for every $m \in [1,+\infty)$ 
and some appropriate weight $\gamma = \gamma(m)$.
\end{lema}
\begin{proof}
Using an index notation, 
we rewrite part of proposition~\eqref{2.18} in the form
\begin{align*}
\textrm{If } u \in 
\mathcal{D}_{a,\operatorname{loc}}^{1,p}(\Omega) \cap L_{a/(e_{i-1}+1),\operatorname{loc}}^{p(d_{i-1}+1)}(\Omega),
\textrm{ then } 
u\in L_{a/(e_{i}+1),\operatorname{loc}}^{p(d_{i}+1)}(\Omega).
\end{align*}

Now we choose these indexes so that
\( d_0 = 0 \), \( e_0 = 0 \),
\( d_{i}+1  =(d_{i-1}+1)p^*(a,b)/p \),
and 
\( e_i+1 =(d_{i-1}+1)a/b \)
for \( i = 1, 2, 3, \dots \); this implies that
\( d_{i}+1 = (p^*(a,b)/p)^i \)
and
\( e_i+1 =(p^*(a,b)/p)^{i-1}(a/b) \)
for \( i = 1, 2, 3, \dots \);
therefore,
\begin{align*}
\mbox{If } 
u\in\mathcal{D}_{a,\operatorname{loc}}^{1,p}(\Omega) 
\cap L_{b(p/p^*(a,b))^{i-2},\operatorname{loc}}^{p(p^*(a,b)/p)^{i-1}}(\Omega),
\text{ then } u\in L_{b(p/p^*(a,b))^{i-1},\operatorname{loc}}^{p(p^*(a,b)/p)^{i}}(\Omega), \qquad (i=2,3,4, \dots).
\end{align*}
Similarly,
\begin{align*}
\mbox{If } u\in\mathcal{D}_{a,\operatorname{loc}}^{1,p}(\Omega) 
\cap L_{c(p/p^*(a,c))^{i-2},\operatorname{loc}}^{p(p^*(a,c)/p)^{i-1}}(\Omega),
\mbox{ then }
u\in L_{c(p/p^*(a,c))^{i-1},\operatorname{loc}}^{p(p^*(a,c)/p)^{i}}(\Omega), \qquad (i=2,3,4, \dots).
\end{align*}
This concludes the proof of the lemma.
\end{proof}

\section{Conclusion of the proof of Theorem~\ref{teo:reg}}
\label{regularidadesecao2}
Now we finish the proof of Theorem~\ref{teo:reg}, 
that is, we show that if the function
$u\in\mathcal{D}_a^{1,p}(\Omega)$ is a weak solution to problem~\eqref{problemapucciservadei},
then $u \in L^\infty_{\textrm{loc}}(\Omega)$.

\begin{proof}[Proof of Theorem~\ref{teo:reg}]
Let $E\subset\mathbb{R}^N$ be a bounded domain with 
boundary of class $C^1$ and let
$\widetilde{\Omega}$ be a bounded set so that
$E \Subset \widetilde{\Omega} \Subset \Omega$. 
Repeating the same arguments of section~\ref{regularidadesecao1} and choosing
the cut off function \( \eta \colon \Omega \to \mathbb{R}\) so that \( 0 \leqslant \eta \leqslant 1 \),
\( \operatorname{supp} \eta \subset \widetilde{\Omega} \) and
\( \eta\equiv 1\) in \( E \), 
by the inequality similar to 
inequality~\eqref{2.23}, we deduce that
\begin{align}\label{2.27}
\left\|\nabla \tilde{\rho}\right\|_{L_a^p(E)}^p
&\leqslant c_1(d+1)^{p^2} \int_E |\tilde{\rho}|^p\left(\frac{|\mu|}{|y|^{p(a+1)}}+\frac{|u|^{p^*(a,b)-p}}{|y|^{bp^*(a,b)}}\right)\, dz.
\end{align}
Using Maz'ya's inequality, we have
\begin{align*}
\mbox{If } u\in\mathcal{D}_{a,\operatorname{loc}}^{1,p}(E) 
\cap L_{a,\operatorname{loc}}^{p}(E),
\mbox{ then } u\in L^{p^*(a,b)}_{b,\operatorname{loc}}(E).
\end{align*}
Hence, by the inductive step proved in 
section~\ref{regularidadesecao1}, we obtain
\begin{align*}
\mbox{If } u\in\mathcal{D}_{a,\operatorname{loc}}^{1,p}(E) 
\cap L_{b,\operatorname{loc}}^{p^*(a,b)}(E),
\mbox{ then } u\in L_{pb/p^*(a,b),\textrm{loc}}^{p(p^*(a,b)/p)^{2}}(E).
\end{align*}

Applying H\"{o}lder's inequality to the second term in inequality~\eqref{2.27}, we deduce that
\begin{align}\label{2.28}
\int_E \frac{|\tilde\rho|^p|u|^{p^*(a,b)-p}}{|y|^{bp^*(a,b)}}\,dz
&\leq\|\tilde\rho\|_{L_{br(a,b)}^{p^*(a,b)/r(a,b)}(E)}^p\|u\|_{L_{bp/p^*(a,b)}^{p^*(a,b)^2/p}(E)}^{p^*(a,b)-p},
\end{align}
where 
$r(a,b)=(p^*(a,b)^2-pp^*(a,b)+p^2)/pp^*(a,b) > 1 $
because 
\( a \leqslant b < a + 1 \).

By the choice of the set $E$, by the continuity of the embedding  
$\mathcal{D}_a^{1,p} \hookrightarrow L_b^{p^*(a,b)}(E)$ and by inequalities~\eqref{2.27} and~\eqref{2.28}, 
it follows that
\begin{align}
\label{2.28a}
\|\tilde{\rho}\|_{L_b^{p^*(a,b)}(E)}^p
&\leqslant K(N,p,\mu,a,b)\left(\int_E\left|\frac{\nabla\tilde{\rho}}{|y|^{a}}\right|^{p}\, dz
-\mu\int_E\left|\frac{\tilde{\rho}}{|y|^{a+1}}\right|^{p}\, dz\right)\nonumber\\
&\leqslant K(N,p,\mu,a,b)c_1(d+1)^{p^2}|\mu|\|\tilde\rho\|_{L_{a+1}^{p}(E)}^p\nonumber\\
&\qquad +K(N,p,\mu,a,b)c_1(d+1)^{p^2}\|\tilde\rho\|_{L_{br(a,b)}^{p^*(a,b)/r(a,b)}(E)}^p\|u\|_{L_{bp/p^*(a,b)}^{p^*(a,b)^2/p}(E)}^{p^*(a,b)-p}\nonumber\\
&\qquad -\mu K(N,p,\mu,a,b)\|\tilde\rho\|_{L_{a+1}^p(E)}^p.
\end{align}

Passing to the limit as
$l \to \infty$ and using the definition of  
$\tilde{\rho}$, we obtain 
\begin{align}\label{2.29}
\|u\|_{L_\frac{b}{d+1}^{p^*(a,b)(d+1)}(E)}^{p(d+1)} 
&\leqslant K(N,p,\mu,a,b)c_1(d+1)^{p^2}|\mu|\|u\|_{L_\frac{a+1}{d+1}^{p(d+1)}(E)}^{p(d+1)}\nonumber\\
&\qquad +K(N,p,\mu,a,b)c_1(d+1)^{p^2}\|u\|_{L_\frac{br(a,b)}{d+1}^{\frac{p^*(a,b)(d+1)}{r(a,b)}}(E)}^{p(d+1)} 
\|u\|_{L_\frac{bp}{p^*(a,b)}^{p^*(a,b)^2/p}(E)}^{p^*(a,b)-p} \nonumber\\
&\qquad -\mu K(N,p,\mu,a,b)\|u\|_{L_\frac{a+1}{d+1}^{p(d+1)}(E)}^{p(d+1)}.
\end{align}

Now we are going to estimate the first and the last  terms on the right-hand side of inequality~\eqref{2.29}. Since $u \in L_\textrm{loc}^m(\Omega)$ for every
$m \in [1,+\infty)$ and $E \Subset \widetilde{\Omega} \Subset \Omega$ is a bounded set, 
we have $|y|\geqslant \frac{1}{M_E}$ 
for every $z=(x,y) \in E$.
Hence, applying H\"{o}lder's inequality, we obtain
\begin{align}\label{2.30a}
\|u\|_{L_\frac{a+1}{d+1}^{p(d+1)}(E)}^{p(d+1)} 
&\leqslant \left(\int_E\frac{1}{|y|^\frac{pp^*(a,b)((a+1)-br(a,b))}{p^*(a,b)-pr(a,b)}}\, dz\right)^\frac{p^*(a,b)-pr(a,b)}{p^*(a,b)} 
\left(\int_E\left|\frac{u}{|y|^\frac{br(a,b)}{d+1}}\right|^\frac{p^*(a,b)(d+1)}{r(a,b)}\, dz\right)^\frac{pr(a,b)}{p^*(a,b)}\nonumber\\
&= M_E^{s(a,b)}|E|^{t(a,b)}
\|u\|_{L_\frac{br(a,b)}{d+1}^\frac{p^*(a,b)(d+1)}{r(a,b)}(E)}^{p(d+1)},
\end{align}
where $|E|$ denotes the measure of the subset $E$, 
$s(a,b)=p(a+1)-b(p^*(a,b)-p)-bp^2/p^*(a,b)$ 
and $t(a,b) = p/p^*(a,b)-p^2/(p^*(a,b)^2)$.

Substituting inequality~\eqref{2.30a} in~\eqref{2.29} 
and using the fact that $(d+1)^{p^2}>1$, we obtain
\begin{align*}
&\|u\|_{L_\frac{b}{d+1}^{p^*(a,b)(d+1)}(E)}^{p(d+1)}
\nonumber\\
&\qquad \leqslant K(N,p,\mu,a,b)c_1(d+1)^{p^2}|\mu|M_E^{s(a,b)}|E|^{t(a,b)}
\max\{\|u\|_{L_\frac{br(a,b)}{d+1}^{\frac{p^*(a,b)(d+1)}{r(a,b)}}(E)}^{p(d+1)},1\} \nonumber \\
&\qquad \qquad +K(N,p,\mu,a,b)c_1(d+1)^{p^2}\|u\|_{L_\frac{bp}{p^*(a,b)}^{p^*(a,b)^2/p}(E)}^{p^*(a,b)-p}
\max\{\|u\|_{L_\frac{br(a,b)}{d+1}^{\frac{p^*(a,b)(d+1)}{r(a,b)}}(E)}^{p(d+1)},1\} \nonumber \\
&\qquad \qquad+(d+1)^{p^2}|\mu| K(N,p,\mu,a,b)M_E^{s(a,b)}|E|^{t(a,b)}
\max\{\|u\|_{L_\frac{br(a,b)}{d+1}^{\frac{p^*(a,b)(d+1)}{r(a,b)}}(E)}^{p(d+1)},1\}, \nonumber \\
\end{align*}
that is,
\begin{align}\label{2.33}
\|u\|_{L_\frac{b}{d+1}^{p^*(a,b)(d+1)}(E)}
&\leqslant \left(A(d+1)^{p}\right)^\frac{1}{d+1}\max\{\|u\|_{L_\frac{br(a,b)}{d+1}^{\frac{p^*(a,b)(d+1)}{r(a,b)}}(E)},1\},
\end{align}
where
\begin{align*}
A^p=A(u)^p&\equiv \max\Bigg\{
\splitfrac{3K(N,p,\mu,a,b)c_1|\mu|M_E^{s(a,b)}|E|^{t(a,b)},}
           {3K(N,p,\mu,a,b)c_1\|u\|_{L_\frac{bp}{p^*(a,b)}^{p^*(a,b)^2/p}(E)}^{p^*(a,b)-p}, 
					3|\mu| K(N,p,\mu,a,b)M_E^{s(a,b)}|E|^{t(a,b)} } \Bigg\},
\end{align*}

Now we use Moser's iteration scheme. Choosing 
$d+1=r(a,b) = r > 1$ in inequality~\eqref{2.33}, 
we get
\begin{align}\label{2.33a}
\|u\|_{L_\frac{b}{r}^{p^*(a,b)r}(E)}
&\leqslant \left(Ar^{p}\right)^\frac{1}{r}\max\{\|u\|_{L_b^{p^*(a,b)}(E)},1\}.
\end{align}
Choosing $d+1=r(a,b)\sp{2} = r\sp{2}$ in inequality~\eqref{2.33} and using inequality~\eqref{2.33a}, we get
\begin{align}\label{2.33b}
\|u\|_{L_\frac{b}{r^2}^{p^*(a,b)r^2}(E)}
&\leqslant A^{\frac{1}{r^2}+\frac{1}{r}}\left(r^p\right)^{\frac{2}{r^2}+\frac{1}{r}}
\max\{\|u\|_{L_b^{p^*(a,b)}(E)}
,1,(Ar^p)^{-\frac{1}{r}}\}.
\end{align}
In general, choosing
$(d+1)=r(a,b)^j = r\sp{j}$ for 
$j\in\mathbb{N}$, we get
\begin{align*}
\|u\|_{L_\frac{b}{r^j}^{p^*(a,b)r^j}(E)} 
&\leqslant \left(A^{\frac{1}{r}+
\cdots +\frac{1}{r^j}}\right)
\left(r^p\right)^{\frac{1}{r}+
\cdots +\frac{j}{r^j}}\\
& \qquad
\times \max \Bigg\{
\splitfrac{\|u\|_{L_b^{p^*(a,b)}(E)},1, }
{ \max\left\{A^{-\frac{1}{r}}\left(r^p\right)^{-\frac{1}{r}},
\dots, A^{-\left(\frac{1}{r}+
\cdots +\frac{1}{r^{j-1}}\right)}
\left(r^p\right)^{-\left(\frac{1}{r}+
\cdots +\frac{j-1}{r^{j-1}}\right)}\right\} }
\Bigg\}\\
&\qquad \leqslant A^{\sigma_i}r^{p\sigma'_i}\max\left\{\|u\|_{L_b^{p^*(a,b)}(E)},1,M_{j-1}\right\},
\end{align*}
where
$M_1 \equiv 1$
and
$M_{j-1} \equiv \max_{1 \leqslant i \leqslant j-1}\left\{A^{-\sigma_i}r^{-p\tau_i} \right\}$
for $j>1$, with
$
\sigma_i 
= \sum_{k=1}^i r^{-k}
$ 
and
$
\tau_i 
= \sum_{k=1}^i kr^{-k}.
$
And since $ r = r(a,b) > 1$, we have
\begin{align*}
M_{j-1} & =
\begin{cases}
A^{-\sigma_{j-1}}r^{-\frac{p}{r}}, 
& \textrm{if $A \leqslant 1$};\\
A^{-\frac{1}{r}}r^{-\frac{p}{r}}, 
&\textrm{if $A>1$}.
\end{cases}
\end{align*}
Passing to the limit as $i \to \infty$, 
it follows that
\begin{alignat*}{3}
\lim_{i\to +\infty} \sigma_i & = \frac{1}{r-1}, 
&\qquad
\lim_{i\to +\infty} \tau_i & = \frac{r}{(r-1)^2}, 
&\quad \textrm{and} \quad
\lim_{i\to +\infty} M_{j-1} \equiv
M=
\begin{cases}
A^{-\frac{1}{r-1}}r^{-\frac{p}{r}}, & \textrm{if $A \leqslant 1$ };\\
A^{-\frac{1}{r}}r^{-\frac{p}{r}}, &\textrm{if $A>1$}.
\end{cases}
\end{alignat*}
Therefore,
\begin{align*}
\|u\|_{L^\infty(E)} 
\leqslant A^\frac{1}{r-1}r^\frac{pr}{(r-1)^2}\max\left\{
\|u\|_{L_b^{p^*(a,b)}(E)},1,M\right\}.
\end{align*}
Finally, since the subset $E \Subset \Omega$ is arbitrary, the proof of the theorem is complete.
\end{proof}
 
\section{A Pohozaev-type identity}
\label{naoexistenciasecao1}

In this section we prove 
Theorem~\ref{teo:naoexistencia} by using 
a Pohozaev-type identity, whose underlying principle is the comparison of two different variations of the energy functional at a critical point.
For more details on these types of identities, see
Ekeland and Ghoussoub~\cite[Remark~2.1]{MR1886088}.

Before we prove our first lemma, however, we note that in the case
$1 < q < p^*(a,b)$, 
if $u \in \mathcal{D}_{a}^{1,p}(\mathbb{R}^N\backslash\{|y|=0\})$,
then
$u \in L_{\textrm{loc}}^q(\mathbb{R}^N\backslash\{|y|=0\})$; consequently, the definition of weak solution makes sense. On the other hand, in the case
\( q > p\sp{*}(a,b) \)
the same conclusion is valid if we suppose additionally that
$u \in 
L\sb{bp^*(a,b)/q, \operatorname{loc}}\sp{q}
(\mathbb{R}^N\backslash\{|y|=0\}) 
\cap
L_{\mathrm{loc}}^\infty(\mathbb{R}^N
\backslash\{|y|=0\})$.

\begin{lema}\label{fprclaim5.1}
Let $\eta$, $u\in C_c^\infty(\mathbb{R}^N)$. 
Then
\begin{align}\label{fpr28}
\int_{\mathbb{R}^N}\frac{|\nabla u|^{p-2}}{|y|^{ap}}\big\langle \nabla u, \nabla(\langle z, \nabla(\eta u)\rangle)\big\rangle \,dz
	+\frac{N-p(a+1)}{p}\int_{\mathbb{R}^N}\eta\frac{|\nabla u|^p}{|y|^{ap}}\,dz
	= B(u,\eta)
\end{align}
where
\begin{align*}
B(u,\eta) &= \int_{\mathbb{R}^N}\frac{|\nabla u|^{p-2}}{|y|^{ap}}u\langle \nabla u, \nabla\eta \rangle \, dz
	+\int_{\mathbb{R}^N}\frac{|\nabla u|^{p-2}}{|y|^{ap}}u \sum_{i,j}^N z_i \frac{\partial u}{\partial z_j} \frac{\partial \eta}{\partial z_j\partial z_i}\,dz\nonumber\\
& \qquad + \int_{\mathbb{R}^N}\frac{|\nabla u|^{p-2}}{|y|^{ap}}\langle \nabla u, \nabla \eta \rangle \langle z, \nabla u \rangle \,dz
	+ \frac{p-1}{p}\int_{\mathbb{R}^N}\frac{|\nabla u|^{p}}{|y|^{ap}} \langle z, \nabla \eta \rangle\, dz.
\end{align*}
\end{lema}
\begin{proof}
To prove equality~\eqref{fpr28} 
we expand the expression 
$\big\langle \nabla u, \nabla(\langle z, \nabla(\eta u)\rangle)\big\rangle$. 
In this way, we obtain
\begin{align}\label{fpr29}
&\int_{\mathbb{R}^N}\frac{|\nabla u|^{p-2}}{|y|^{ap}}\big\langle \nabla u, \nabla(\langle z, \nabla(\eta u)\rangle)\big\rangle \,dz\nonumber\\
& \qquad= \int_{\mathbb{R}^N}\eta\frac{|\nabla u|^{p}}{|y|^{ap}}\, dz 
	+ \int_{\mathbb{R}^N}\eta\frac{|\nabla u|^{p-2}}{|y|^{ap}}\sum_{i,j=1}^N z_i\frac{\partial u}{\partial z_j}\frac{\partial^2 u}{\partial z_j\partial z_i}\,dz		
	\nonumber\\
&	\qquad \qquad + \int_{\mathbb{R}^N}\frac{|\nabla u|^{p-2}}{|y|^{ap}}u\langle \nabla u, \nabla\eta \rangle\, dz
+ \int_{\mathbb{R}^N}\frac{|\nabla u|^{p-2}}{|y|^{ap}}u\sum_{i,j=1}^N z_i\frac{\partial u}{\partial z_j}\frac{\partial^2 \eta}{\partial z_j\partial z_i}\,dz
	\nonumber\\
& \qquad \qquad + \int_{\mathbb{R}^N}\frac{|\nabla u|^{p-2}}{|y|^{ap}}\langle \nabla u, \nabla \eta\rangle \langle z, \nabla u\rangle\,dz + \int_{\mathbb{R}^N}\frac{|\nabla u|^{p}}{|y|^{ap}}\langle z, \nabla \eta \rangle \, dz.
\end{align}

Now we expand the second term on the right-hand side of inequality~\eqref{fpr29}.
Using the divergence theorem and recalling that $\eta \in C_0^\infty$ we obtain
\begin{align}\label{fpr30}
\int_{\mathbb{R}^N}\eta \frac{|\nabla u|^{p-2}}{|y|^{ap}}
\sum_{i,j=1}^N z_i\frac{\partial u}{\partial z_j}\frac{\partial^2 u}{\partial z_i\partial z_j}\,dz
& =\int_{\mathbb{R}^N}\frac{\eta}{|y|^{ap}} \sum_{i=1}^N z_i \frac{\partial}{\partial z_i}\left(\frac{|\nabla u|^p}{p}\right) \,dz\nonumber\\
& = - \frac{1}{p}\int_{\mathbb{R}^N}|\nabla u|^p \sum_{i=1}^N\frac{\partial}{\partial z_i}\left(\frac{\eta}{|y|^{ap}}z_i\right) \,dz\nonumber\\
& = - \frac{1}{p}\int_{\mathbb{R}^N}\frac{|\nabla u|^p}{|y|^{ap}}\langle\nabla\eta,z\rangle \,dz
 + \left(a- \frac{N}{p}\right)\int_{\mathbb{R}^N}\eta\frac{|\nabla u|^p}{|y|^{ap}} \,dz.
\end{align}
Substituing the equalities~\eqref{fpr29} and~\eqref{fpr30} on the left-hand side of equality~\eqref{fpr28}
 we conclude the proof of the lemma.
\end{proof}

\begin{lema}\label{fprclaim5.2}
If  $u \in \mathcal{D}_a^{1,p}(\mathbb{R}^N\backslash\{|y|=0\}) \cap C^1(\mathbb{R}^N\backslash\{|y|=0\}) \cap W^{2,1}_\mathrm{loc}(\mathbb{R}^N\backslash\{|y|=0\})$, then the identity~\eqref{fpr28} is valid.
\end{lema}
\begin{proof}
By a density argument, there exists a sequence 
$(\varphi_n)_{n\in\mathbb{N}} \in C_c^\infty(\mathbb{R}^N\backslash\{|y|=0\})$ 
such that
$\lim_{n \to \infty} \varphi_n = u$ in $C_\textrm{loc}^1(\mathbb{R}^N\backslash\{|y|=0\}) \cap W^{2,1}_\mathrm{loc}(\mathbb{R}^N\backslash\{|y|=0\})$. Applying Lemma~\ref{fprclaim5.1} 
to the functions $\eta$ and $\varphi_n$,
and passing to the limit as $n \to \infty$ in equality~\eqref{fpr28}
we conclude the proof of the lemma.
\end{proof}

\section{Conclusion of the proof of 
Theorem~\ref{teo:naoexistencia}}
\label{naoexistenciasecao2}

Now we can use the Pohozaev-type identity~\eqref{fpr28} to show that there exists no nontrivial solution to 
problem~\eqref{fprteorema3}. 

\begin{lema}\label{fprclaim5.3}
Let $f \in C^0((\mathbb{R}^N\backslash\{|y|=0\})\times \mathbb{R})$ 
and let 
$u \in \mathcal{D}_a^{1,p}(\mathbb{R}^N\backslash\{|y|=0\}) \cap C^1(\mathbb{R}^N\backslash\{|y|=0\}) \cap W^{2,1}_\mathrm{loc}(\mathbb{R}^N\backslash\{|y|=0\})$ be a solution to problem
\begin{align}\label{fpr31}
-\operatorname{div} \left[\frac{|\nabla \xi |^{p-2}}{|y|^{ap}}\nabla \xi \right] = f(z,\xi ) \quad  z\in\mathbb{R}^N.
\end{align}
Suppose that the function 
$F(z,\xi ) \equiv \displaystyle \int^{\xi }_{0}f(z,v)\, dv$ 
is such that
$F \in C^1((\mathbb{R}^N\backslash\{|y|=0\})\times \mathbb{R})$; suppose also that 
$\xi f(\cdot,\xi )$, $F(\cdot,\xi )$, $\sum_{i=1}^Nz_i\dfrac{\partial F}{\partial z_i}(\cdot,\xi ) \in L_a^1(\mathbb{R}^N)$. 
Then
\begin{align}\label{fpr32}
\int_{\mathbb{R}^N} \left[ \frac{N-p(a+1)}{p}\xi f(z,\xi ) - NF(z,\xi ) - \sum_{i=1}^Nz_i\frac{\partial F}{\partial z_i}(z,\xi )\right] \,dz = 0
\end{align}
\end{lema}
\begin{proof}
Let us consider the test function $\eta \in C_c^\infty(\mathbb{R}^N\backslash\{|y|=0\})$; 
let us also consider the sequence
$(\varphi_n)_{n\in\mathbb{N}} \in C_c^\infty(\mathbb{R}^N\backslash\{|y|=0\})$ 
such that 
$\lim_{n \to \infty} \varphi_n = \xi $ 
in $C_\textrm{loc}^1(\mathbb{R}^N\backslash\{|y|=0\}) \cap W^{2,1}_\mathrm{loc}(\mathbb{R}^N\backslash\{|y|=0\})$. 

Multiplying both sides of equation~\eqref{fpr31}
by $\langle z, \nabla(\eta \varphi_n) \rangle$ 
and using the divergence theorem, we obtain
\begin{align}\label{fpr33}
&\int_{\mathbb{R}^N}\frac{|\nabla \xi |^{p-2}}{|y|^{ap}} \big\langle \nabla \xi , \nabla \langle z,\nabla(\eta \xi ) \rangle \big\rangle \,dz
= \lim_{n \to +\infty} \int_{\mathbb{R}^N}\frac{|\nabla \xi |^{p-2}}{|y|^{ap}}
	\big\langle \nabla \xi , \nabla \langle z,\nabla(\eta \varphi_n) \rangle \big\rangle \,dz\nonumber\\
& \qquad = \lim_{n \to +\infty} \int_{\mathbb{R}^N}f(z,\xi ) \langle z,\nabla(\eta \varphi_n) \rangle  \,dz
= \int_{\mathbb{R}^N}f(z,\xi ) \langle z,\nabla(\eta \xi ) \rangle  \,dz\nonumber\\
& \qquad = \int_{\mathbb{R}^N}\xi f(z,\xi ) \langle z,\nabla\eta \rangle  \,dz
	+ \int_{\mathbb{R}^N}\eta \sum_{i=1}^Nz_i\frac{\partial F(z,\xi )}{\partial z_i}  \,dz
 - \int_{\mathbb{R}^N}\eta \sum_{i=1}^Nz_i\frac{\partial F}{\partial z_i}(z,\xi )\, dz\nonumber\\
& \qquad = \int_{\mathbb{R}^N}\xi f(z,\xi ) \langle z,\nabla\eta \rangle  \,dz
	- \int_{\mathbb{R}^N} \sum_{i=1}^N\frac{\partial (\eta z_i)}{\partial z_i}F(z,\xi )  \,dz
 - \int_{\mathbb{R}^N}\eta \sum_{i=1}^Nz_i\frac{\partial F}{\partial z_i}(z,\xi )\, dz.
\end{align}

On the other hand, multiplying both sides of equation~\eqref{fpr31} by the test function $\eta\varphi_n$, we obtain
\begin{align*}
\int_{\mathbb{R}^N}\frac{|\nabla \xi |^{p-2}}{|y|^{ap}} \langle \nabla \xi , \nabla (\eta \xi ) \rangle \,dz
& = \lim_{n \to +\infty} \int_{\mathbb{R}^N}\frac{|\nabla \xi |^{p-2}}{|y|^{ap}} \langle \nabla \xi , \nabla (\eta \varphi_n) \rangle \,dz\\
& = \lim_{n \to +\infty} \int_{\mathbb{R}^N}f(z,\xi ) \eta \varphi_n \,dz\\
& = \int_{\mathbb{R}^N}f(z,\xi ) \eta \xi  \,dz.
\end{align*}
Expanding the left-hand side of the previous equality, it follows that
\begin{align}\label{fpr34}
\int_{\mathbb{R}^N}\eta\frac{|\nabla \xi |^{p}}{|y|^{ap}} \,dz
= \int_{\mathbb{R}^N}f(z,\xi ) \eta \xi  \,dz
	-\int_{\mathbb{R}^N}\xi \frac{|\nabla \xi |^{p-2}}{|y|^{ap}} \langle \nabla \xi , \nabla \eta \rangle \,dz.
\end{align}
Substituting equalities~\eqref{fpr33} and~\eqref{fpr34} in equality~\eqref{fpr28}, we find that 
\begin{align}\label{fpr34-2}
&\int_{\mathbb{R}^N}\xi f(z,\xi ) \langle z,\nabla\eta \rangle  \,dz
	- \int_{\mathbb{R}^N} \sum_{i=1}^N\frac{\partial (\eta z_i)}{\partial z_i}F(z,\xi )  \,dz
	- \int_{\mathbb{R}^N}\eta \sum_{i=1}^Nz_i\frac{\partial F}{\partial z_i}(z,\xi )\, dz\nonumber\\
&\qquad \qquad +\frac{N-p(a+1)}{p}\int_{\mathbb{R}^N}f(z,\xi ) \eta \xi  \,dz
	-\frac{N-p(a+1)}{p}\int_{\mathbb{R}^N}\xi \frac{|\nabla \xi |^{p-2}}{|y|^{ap}} \langle \nabla \xi , \nabla \eta \rangle \,dz\nonumber\\
&\qquad = \int_{\mathbb{R}^N}\frac{|\nabla \xi |^{p-2}}{|y|^{ap}}\xi \langle \nabla \xi , \nabla\eta \rangle \, dz
	+\int_{\mathbb{R}^N}\frac{|\nabla \xi |^{p-2}}{|y|^{ap}}\xi  \sum_{i,j}^N z_i \frac{\partial \xi }{\partial z_j} 
	\frac{\partial^2 \eta}{\partial z_j\partial z_i}\,dz\nonumber\\
& \qquad \qquad + \int_{\mathbb{R}^N}\frac{|\nabla \xi |^{p-2}}{|y|^{ap}}\langle \nabla \xi , \nabla \eta \rangle \langle z, \nabla \xi  \rangle \,dz
	+ \frac{p-1}{p}\int_{\mathbb{R}^N}\frac{|\nabla \xi |^{p}}{|y|^{ap}} \langle z, \nabla \eta \rangle\, dz.
\end{align}

Now we expand the second term on the left-hand side of the previous equality, and we obtain
\begin{align*}
\int_{\mathbb{R}^N} \sum_{i=1}^N\frac{\partial (\eta z_i)}{\partial z_i}F(z,\xi )  \,dz
& = \int_{\mathbb{R}^N} \langle \nabla\eta, z \rangle F(z,\xi )  \,dz
	+ \int_{\mathbb{R}^N}N\eta F(z,\xi )  \,dz.
\end{align*}
Substituting the previous equality in equality~\eqref{fpr34-2}
and reordering the terms, it follows that
\begin{align}\label{fpr34-3}
&	\frac{N-p(a+1)}{p}\int_{\mathbb{R}^N}f(z,\xi ) \eta \xi  \,dz
	- \int_{\mathbb{R}^N}N\eta F(z,\xi )  \,dz
	- \int_{\mathbb{R}^N}\eta \sum_{i=1}^Nz_i\frac{\partial F}{\partial z_i}(z,\xi )\, dz\nonumber\\
& \qquad = \frac{N-ap}{p}\int_{\mathbb{R}^N}\xi \frac{|\nabla \xi |^{p-2}}{|y|^{ap}} \langle \nabla \xi , \nabla \eta \rangle \,dz
	+\int_{\mathbb{R}^N}\frac{|\nabla \xi |^{p-2}}{|y|^{ap}}\xi  \sum_{i,j}^N z_i \frac{\partial \xi }{\partial z_j} 
	\frac{\partial^2 \eta}{\partial z_j\partial z_i}\,dz\nonumber\\
& \qquad \qquad + \int_{\mathbb{R}^N}\frac{|\nabla \xi |^{p-2}}{|y|^{ap}}\langle \nabla \xi , \nabla \eta \rangle \langle z, \nabla \xi  \rangle \,dz	
	+ \frac{p-1}{p}\int_{\mathbb{R}^N}\frac{|\nabla \xi |^{p}}{|y|^{ap}} \langle z, \nabla \eta \rangle\, dz\nonumber\\
& \qquad \qquad -\int_{\mathbb{R}^N}\xi f(z,\xi ) \langle z,\nabla\eta \rangle  \,dz
	+ \int_{\mathbb{R}^N} \langle \nabla\eta, z \rangle F(z,\xi )  \,dz.
\end{align}

Now we must estimate each one of the terms on the right-hand side of equality~\eqref{fpr34-3}. For the first and second terms, using H\"{o}lder's inequality we obtain
\begin{align*}
\left|\frac{N-ap}{p}\int_{\mathbb{R}^N}\xi \frac{|\nabla \xi |^{p-2}}{|y|^{ap}} \langle \nabla \xi , \nabla \eta \rangle \,dz\right|
\leqslant \frac{N-ap}{p}\|\nabla \xi \|_{L_a^p(\operatorname{supp}|\nabla \eta|)}^{p-1}\|\xi \|_{L_a^{p^*}(\mathbb{R}^N)}\|\nabla \eta\|_{L^N(\mathbb{R}^N)}
\end{align*}
and
\begin{align*}
&\int_{\mathbb{R}^N}\frac{|\nabla \xi |^{p-2}}{|y|^{ap}}\xi  \sum_{i,j}^N z_i \frac{\partial \xi }{\partial z_j} 
	\frac{\partial^2 \eta}{\partial z_j\partial z_i}\,dz 
\leqslant c\|\nabla \xi \|_{L_a^p(\operatorname{supp}|\nabla \eta|)}^{p-1}\|\xi \|_{L_a^{p^*}(\mathbb{R}^N)}
	\left\||z|\sum_{i=1}^{N}\left|\nabla\left(\frac{\partial\eta}{\partial z_i}\right)\right|\right\|_{L^N(\mathbb{R}^N)}.
\end{align*}
For the third and fourth terms, we find
\begin{align*}
\left|\int_{\mathbb{R}^N}\frac{|\nabla \xi |^{p-2}}{|y|^{ap}} \langle \nabla \xi , \nabla \eta \rangle
	\langle \nabla z, \nabla \xi  \rangle \,dz\right|
& \leqslant \big\||\nabla \eta||z|\big\|_\infty\|\nabla \xi \|_{L_a^p(\operatorname{supp}|\nabla \eta|)}
\end{align*}
and
\begin{align*}
\left|\frac{p-1}{p}\int_{\mathbb{R}^N}\frac{|\nabla \xi |^{p}}{|y|^{ap}} \langle z, \nabla \eta \rangle \,dz\right|
& \leqslant \frac{p-1}{p}\big\||\nabla \eta||z|\big\|_\infty\|\nabla \xi \|_{L_a^p(\operatorname{supp}|\nabla \eta|)}.
\end{align*}
Finally, for the fifth and sixth terms, we find
\begin{align*}
\left|\int_{\mathbb{R}^N} \langle \nabla\eta, z \rangle \left(F(z,\xi )-\xi f(z,\xi )\right)  \,dz\right|
\leqslant \big\||\nabla \eta||z|\big\|_\infty \int_{\operatorname{supp}|\nabla\eta|}|F(z,\xi )-\xi f(z,\xi )| \,dz.
\end{align*}

Our objective now is to study the asymptotic behavior of the bounds from above of these several inequalities. To do this we choose an appropriate cut off function $\eta$. 
Consider the function 
$h \in C^\infty(\mathbb{R})$ such that
 $h\big|_{\{t\leqslant 1\}} \equiv 0$, $h\big|_{\{t\geqslant 2\}} \equiv 1$ and $0 \leqslant h \leqslant 1$. 
Given $\epsilon > 0$ small enough, we define $\eta_\epsilon(z) = h(|z|/\epsilon)$ 
if $|z|\leqslant 3\epsilon$, 
$\eta_\epsilon(z) = h(1/(\epsilon|z|))$ 
if $|z| \geqslant (2\epsilon)^{-1}$,
and $\eta_\epsilon(z) = 1$
otherwise.
We also choose the function $h$ such that
$|h'(z)|\leqslant 2$; with these choices, we have $\eta_\epsilon \in C_c^\infty(\mathbb{R}\backslash\{0\})$. 
Using $\eta = \eta_\epsilon$, we show that the bounds from above vanish as $\epsilon \to 0$.

In the first place, we estimate the term $|\nabla\eta|$. Since the function $\eta$ is radial, we use the notation 
$|z|=r$. Hence, for $r<2\epsilon$ we obtain
\begin{align*}
\eta_\epsilon(r) = h\left(\frac{r}{\epsilon}\right)
\quad \textrm{and} \quad
|\eta'_\epsilon(r)| \leqslant \left|h'\left(\frac{r}{\epsilon}\right)\right|\frac{1}{\epsilon} \leqslant \frac{2}{\epsilon}.
\end{align*}
In a similar way, if $r>1/(2\epsilon)$, 
then it follows that
\begin{align*}
\eta_\epsilon(r) = h\left(\frac{1}{\epsilon r}\right)
\quad \textrm{and} \quad
|\eta'_\epsilon(r)| \leqslant \left|h'\left(\frac{1}{\epsilon r}\right)\right|\frac{1}{\epsilon}\frac{1}{r^2} \leqslant \frac{2}{\epsilon}\frac{1}{r^2}.
\end{align*}

Now we estimate the term
$\|\nabla \eta\|_{L^N(\mathbb{R}^N)}^N$. Using the properties of the cut off function $\eta$, we obtain
\begin{align*}
\|\nabla\eta\|_{L^N(\mathbb{R}^N)}^N 
&= \int_{\epsilon<r<2\epsilon}\int_{\partial B_r}|\nabla\eta|^N \,d\nu dr
	+ \int_{\frac{1}{2\epsilon}<r<\frac{1}{\epsilon}}\int_{\partial B_r}|\nabla\eta|^N \,d\nu dr\\
& \leqslant \int_{\epsilon<r<2\epsilon} \frac{2^N}{\epsilon^N}r^{N-1} \,dr 
	+ \int_{\frac{1}{2\epsilon}<r<\frac{1}{\epsilon}}\frac{2^N}{\epsilon^N}\frac{1}{r^{2N}}r^{N-1}\,dr\\
&= \frac{2^{2+1}}{N}(2^N-1).
\end{align*}
Therefore, $\|\nabla \eta\|_{L^N(\mathbb{R}^N)}^N$ 
is finite and does not depend on $\epsilon$.

Moreover, we have
\begin{align*}
\left\||z|\sum_{i=1}^{N}\left|\nabla\left(\frac{\partial\eta}{\partial z_i}\right)\right|\right\|_{L^N(\mathbb{R}^N)} 
& \leqslant c\int_{\mathbb{R}^N}\left(\sum_{i=1}^{N}\left|\nabla\left(\frac{\partial\eta}{\partial z_i}\right)\right|\right)^N\, dz.
\end{align*}
By the properties of the cut off function $\eta$ this term is finite, as well as the term $\|\left|\nabla \eta\right| |z|\|_\infty$.

Following up, 
we show that $\|\nabla \xi \|_{L_a^p(\operatorname{supp}|\nabla\eta|)}^{p-1} \to 0$ as $\epsilon \to 0$. 
Since
$\|\nabla \xi \|_{L_a^p(\mathbb{R}^N)} < +\infty$, we consider balls centered at the origin with radii 
$0 < r_\epsilon < R_\epsilon$ such that $\operatorname{supp}|\nabla\eta|\subset B_{r_\epsilon}(0) \cup \left(\mathbb{R}^N\backslash B_{R_\epsilon}(0)\right)$; for example, we can take 
\( r\sb{\epsilon} = 2\epsilon \) and 
\( R\sb{\epsilon} = 1/2\epsilon \). 
As $\epsilon \to 0$ we obtain $r_\epsilon \to 0$ and $R_\epsilon\to +\infty$; hence,
$$
\int_{\operatorname{supp}|\nabla\eta|} \frac{|\nabla \xi |}{|y|^{ap}} \, dz
\leqslant \int_{B_{r_\epsilon}(0) \cup \left(\mathbb{R}^N\backslash B_{R_\epsilon}(0)\right)} \frac{|\nabla \xi |}{|y|^{ap}} \, dz
\to 0.
$$
Using this same argument, we can show that
$$
\int_{\operatorname{supp}|\nabla\eta|}|F(z,\xi )-\xi f(z,\xi )| \,dz \to 0 
$$ 
as $\epsilon \to 0$.
As a result, we obtain equality~\eqref{fpr32}. The lemma is proved.  
\end{proof}

\begin{lema}
\label{fprclaim5.4}
If $u \in \mathcal{D}_a^{1,p}(\mathbb{R}^N\backslash\{|y|=0\}) \cap C^1(\mathbb{R}^N\backslash\{|y|=0\}) \cap W^{2,1}_\mathrm{loc}(\mathbb{R}^N\backslash\{|y|=0\})$ is a weak solution to problem~\eqref{fprteorema3} 
for $q>1$ and $q \neq p^*(a,b)$, then $u \equiv 0$. 
\end{lema}
\begin{proof}
To use Lemma~\ref{fprclaim5.3} we need to prove that
$u \in L_{bp^*(a,b)/q}^q(\mathbb{R}^N)$. 
To do this, we use the test function 
$\eta_\epsilon u$ in the weak solution to problem~\eqref{fprteorema3}, where 
$\eta_\epsilon\in C_c^\infty(\mathbb{R}^N\backslash\{|y|=0\})$ is the same function defined in the proof of Lemma~\ref{fprclaim5.3}. Thus, we have
\begin{align}\label{primeiraigualdadedolemafprclaim5.4}
\int_{\mathbb{R}^N}\frac{|\nabla u|^{p-2}}{|y|^{ap}} \langle \nabla u, \nabla (\eta_\epsilon u) \rangle	 \,dz
-\mu \int_{\mathbb{R}^N} \frac{\eta_\epsilon |u|^p}{|y|^{p(a+1)}} \, dz
&= \int_{\mathbb{R}^N} \frac{\eta_\epsilon |u|^{p^*(a,c)}}{|y|^{cp^*(a,c)}} \, dz
+ \int_{\mathbb{R}^N} \frac{\eta_\epsilon |u|^q}{|y|^{bp^*(a,b)}} \, dz.
\end{align}

In what follows we estimate the first term on the right-hand side of equality~\eqref{primeiraigualdadedolemafprclaim5.4}, and we obtain
\begin{align*}
\left|\int_{\mathbb{R}^N} \frac{\eta_\epsilon |u|^{p^*(a,c)}}{|y|^{cp^*(a,c)}} \, dz\right|
& \leqslant \int_{\mathbb{R}^N} \frac{|u|^{p^*(a,c)}}{|y|^{cp^*(a,c)}} \, dz\\
& \leqslant K(N,p,\mu,a,c)^\frac{p^*(a,c)}{p}\left(\int_{\mathbb{R}^N} \frac{|\nabla u|^p}{|y|^{ap}} \, dz
	- \mu\int_{\mathbb{R}^N} \frac{|u|^{p}}{|y|^{p(a+1)}} \, dz\right)^\frac{p^*(a,c)}{p}.
\end{align*}
We also have
\begin{align*}
&\left|\int_{\mathbb{R}^N}\frac{|\nabla u|^{p-2}}{|y|^{ap}} \langle \nabla u, \nabla (\eta_\epsilon u) \rangle	 \,dz\right|\\
& \qquad \leqslant c + \left(\int_{\mathbb{R}^N}\frac{|\nabla u|^{p}}{|y|^{ap}}\,dz\right)^\frac{p-1}{p}
\left(\int_{\mathbb{R}^N} \frac{|u|^{p^*}}{|y|^{ap^*}} \, dz\right)^\frac{1}{p^*}
\left(\int_{\mathbb{R}^N}|\nabla \eta_\epsilon|^N\,dz\right)^\frac{1}{N},
\end{align*}
where we have applied Maz'ya's and H\"{o}lder's inequalities to obtain three factors that are finite and independent from $\epsilon$. 

To estimate the second term on the left-hand side of equality~\eqref{primeiraigualdadedolemafprclaim5.4}, 
we write
\begin{align*}
\left|\mu \int_{\mathbb{R}^N} \frac{\eta_\epsilon |u|^p}{|y|^{p(a+1)}} \, dz\right| 
&\leqslant K(N,p,\mu,a,a+1)\left(\int_{\mathbb{R}^N} \frac{|\nabla u|^p}{|y|^{ap}} \, dz
	- \mu\int_{\mathbb{R}^N} \frac{|u|^{p}}{|y|^{p(a+1)}} \, dz\right),
\end{align*}
which is finite because
$u\in\mathcal{D}_a^{1,p}(\mathbb{R}^N\backslash\{|y|=0\})$. 
Therefore, passing to the limit as 
$\epsilon \to 0$, we deduce that
$u \in L_{bp^*(a,b)/q}^q(\mathbb{R}^N)$. 

Now we can apply Lemma~\ref{fprclaim5.3}
to the functions 
\begin{align*}
f(z,\xi )= \mu\frac{|\xi |^{p-2}}{|y|^{p(a+1)}}\xi  + \frac{|\xi |^{p^*(a,c)-2}}{|y|^{cp^*(a,c)}}\xi  + \frac{|\xi |^{q-2}}{|y|^{bp^*(a,b)}}\xi 
\end{align*}
and
\begin{align*}
F(z,\xi )
 & =\frac{\mu}{p}\frac{|\xi |^{p}}{|y|^{p(a+1)}} + \frac{1}{p^*(a,c)}\frac{|\xi |^{p^*(a,c)}}{|y|^{cp^*(a,c)}} + \frac{1}{q}\frac{|\xi |^{q}}{|y|^{bp^*(a,b)}}.
\end{align*}
In this way, from equality~\eqref{fpr32} we deduce that
\begin{align*}
0
& = \left(\frac{N-p(a+1)}{p} - \frac{N}{p} + (a+1) \right)\mu\int_{\mathbb{R}^N} \frac{|\xi |^{p}}{|y|^{p(a+1)}} \, dz\\
& \qquad+\left(\dfrac{N-p(a+1)}{p} - \dfrac{N}{p^*(a,c)} + c \right)\int_{\mathbb{R}^N} \frac{|\xi |^{p^*(a,c)}}{|y|^{cp^*(a,c)}} \,dz\\
&\qquad +\left(\dfrac{N-p(a+1)}{p} -\dfrac{N}{q} + \dfrac{bp^*(a,b)}{q} \right)\int_{\mathbb{R}^N} \frac{|\xi |^{q}}{|y|^{bp^*(a,b)}} \,dz\\
& = \bigg(\dfrac{1}{p^*(a,b)}-\dfrac{1}{q}\bigg)(N-bp^*(a,b)) \int_{\mathbb{R}^N} \frac{|\xi |^{q}}{|y|^{bp^*(a,b)}} \, dz.
\end{align*}
By hypothesis we have $q \neq p^*(a,b)$; 
moreover, $a < (N-p)/p$ implies that $(N-bp^*(a,b)) \neq 0 $. 
Hence,
\( \| \xi  \|\sb{L_{bp^*(a,b)/q}^q(\mathbb{R}^N)}\sp{q} = 0 \), that is, \( \xi  \equiv 0 \). 
This concludes the proof of the lemma. 
\end{proof}
\begin{lema}\label{fprclaim5.5}
Let $u \in \mathcal{D}_{a}^{1,p}(\mathbb{R}^N\backslash\{|y|=0\})$ 
be a weak solution to problem~\eqref{fprteorema3}, 
where $1<q<p^*(a,b)$. Then
$
u \in \mathcal{D}_a^{1,p}(\mathbb{R}^N\backslash\{|y|=0\}) \cap C^1(\mathbb{R}^N\backslash\{|y|=0\}) \cap W^{2,1}_\mathrm{loc}(\mathbb{R}^N\backslash\{|y|=0\}).
$
\end{lema}
\begin{proof}
We begin by writing problem~\eqref{fprteorema3} in the form
$$-\operatorname{div}\left[\frac{|\nabla u|^{p-2}}{|y|^{ap}}\nabla u\right]=f(x,u),
$$
where
$$
f(x,u) = \mu\frac{|u|^{p-2}u}{|y|^{p(a+1)}}
+\frac{|u|^{p^*(a,c)-2}u}{|y|^{cp^*(a,c)}}+\frac{|u|^{q-2}u}{|y|^{bp^*(a,b)}}.
$$
Hence, for every subset
$\omega \Subset \mathbb{R}^N\backslash\{|y|=0\}$ 
there exists a positive constant 
$C(\omega) > 0$ such that 
$$|f(x,u)| 
\leqslant C(\omega)\left(\mu\frac{|u|^{p-2}u}{|y|^{p(a+1)}}+\frac{|u|^{q-2}u}{|y|^{bp^*(a,b)}}\right).
$$
In this way, from Theorem~\ref{teo:reg} it follows that
$u \in L_{\mathrm{loc}}^\infty(\mathbb{R}^N\backslash\{|y|=0\})$. 
Using a result by 
Tolksdorf~\cite{MR727034}, 
we deduce that 
$u\in C^1(\mathbb{R}^N\backslash\{|y|=0\}) \cap W^{2,1}_\mathrm{loc}(\mathbb{R}^N\backslash\{|y|=0\})$.
The lema is proved.
\end{proof}

\begin{proof}[Proof of Theorem~\ref{teo:naoexistencia}]
In the case $1<q< p^*(a,b)$ the proof follows immediately from 
Lemmas~\ref{fprclaim5.4} and~\ref{fprclaim5.5}. 
In the case $q > p^*(a,b)$ we cannot  apply directly 
Lemma~\ref{fprclaim5.5} because the conclusion of
Theorem~\ref{teo:reg} is not valid. To overcome this difficulty, we suppose additionally that
$u \in
L\sb{bp^*(a,b)/q, \operatorname{loc}}\sp{q}
(\mathbb{R}^N\backslash\{|y|=0\}) 
\cap
 L_{\mathrm{loc}}^\infty(\mathbb{R}^N\backslash\{|y|=0\})$; the conclusion follows. 
\end{proof}

\bibliographystyle{abbrv}

\bibliography{bibtese}

\end{document}